\documentclass[a4paper,11pt]{article}

\usepackage{amsmath}
\usepackage{amssymb}
\usepackage{amsthm}
\usepackage{graphicx}
\usepackage{psfrag}
\usepackage{verbatim}
\usepackage{color}
\usepackage[%
ocgcolorlinks,%
linkcolor=blue,%
filecolor=blue,%
citecolor=blue,%
urlcolor=blue]{hyperref}

\usepackage[font={small,it}]{caption}

\newcommand{\bbL}{\mathbb{L}}
\newcommand{\bbT}{\mathbb{T}}

\newcommand{\bbE}{\mathbb{E}}
\newcommand{\bbV}{\mathbb{V}}

\newcommand{\Var}{\bbV{\rm ar}}
\newcommand{\Cov}{\bbC{\rm ov}}
\newcommand{\bbP}{\mathbb{P}}

\newcommand{\bbN}{\mathbb{N}}

\newcommand{\bbR}{\mathbb{R}}
\newcommand{\bbC}{\mathbb{C}}

\newcommand{\ol}{\overline}

\newcommand{\wh}{\widehat}

\newcommand{\cC}{\mathcal C}

\newcommand{\cH}{\mathcal H}

\newcommand{\cB}{\mathcal B}
\newcommand{\cG}{\mathcal G}

\newcommand{\cW}{\mathcal W}

\newcommand{\cE}{\mathcal E}
\newcommand{\cN}{\mathcal N}
\newcommand{\cF}{\mathcal F}
\newcommand{\cM}{\mathcal M}
\newcommand{\cL}{\mathcal L}
\newcommand{\cR}{\mathcal R}
\newcommand{\cQ}{\mathcal Q}
\newcommand{\cU}{\mathcal U}
\newcommand{\cV}{\mathcal V}
\newcommand{\cD}{\mathcal D}

\newcommand{\frR}{\mathfrak R}

\newcommand{\rmc}{{\rm c}}

\newcommand{\rmd}{{\rm d}}
\newcommand{\rme}{{\rm e}}
\newcommand{\rmB}{{\rm B}}

\newcommand{\rmP}{{\rm P}}
\newcommand{\rmC}{{\rm C}}
\newcommand{\rmA}{{\rm A}}

\newcommand{\bfL}{{\mathbf L}}
\newcommand{\bfT}{{\mathbf T}}
\newcommand{\bfC}{{\mathbf C}}
\newcommand{\bfZ}{{\mathbf Z}}
\newcommand{\bfh}{{\mathbf h}}

\newcommand{\supp}{{\rm supp}}

\newcommand{\lto}{\longrightarrow}
\newcommand{\wto}{\Longrightarrow}
\newcommand{\eqd}{\overset{\rm d}{=}}

\newtheorem{theorem}{Theorem}[section]

\newtheorem{lem}[theorem]{Lemma}
\newtheorem{prop}[theorem]{Proposition}
\newtheorem{thm}[theorem]{Theorem}
\newtheorem*{rem*}{Remark}

\newtheorem{mainthm}{Theorem}
\newtheorem{maincor}[mainthm]{Corollary}

\numberwithin{equation}{section}
\graphicspath{{pic/}}

\title
{A Scaling limit for the Cover Time of the Binary Tree}
\author
{Aser Cortines\thanks{aser.cortinespeixoto@math.uzh.ch, oren.louidor@gmail.com,  saglietti.s@campus.technion.ac.il.} \\ Universit\"at Z\"urich
\and Oren Louidor\footnotemark[1]\\ Technion, Israel \and Santiago Saglietti\footnotemark[1]\\Technion, Israel}
\date{}

\begin{document}
\maketitle

\begin{abstract} 
We consider a continuous time random walk on the rooted binary tree of depth $n$ with all transition rates equal to one and study its cover time, namely the time until all vertices of the tree have been visited. We prove that, normalized by $2^{n+1} n$ and then centered by $(\log 2) n - \log n$, the cover time admits a weak limit as the depth of the tree tends to infinity. The limiting distribution is identified as that of a Gumbel random variable with rate one, shifted randomly by the logarithm of the sum of the limits of the derivative martingales associated with two negatively correlated discrete Gaussian free fields on the infinite version of the tree. The existence of the limit and its overall form were conjectured in the literature. Our approach is quite different from those taken in earlier works on this subject and relies in great part on a comparison with the extremal landscape of the discrete Gaussian free field on the tree.
\end{abstract} 

\setcounter{tocdepth}{2}
\tableofcontents

\section{Introduction and Results}
\label{s:1}
Let $\bbT_n$ be the rooted binary tree of depth $n \geq 1$, of which the root we denote by $0$ and its set of leaves by $\bbL_n$. Suppose that $X_n = (X_{n,t} :\: t \geq 0)$ is a continuous time random walk (RW) on $\bbT_n$ starting from $0$ with all transition rates equal to one. Equivalently, starting from the root, upon reaching a vertex $x \in \bbT_n$ the process $X_n$ first holds for an exponentially distributed time with mean $1/\deg(x)$ and then jumps to one of the neighbors of $x$ in $\bbT_n$, uniformly at random.
Given $x \in \bbT_n$ and $t \geq 0$, the local time at $x$ up to time $t$ is defined as
\begin{equation}
\label{e:1.1}
\bfL_{n,t}(x) := \int_{0}^t 1_x\big(X_{n,s}\big) \rmd s \,.
\end{equation}
The cover time of $\bbT_n$ is the first time when all vertices in $\bbT_n$ have been visited, namely
\begin{equation}
\bfT^{\bfC}_n := \inf \big\{t \geq 0:\: \min_{x \in \bbT_n} \bfL_{n,t}(x) > 0 \big\} \,.
\end{equation}

The goal of the present work is to study the large $n$ asymptotics of $\bfT^\bfC_n$. As in many of the earlier works on this subject, it is easier to first consider the cover time measured in terms of the local time at the root. Formally, we let $\bfL^{-1}_{n,t}(x)$ be the generalized inverse of $t \mapsto \bfL_{n,t}(x)$, namely
\begin{equation}
\label{e:1.2}
\bfL^{-1}_{n,t}(x) := \inf \big\{s \geq 0 :\: \bfL_{n,s}(x) > t \big\} \,.
\end{equation}
Then, the local time at $x \in \bbT_n$ {\em measured in terms of the local time at the root}
is given by
\begin{equation}
L_{n,t}(x) := \bfL_{n,\bfL^{-1}_{n,t}(0)}(x) \,.
\end{equation}
That is, $L_{n,t}(x)$ is the local time at $x$ up to the moment that the local time at the root is $t$.
The corresponding version of the cover time is then defined as
\begin{equation}
T^{\rmC}_n := \inf \big\{t \geq 0:\: \min_{x \in \bbT_n} L_{n,t}(x) > 0 \big\} \,.
\end{equation}

Our first theorem shows that $\sqrt{T^\rmC_n}$ admits a distributional limit as $n \to \infty$ after centering by $\sqrt{t_n^\rmC}$, where $t_n^{\rmC}$ is defined via
\begin{equation}
\label{e:1.3}
\sqrt{t^{\rmC}_n} := \sqrt{\log 2}\, n - \frac{1}{2 \sqrt{\log 2}} \log n \,.
\end{equation}
The limiting law is described in terms of the limit of the {\em derivative martingale} associated with the {\em discrete Gaussian free field} (DGFF) on the infinite rooted binary tree $\bbT$ with zero imposed at the root as boundary conditions. To be more explicit, let us now properly define these objects.

Suppose that each of the edges of the tree $\bbT$ is assigned a random value, chosen independently according to the Gaussian distribution with mean $0$ and variance $1/2$. Then, for each $x \in \bbT$, let $h(x)$ be the sum of the values assigned to the edges on the path going from the root of $\bbT$ to $x$. The field $h = (h(x) :\: x \in \bbT)$ is then (a multiple of) the DGFF on $\bbT$. Alternatively, $h(x)$ can be seen as recording the position of particle $x$ in a branching random walk (BRW) process, with deterministic binary branching and centered Gaussian steps with variance $1/2$.

The derivative martingale associated with $h$ is the process $(Z_n :\: n \geq 0)$ defined as
\begin{equation}
\label{e:1.8}
Z_n := 2^{-2n} \sum_{x \in \bbL_n} \big(\sqrt{\log 2}\, n - h(x)\big) \rme^{2 \sqrt{\log 2}\,  h(x)} \,.
\end{equation}
It is well known~\cite[Proposition A.3]{aidekon2013convergence} that the sequence $(Z_n :\: n \geq 0)$ converges almost-surely to a positive and finite limit, and accordingly we set
\begin{equation}
\label{e:1.9}
Z := \lim_{n \to \infty} Z_n \quad \text{a.s.}
\end{equation}
We can now formulate our first theorem.
\begin{mainthm}
\label{t:1.1}
For all $s \in \bbR$,
\begin{equation}
\label{e:1.6}
\lim_{n \to \infty} \bbP \Big(\sqrt{T^{\rmC}_n} - \sqrt{t_n^{\rmC}} \leq s \Big) 
= \bbE \exp \big(\!-C_\star Z \rme^{-2 \sqrt{\log 2} \, s} \big) \,,
\end{equation}
where $C_\star \in (0,\infty)$ and $Z$ is as in~\eqref{e:1.9}.
\end{mainthm}

In order to deduce the weak convergence of the {\em real} cover time $\bfT^{\bfC}_n$ from the the convergence of $T^C_n$, we prove the following theorem which relates the two.
\begin{mainthm}
\label{t:2.5}
For all $s \in \bbR$,
\begin{equation}
\lim_{n \to \infty} \Big| \bbP \Big(\sqrt{2^{-(n+1)} \bfT^{\bfC}_n} - \sqrt{t_n^{\rmC}} + \xi \leq s \Big) -
\bbP \Big( \sqrt{T^{\rmC}_n} - \sqrt{t_n^{\rmC}} \leq s\Big) \Big| = 0 \,,
\end{equation}
where $\xi$ is a centered Gaussian random variable with variance $1/2$, which is independent of $X_n$.
\end{mainthm}

Using Theorem~\ref{t:1.1} and Theorem~\ref{t:2.5}, deriving a scaling limit for $\bfT^{\bfC}_n$ is not a difficult task. To state the result, however, we first need to define a surrogate for $Z$ in the description of the limiting law. To this end, we define a new centered Gaussian field $\bfh = (\bfh(x) :\: x \in  \bbT^1 \cup \bbT^2)$, where $\bbT^1$, $\bbT^2$ are two disjoint isomorphic copies of $\bbT$. The law of $\bfh$ is completely determined by insisting that 
\begin{equation}
\label{e:301.11}
\big(\bfh(x^1) :\: x^1 \in \bbT^1\big) \eqd \big(\bfh(x^2) :\: x^2 \in \bbT^2\big) \eqd \big(h(x) :\: x \in \bbT\big)\,,
\end{equation}
with $h$ being the DGFF from before and that
\begin{equation}
\label{e:301.12}
\bbE\, \bfh(x^1) \bfh(x^2) = -\tfrac{1}{2} \big( 1-2^{-|x^1| \wedge |x^2|} \big) 
\quad : \qquad x^1 \in \bbT^1 ,\, x^2 \in \bbT^2 \,,
\end{equation}
where $|x^1|$ and $|x^2|$ denote the depths of $x^1$ and $x^2$ in $\bbT^1$ and $\bbT^2$, respectively.
The existence of such a field will be shown in Proposition~\ref{p:2.6} of Section~\ref{s:2}. In the meantime, we observe that 
for all $n \geq 1$ and $x^1 \in \bbL^1_n$, $x^2 \in \bbL^2_n$, 
\begin{equation}
\bbE \bfh(x^1) \bfh(x^2) = -\tfrac12 (1-2^{-n}) \lto -\tfrac12 
\quad \text{as } n \to \infty \,,
\end{equation}
where $\bbL^1_n$ and $\bbL^2_n$ denote the set of leaves at depth $n$ of $\bbT^1$ and $\bbT^2$, respectively.

In light of~\eqref{e:301.11}, we can define $Z^1$ and $Z^2$ to be the almost-sure limits of the derivative martingales with respect to $\big(\bfh(x^1) :\: x^1 \in \bbT^1\big)$ and $\big(\bfh(x^2) :\: x^2 \in \bbT^2\big)$, defined exactly as in~\eqref{e:1.8} and~\eqref{e:1.9} only with $h$ replaced by $\bfh$ restricted to $\bbT^1$ and $\bbT^2$, respectively. We then set
\begin{equation}
\label{e:301.13}
\bfZ := Z^1 + Z^2 \,.
\end{equation}
The weak convergence of the scaled real cover time is then given by,
\begin{mainthm}
\label{t:1.3}
For all $s \in \bbR$,
\begin{equation}
\label{e:201.11}
\lim_{n \to \infty}
\bbP \Big(\sqrt{2^{-(n+1)} \bfT^{\bfC}_n} - \sqrt{t_n^{\rmC}} \leq s \Big)
= \bbE \exp \big(\!-\bfC_\star \bfZ \rme^{-2 \sqrt{\log 2} \, s} \big) \,,
\end{equation}
where $\bfC_\star \in (0,\infty)$ and $\bfZ$ is as in~\eqref{e:301.13} and is positive and finite almost surely. Furthermore, $\bfC_\star$ and $\bfZ$ are related to $C_\star$ and $Z$ from Theorem~\ref{t:1.1} via,
\begin{equation}
\label{e:1.16}
\bfC_\star = C_\star/4
\quad, \qquad
\bfZ \Lambda \eqd Z\,,
\end{equation}
where $\Lambda \sim \text{Log-normal}(-2 \log 2, 2 \log 2)$ is assumed to be independent of $\bfZ$.

\end{mainthm}

\medskip
\noindent Since $\Big|\Big(2^{-(n+1)} n^{-1} \bfT^{\bfC}_n - \big((\log 2) n - \log n\big)\Big) - \Big(2\sqrt{\log 2} \big(\sqrt{2^{-(n+1)}\bfT^{\bfC}_n} - \sqrt{t^\rmC_n}\big)\Big)\Big|$
is 
\begin{equation}
%2\sqrt{\log 2} \Big(\sqrt{2^{-(n+1)}\bfT^{\bfC}} - \sqrt{t^\rmC_n}\Big)  
O \bigg(\Big(\Big(\sqrt{2^{-(n+1)} \bfT^{\bfC}_n} - \sqrt{t^\rmC_n}\Big)^2 + 1 \Big) n^{-1} \log^2 n \bigg) \,,
\end{equation}
as a straightforward corollary of Theorem~\ref{t:1.3} we get,
\begin{maincor}
\label{c:4}
For all $s \in \bbR$,
\begin{equation}
\label{e:401.11}
\lim_{n \to \infty}
\bbP \bigg(\frac{\bfT^{\bfC}_n}{2^{n+1} n} - \big((\log 2) n - \log n\big) \leq s \bigg)
= \bbE \exp \big(\!-\bfC_\star \bfZ \rme^{-s} \big) \,,
\end{equation}
where $\bfC_\star$ and $\bfZ$ are as in Theorem~\ref{t:1.3}.
\end{maincor}

\noindent {\bf Remarks.}
\\ \noindent
1. In many other works on this subject, the mean holding time at a vertex $x$ is assumed to be one, but the local time at $x$ is then scaled by $1/\deg(x)$. We chose instead (purely for aesthetic reasons) to define the mean holding time as $1/\deg(x)$ and not to scale the local time. This makes no difference when one considers $T_n^\rmC$, but results in a difference by an overall multiplicative factor of $2$ in the asymptotics of $\bfT_n^\bfC$ (that is, $\bfT^\bfC_n$ in the case of mean one is asymptotically twice what it is here). This is rather evident from the proof of Theorem~\ref{t:1.3}. 

\noindent
2. The asymptotics for the real cover time in the case of mean one holding times remain the same, if one then replaces the continuous time random walk with its discrete-time analog (in which case $\bfT_n^\rmC$ measures the number of {\em steps} taken before all vertices have been visited). This is because the fluctuations in the number of steps which $X_n$ makes during $t \geq 0$ time is of order $\sqrt{t}$. Consequently, at times of the order of the real cover time $2^n n^2$, these fluctuations are of order $2^{n/2} n$, which is far smaller than $2^{n+1}n$ -- the order of fluctuations of $\bfT^\rmC_n$, as shown by Corollary~\ref{c:4}.

\subsection{Discussion and related works}
The asymptotics of the cover time on the tree have been studied considerably in the past. The leading order term $(\log 2)2^{n+1} n^2$ was found by Aldous in~\cite{aldous1991random}. The second order term $-2^{n+1} n \log n$ was then derived by Ding and Zeitouni in~\cite{ding2012sharp}. (Both expressions are scaled by $2$ for the sake of comparison, as the authors considered the mean one holding time case.) Tightness of the scaled cover time after centering by the median was shown by Bramson and Zeitouni~\cite{bramson2009tightness}, without an explicit expression for its asymptotics. Very recently, Belius, Rosen and Zeitouni~\cite{belius2017barrier} showed that the remaining terms in the asymptotic expansion are $O(1)$ and consequently proved tightness of the scaled cover time. 

What prevented the authors in~\cite{belius2017barrier} from obtaining the convergence in law of $T_n^\rmC$ was the lack of precise asymptotics for the right tail of $\sqrt{T^\rmC_n} - \sqrt{t_n^\rmC}$. The authors do obtain the right order of the tail (\cite[Theorm 1.3]{belius2017barrier}) using careful barrier estimates, but with non-matching constants for the upper and lower bounds. We remark that the approach taken in this work is quite different and, instead of using barrier estimates, relies almost entirely on comparison with the extremal landscape of the DGFF. As such, we avoid the need of deriving asymptotics for the right tail altogether and, as a side effect, essentially reproduce all the results mentioned above using different arguments (see Subsection~\ref{ss:1.2} for more details).

The limiting law in Corollary~\ref{c:4} (as well as those in Theorem~\ref{t:1.1} and Theorem~\ref{t:1.3}) can be interpreted as the distribution of $G + \log \bfZ$, where $G$ is a  Gumbel distributed random variable satisfying $\bbP(G \leq s) = \exp(-{\bfC}_\star \rme^{-s})$ for all $s \in \bbR$, and $\bfZ$ is as in~\eqref{e:301.13} and independent of $G$. As such, the limiting law of the cover time bears strong resemblance to the weak limit of the centered minimum/maximum of the DGFF on $\bbT$. Indeed, it was shown by A\"id\'ekon~\cite{aidekon2013convergence} that for all $s \in \bbR$,
\begin{equation}
\label{e:1.19}
\lim_{n \to \infty} \bbP \Big(\max_{x \in \bbL_n} h(x) - m_n \leq s \Big)
= \bbE \exp \big(C_\diamond Z \rme^{-2 \sqrt{\log 2} s} \big)
\ ; \quad
m_n := \sqrt{\log 2}\, n - \frac{3}{4 \sqrt{\log 2}} \log n \,,
\end{equation}
where $Z$ is precisely as in~\eqref{e:1.9} and $C_\diamond \in (0,\infty)$. 

In view of Theorem~\ref{t:1.1}, we see that when the cover time is measured in terms of the local time at the root, its square root admits {\em exactly} the same limiting law (up to an arbitrary shift) as that of the maximum of the DGFF, although the centering sequence in both cases differ by
$\sqrt{t_n^\rmC} - m_n = (4\sqrt{\log 2})^{-1} \log n$. In particular, from the known asymptotics for the right tail of the limit of the centered maximum (\cite[Proposition 1.3]{aidekon2013convergence}), it readily follows that
\begin{equation}
\lim_{n \to \infty} \bbP \Big(\sqrt{T^{\rmC}_n} - \sqrt{t_n^{\rmC}} > s \Big) \sim C_\circ s \rme^{-2\sqrt{\log 2} s} 
\quad \text{as } s \to \infty \,,
\end{equation}
for some $C_\circ \in (0,\infty)$. 
This was the missing stronger form of~\cite[Theorem 1.3]{belius2017barrier} mentioned above.

The overall form of the limiting law of the cover time as a randomly shifted Gumbel and its resemblance to the limiting law of the centered maximum of the DGFF were expected in the literature (see below). In fact, a randomly shifted Gumbel distribution is conjectured to be the universal limiting law of the maximum/minimum of fields with logarithmic or approximate logarithmic (e.g. hierarchical) correlations. This has been verified, for example, in the case of branching Brownian motion~\cite{bramsonBBM} and the discrete Gaussian free field in two dimensions~\cite{bramson2016convergence}.

Somehow less expected, as far as we know, is the appearance of $\bfZ$ in the definition of the random shift which governs the limiting law of the real cover time. The random variable $\bfZ$, as defined in~\eqref{e:301.13}, is the sum of two copies of the limit $Z$ of the derivative martingale, defined with respect to two DGFFs on $\bbT$ which are coupled together such that the covariances between their values on $\bbL_n$ all tend to $-1/2$ as $n \to \infty$. Alternatively, $\bfZ$ can be viewed as the limit of a derivative-martingale-like process, defined similarly to~\eqref{e:1.8} only with respect to a centered Gaussian field $\bfh$ satisfying~\eqref{e:301.11} and~\eqref{e:301.12} (one can also take~\eqref{e:1.16} for a definition of $\bfZ$). The appearance of negative correlations in the definition of $\bfZ$ is due to the negative dependency between $T^\rmc_n$ and $\bfL^{-1}_{n,t}$ (see Subsection~\ref{ss:1.2}). 

The connection between the cover time of the tree and the maximum/minimum of the DGFF is well known. A general theory relating the two on general graphs was developed by Ding, Lee and Peres  in~\cite{ding2012cover} and extended by Ding~\cite{ding2014asymptotics}. Further evidence for the connection between the extreme values of the local time field and those of the DGFF was demonstrated by Abe in~\cite{abe2018extremes}, who showed that the thinned extremal process of local times converges in law to a Cox process, driven by the limit of the derivative martingale  (in its random-measure form) of the DGFF. 

The reason why these two seemingly very different objects exhibit very similar extreme value statistics becomes apparent thanks to the second Ray-Knight Theorem (also sometimes referred to as a version of Dynkin's Isomorphism Theorem). The theorem, which in this context is due to Eisenbaum, Kaspi, Marcus and Rosen~\cite{eisenbaum2000ray} relates the law of the local time field to the law of the DGFF on the same graph. Specializing to the tree, this relation can be expressed in terms of a coupling between $L_{n,t}$ and two copies $h$, $h'$ of the DGFF, under which $L_{n,t}$ and $h$ are independent of each other and the following identity holds almost-surely:
\begin{equation}
\label{e:3.1a}
L_{n,t}(x) + h^2(x) = 
\big(h'(x) + \sqrt{t})^2 
\quad : \ x \in \bbT_n \,.
\end{equation}

To illustrate that the above relation readily yields a strong connection between the extreme values of the local time field and those of the DGFF, consider the case when $\sqrt{t}$ is much larger than $m_n$ - the order of extreme values of $h$ and $h'$ on $\bbL_n$. In this case, under the coupling we have  $\sqrt{L_{n,t}(x)} - \sqrt{t} \approx  
h'(x)$, which shows that extreme values of $\sqrt{L_{n,t}} - \sqrt{t}$ on $\bbL_n$ are approximately the same as those of the DGFF on $\bbL_n$. Unfortunately, in the study of the cover time $\sqrt{t}$ is taken to be $\sqrt{t_n^\rmC} + O(1) \sim m_n$ and thus the use of the isomorphism is far less trivial. 

Lastly, let us mention that in a parallel development thread, considerable effort has been devoted to studying the cover time of a random walk on the two-dimensional torus or the cover time of a Brownian motion (BM) on a two-dimensional compact manifold (see, e.g.~\cite{dembo2004cover} for the definition in this case). Results here include, in chronological order, the derivation of the leading order term for RW and BM~\cite{dembo2004cover}, a bound on the second order term for RW~\cite{ding2012cover2D}, a derivation of the second order term for BM~\cite{belius2017subleading}, a derivation of the second order term for RW~\cite{abe2017second} and most recently tightness in the case of BM~\cite{belius2017}. Despite the very different natures of the underlying graphs, and as in the case of the DGFF, many of the results and techniques in the case of the tree, carry over to the two-dimensional setup, albeit with notable and difficult technical complications involved.

\subsection{Proof outline}
\label{ss:1.2}
In this subsection we give a short overview of the proofs of Theorem~\ref{t:1.1} and Theorem~\ref{t:2.5}. In the case of Theorem~\ref{t:1.3}, the argument is rather straightforward and can be easily understood from the top-level proof of the theorem in Section~\ref{s:2}.

\subsubsection{Theorem~\ref{t:1.1}}
The proof of Theorem~\ref{t:1.1} takes up the largest part of the paper. The first key idea is to split the running time of the walk into two consecutive phases: $A$ and $B$. In phase $A$ the random walk is run for time $t^{\rmA}_n$ and in phase $B$ for time $t^{\rmB}_n+sn$, where $s \in \bbR$ and 
\begin{equation}
\label{e:302.4}
\sqrt{t_n^{\rm A}} : = m_n = \sqrt{\log 2}\, n - \frac{3}{4 \sqrt{\log 2}} \log n 
\quad, \qquad
t_n^{\rmB} := \frac{1}{2} n\log n \,,
\end{equation}
Above $m_n$ is as in~\eqref{e:1.19} and both running times are measured in terms of the local time at the root.
Observe that for fixed $s$ as $n \to \infty$,
\begin{equation}
\label{e:2.5a}
\sqrt{t_n^{\rmA} + t_n^{\rmB} + sn} = \sqrt{t_n^{\rmC}} + (2\sqrt{\log 2})^{-1} s + o(1) \,.
\end{equation}

The motivation for this split comes from the isomorphism with the DGFF and the relation
\begin{equation}
\label{e:3.1b}
L_{n,t}(x) + h^2(x) = 
\big(h'(x) + \sqrt{t})^2 
\quad : \ x \in \bbT_n \,,
\end{equation}
under the coupling employed in~\eqref{e:3.1a}.
Since the min-extreme (near minima) values of the DGFF on $\bbL_n$ are typically at a height $-m_n + O(1)$, at time $\sqrt{t}=\sqrt{t_n^{\rmA}}$ the right hand side in~\eqref{e:3.1b} for such {\em min-extreme leaves} will be $O(1)$. On the other hand, to get an $O(1)$ value on the left hand side in~\eqref{e:3.1b}, we must have both $L_{n,t}(x) = O(1)$ and $h_n(x) = O(1)$. It follows that, under the isomorphism, min-extreme leaves of $h'$ on $\bbL_n$ correspond to {\em low-local-time leaves} (namely, leaves with $O(1)$ local time under $L_{n,t}$), which {\em``survived the isomorphism''} because their corresponding value under $h$ is also $O(1)$. Since $h(x)$ for $x \in \bbL_n$ has a Gaussian law with mean $0$ and variance $n/2$,  a low-local-time leaf survives with probability of order $1/\sqrt{n}$. 

Now, as it turns out, at time $t_n^{\rmA}$, except for a negligible subset, with high probability most of the low-local-time leaves agglomerate in an order of $\sqrt{n}$ clusters, each having diameter $O(1)$ and at distance $n-o(n)$ apart (both measured in graph distance). Moreover, the restrictions of the local time field to each of these {\em low-local-time clusters} follow jointly an asymptotic i.i.d. law. Since the law of the DGFF $h$ restricted to such a clustered collection of leaves in $\bbL_n$ follows a similar i.i.d. structure, the number of low-local-time clusters in which at least one leaf survives has asymptotically a Binomial distribution with order $\sqrt{n}$ trials and success probability of order $1/\sqrt{n}$.

It follows from the Poisson approximation to the Binomial distribution that the number of ``surviving'' clusters (namely, low-local-time clusters in which at least one leaf survives) is asymptotically Poisson with rate proportional to $1/\sqrt{n}$ times the total (random) number of low-local-time clusters. Since surviving clusters correspond exactly to min-extreme clusters of $h'$ (namely, clusters of leaves whose value under $h'$ is $-m_n + O(1)$), we can equate their respective laws. Thanks to earlier work on the extreme values of the DGFF, the asymptotic joint law of the min-extreme values of $h'$ is well known. In the limit, these values have the same distribution as that of the atoms of a clustered Poisson point process with  (random) intensity given by $Z \rme^{2\sqrt{\log 2}u} \rmd u$, where $Z$ is as in~\eqref{e:1.9}. In particular, the number of min-extreme clusters has a Poissonian law with (random) rate proportional to $Z$.

Equating the law of surviving clusters with the law of the min-extreme clusters, we find that the number of low-local-time clusters at the end of phase $A$ has asymptotically the same law as $C' \sqrt{n} Z$ for some $C > 0$ (which explicitly depends on all $O(1)$ terms above). Using the i.i.d. nature of the clusters, the same asymptotics in law can be shown to hold (with the constant $C_\star$ from Theorem~\ref{t:1.1} replacing $C$) for the number of {\em zero-local-time clusters}, namely clusters with leaves which were not visited at all by time $t_n^\rmA$.

Turning to phase $B$, when the random walk leaves the root for an excursion in the direction of $x \in \bbL_n$, a simple {\em Gambler's Ruin} estimate shows that the probability that $x$ will be visited before the walk returns to the root is $1/n$. It is not difficult to show that, whenever the random walk reaches $x$, it will also visit a cluster of diameter $O(1)$ around it in the same excursion, with overwhelming probability. Since the number of excursions in the direction of $x$ within time $t_n^\rmB + sn$ is essentially $t_n^\rmB + sn$, the probability that a cluster of diameter $O(1)$ will not be visited within such time is $(1-1/n)^{t_n^\rmB + sn} = \rme^{-s}/\sqrt{n}$. We now see that estimating the number of zero-local-time clusters after phase $A$ which are then not entirely visited in phase $B$ becomes a {\em Coupon Collector Problem}. 

Because the clusters are $n-o(n)$ apart, the ``not-entirely-visited'' events are essentially independent for different clusters. Since asymptotically there are $C_\star \sqrt{n} Z$ clusters after phase $A$, each not entirely visited in phase $B$ with probability $\rme^{-s}/\sqrt{n}$, we are again in the Poisson regime of the Binomial distribution. It follows that the number of zero-local-time clusters not entirely visited after phase $B$ is again Poisson in the limit with (random) rate $C_\star \rme^{-s} Z$. Observing that leaves which are not visited after both phases  are precisely those contained in such clusters, the probability that the tree $\bbT_n$ is not entirely visited within time $t_n^{\rmA} + t_n^{\rmB} + sn$ is the same as the probability that a Poisson with rate $C_\star \rme^{-s} Z$ is equal to zero. In view of~\eqref{e:2.5a}, this is precisely the statement of Theorem~\ref{t:1.1}.

Evidently, a crucial ingredient in the argument presented above is the sharp description of the clustering structure of the set of low-local-time leaves after phase $A$. One way of obtaining such a description is via truncated first moment bounds, by imposing a barrier on the local time trajectory of such leaves. This {\em barrier method} has been used, e.g. in~\cite{belius2017barrier}, to derive the tightness of the cover time. Our approach, which can be viewed as another key idea in the argument, is rather different. We instead rely almost entirely on comparison with the DGFF via~\eqref{e:3.1b}. 

More precisely, instead of just comparing their values, we compare the full trajectory of low-local-time leaves under $L_{n,t_n^\rmA}$ with the full trajectory of min-extreme leaves under $h'$. In a rather delicate analysis, we are then able to deduce from the known repulsion of trajectories of min-extreme leaves under $h'$ a similar repulsion for the local time trajectories of the low-local-time leaves. More explicitly, if $L_{n,t^\rmA_n}(x) = O(1)$ for $x \in \bbL_n$, then we show that typically
\begin{equation}
\sqrt{L_{n,t^\rmA_n}([x]_k)} - \frac{n-k}{n} \sqrt{t^\rmA_n} > \big(k \wedge (n-k)\big)^{1/2-\epsilon}
\ : \quad k \in [n^{1/2-\epsilon'}, n] \,, 
\end{equation}
where $[x]_k$ denotes the ancestor of $x$ at depth $k$ and $\epsilon, \epsilon' > 0$ can be chosen arbitrarily small. This sharp repulsion statement (which would have been also a necessary step had we employed the barrier method) in turn yields the needed sharp clustering description.  

We remark that not all low-local-time leaves obey the above clustering structure. Indeed, there are low-local-time leaves whose distance from other such leaves is $o(n)$ but not $O(1)$. Unfortunately, in order to show that they form a negligible set (which will be visited entirely in phase $B$ with high probability, and therefore can be ignored), we were forced to rely on the usual barrier method, as we did not find a suitable comparison argument with the DGFF. Luckily, a coarse barrier is sufficient for this purpose, and consequently the application of this method is considerably simpler when compared to, say, its use in~\cite{belius2017barrier}. See the proof of the second part of Theorem~\ref{t:2a} and Subsection~\ref{ss:A2} of the appendix.

\subsubsection{Theorem~\ref{t:2.5}}
Turning to the proof of Theorem~\ref{t:2.5}, it requires a simple computation to show that the total running time $R_{n,t}$ of the walk on $\bbT_n$ when the local time at the root is $t \geq 0$, has mean $(2^{n+1} - 1)t$ and variance $2^{2n+3}t(1+o(1))$. In particular, the fluctuations of this quantity are uniformly of the same order as its mean. This implies that tightness of $\sqrt{\bfT^{\bfC}_n/2^{n+1}}$ around $\sqrt{t_n^\rmC}$ will follow from the tightness of  $\sqrt{T^\rmC_n}$ around the same centering value. On the other hand, it also shows that there will be an additional randomness in the conversion between the two cover times which will not disappear in the limit. Moreover, there is no reason why $(R_{n,t} - \bbE R_{n,t})/\sqrt{\Var R_{n,t}}$ at $t = T^\rmC_n$ should be independent of $T^\rmC_n$. In fact, it is not difficult to guess that they should be negatively correlated.

To control this additional randomness, we first reduce the problem to that of a finite order, regardless of $n$, by considering the tree up to level $k \leq n$. To this end, we define the {\em normalized running time} by $\wh{R}_{n,t} := 2^{-n} R_{n,t}$. Then, conditioning on $(L_{n,t}(x) :\: x \in \bbT_k)$ and neglecting the local time at the first $k-1$ levels of the tree, the distribution of $\wh{R}_{n,t}$ is the same as that of
\begin{equation}
\label{e:1.24}
2^{-k} \sum_{x \in \bbL_k} 2^{-(n-k)} R^{(x)}_{n-k,L_{n,t}(x)} \,,
\end{equation}
where $(R^{(x)}_{n-k,t} :\: t \geq 0)$ are independent for different $x$-s and have the same law as $(R_{n-k, t} :\: t \geq 0)$. This is a simple consequence of the spatial Markov property of $L_{n,t}$.
Furthermore, setting also $S_{k,t} := \sum_{x \in \bbL_k} L_{n,t}(x)$ and $\wh{S}_{k,t} := 2^{-k} S_{k,t}$, under the conditioning above, the quantity in~\eqref{e:1.24} has mean $2\wh{S}_{k,t}(1+o(1))$ and variance $O(2^{-k} \wh{S}_{k,t})$. Therefore $2\wh{S}_{k,t}$ is a good approximation for $\wh{R}_{n,t}$ with high probability as $k \to \infty$.

This implies that, instead of running $X_n$ until (real) time $2^{n+1} s$, we can consider the walk up to the stopping time $\tau_{k, 2^{k+1}s}$, where $\tau_{k, s} := \inf \{t \geq 0 :\: 2 S_{k,t} > s\}$. Using a version of the central limit theorem for a sum with a random number of terms, it can be shown that the joint law of $\big(L_{n,\tau_{k,2^{k+1} s}}(x) :\: x \in \bbL_k\big)$ is then, asymptotically as $s \to \infty$, Gaussian with means $s$ and covariances 
\begin{equation}
\label{e:1.25}
\big(2s (|x\wedge y| - 1) + o(1)\big)_{x,y \in \bbL_k} \,,
\end{equation}
with the $o(1)$ term tending to $0$ as $k \to \infty$. 

On the other hand, in view of the isomorphism relation~\eqref{e:3.1b}, for fixed $k$ and $s \to \infty$, with high probability we have that
\begin{equation}
L_{n,s}(x) = s + 2\sqrt{s} h'(x) + (h')^2(x) - h^2(x) =  s + 2\sqrt{s} \big(h'(x) + o(1)\big)  
\quad :\: x \in \bbL_k \,.
\end{equation}
This shows that $\big(L_{n,s}(x) :\: x \in \bbL_k\big)$ is, asymptotically as $s \to \infty$, also  Gaussian with means $s$, but with covariances given by 
\begin{equation}
\label{e:1.26}
\big(4s\,\bbE h'(x) h'(y)\big)_{x,y \in \bbL_k} = \big(2s |x \wedge y|\big)_{x,y \in \bbL_k} \,.
\end{equation}
Comparing~\eqref{e:1.25} with~\eqref{e:1.26}, we see that on $\bbL_k$ the local time field up to real time $2^{n+1} s$ has asymptotically (as $s \to \infty$ followed by $k \to \infty)$ the same law as that of the local time field when the local time at the root is $s$, {\em up to a negative $-2s$ term added to all covariances}. This negative term is an artifact of the negative correlation mentioned above.

To compensate for this negative covariance term, we consider instead of the stopping time $\tau_{k, 2^{k+1}s}$ the stopping time $\nu_{k,s} := \tau_{k, 
2^{k+1}\theta_s}$ where $\theta_s := s + 2\sqrt{s} \xi$ and $\xi \sim \cN(0,1/2)$ drawn independently of the walk. This essentially amounts to running the random walk until real time $2^{n+1} (s + 2\sqrt{s} \xi)$, or equivalently until the square root of the running time scaled by $2^{-(n+1)}$ is $\sqrt{s} + \xi$. It can then be easily shown that the law of $\big(L_{n,\nu_{k,s}}(x) :\: x \in \bbL_k\big)$ is asymptotically the same as that of $\big(L_{n,s}(x) :\: x \in \bbL_k\big)$. Thanks to the spatial Markov property again, the last two assertions now imply the asymptotic equivalence in law of $\sqrt{T_n^\rmC}$ and $\sqrt{2^{-(n+1)} \bfT^\bfC_n} + \xi$ around the centering value $\sqrt{t_n^\rmC}$, which is precisely the statement in Theorem~\ref{t:2.5}.

\subsection{Organization of the paper}
Section~\ref{s:2} includes the top-level proofs of all of our main theorems. These proofs capture the outline of the arguments leading to each of the main theorems and as such rely on key statements which are proved later in the manuscript. All the necessary preliminaries are stated in Section~\ref{s:3} with all lengthy but standard proofs deferred to Appendix~\ref{s:A}. Section~\ref{s:4} includes upper bounds which are direct consequences of the isomorphism. They are used in Section~\ref{s:5} to obtain the sharp clustering description as discussed in Subsection~\ref{ss:1.2}. Section~\ref{s:6} deals with the i.i.d. nature of low-local-time clusters, as well as the i.i.d. nature of the DGFF on clustered sets. Section~\ref{s:7} includes the proof of the main theorem for phase $A$. Section~\ref{s:8} includes the proof of the main theorem for phase $B$. Both are building blocks in the proof of Theorem~\ref{t:1.1}. Section~\ref{s:9} and Section~\ref{s:10} include the main results needed for the proofs of Theorem~\ref{t:2.5} and Theorem~\ref{t:1.3}, respectively. Finally, the appendix includes the proofs which were left out from Section~\ref{s:3}.

\section{Top-level proofs}
\label{s:2}
In this section we provide top-level proofs for Theorems~\ref{t:1.1},~\ref{t:2.5} and~\ref{t:1.3}. 
These proofs build on key statements whose proofs, in turn, constitute the main effort in this work. To obtain a top-down exposition, we shall present these key statements here and first show how to use them in order to prove the main results of the paper. Then, in the remainder of the paper, we will provide the necessary proofs of these key results.

\subsection{General notation}
We start by introducing some further notation which will be used throughout the sequel. First and foremost, we will often consider the tree $\ol{\bbT}_n$ in place of $\bbT_n$. The former is defined exactly as $\bbT_n$, except that the root $0$ has only one child, so that its degree is now one instead of two. Observe that the size of the leaf set $\ol{\bbL}_n$ of $\ol{\bbT}_n$ is $2^{n-1}$ instead of $2^n$ (as it is for $\bbL_n$). Furthermore, we will think of $\ol{\bbT}_n$ for all $n \geq 1$ as embedded in one single infinite rooted tree $\ol{\bbT}$ where all vertices have degree $3$, except for the root which has degree one. The slight advantage of working with this unconventional tree is that for any $x \in \ol{\bbT}$, each of the subtrees $\ol{\bbT}^r(x)$, $\ol{\bbT}^l(x)$ -- consisting of $x$, one of its children and all of their descendants -- is isomorphic to $\ol{\bbT}$. This will be somewhat handy in the sequel. We note that the union $\bbT(x) := \ol{\bbT}^l(x) \cup \ol{\bbT}^r(x)$ forms an infinite regular rooted binary tree which is isomorphic to $\bbT$.

Given $x \in \ol{\bbT}$, we write $|x|$ for the {\em depth} or {\em generation} of $x$, namely the (graph) distance between $x$ and $0$. For $0 \leq k \leq |x|$, we write $[x]_k$ for the ancestor of $x$ at depth $k$. The deepest common ancestor of $x,y \in \ol{\bbT}$ will be denoted by $x \wedge y$. These last two definitions extend to an arbitrary collection of vertices $\cV \subset \ol{\bbT}$ via $\bigwedge \cV$ for the deepest common ancestor of all $x \in \cV$ and $[\cV]_k$ for the set $\{[x]_k :\: x \in \cV\}$. All of the above notation extends to the regular binary tree $\bbT$ with the obvious necessary changes. Furthermore, we will also use this notation for subtrees of $\ol{\bbT}$ or $\bbT$. For instance, if $x \in \ol{\bbT}$ then $\bbL_k(x)$ is the set of vertices in $\bbT(x)$ at distance $k$ from $x$, or equivalently the set of leaves in $\bbT_k(x)$. 

When the underlying tree is $\ol{\bbT}_n$, the process $X_n = \{X_{n,t} :\: t \geq 0\}$ still denotes a continuous time random walk with all transition rates equal to one, and the local times $\bfL_{n,t}(x)$ and $L_{n,t}(x)$ are defined exactly as in the case of $\bbT_n$. The precise underlying graph will always be made explicit, so that there should be no confusion. For both graphs and any $1 \leq k \leq n, n'$, since the law of $(L_{n,t}(x) :\: |x| \leq k \,,\, t \geq 0)$ is the same as that of $(L_{n',t}(x) :\: |x| \leq k \,,\, t \geq 0)$, we shall often omit $n$ from the notation and just write $L_t(x)$.

Finally, for some general conventions, if $f$ is a function on some domain $D$ then we denote by $f(D)$ the collection of all of its values on $D$. We denote the cardinality of a set $D$ by $|D|$, not to be confused with the depth $|x|$ of the vertex $x$.
As usual, constants whose value is immaterial and may change from one use to another are denoted by $C, C', C''$, etc. These are always positive and finite and are independent of any parameter, unless stated explicitly otherwise.

\subsection{Proof of Theorem~\ref{t:1.1}}
In order to prove Theorem~\ref{t:1.1}, we first consider the tree $\ol{\bbT}_n$ in place of $\bbT_n$ and, as explained in Subsection~\ref{ss:1.2}, split the running time of the walk on $\ol{\bbT}_n$ into two consecutive phases: $A$ and $B$. In phase $A$ the random walk is run for time $t^{\rmA}_n$ and in phase $B$ for time $t^{\rmB}_n+sn$, where $s \in \bbR$ and 
\begin{equation}
\label{e:2.4}
\sqrt{t_n^{\rm A}} := \sqrt{\log 2}\, n - \frac{3}{4 \sqrt{\log 2}} \log n 
\quad, \qquad
t_n^{\rmB} := \frac{1}{2} n\log n \,.
\end{equation}
Both running times are measured in terms of the local time at the root. As before, we observe that for fixed $s$ as $n \to \infty$,
\begin{equation}
\label{e:2.5}
\sqrt{t_n^{\rmA} + t_n^{\rmB} + sn} = \sqrt{t_n^{\rmC}} + (2\sqrt{\log 2})^{-1} s + o(1) \,.
\end{equation}

We will show that after phase $A$, aside from a negligible subset of size $o(\sqrt{n})$, the collection of non-visited leaves agglomerate in clusters whose diameter (in graph distance) is $o(n^{1/2})$ (in fact, as the proof shows, most of these clusters are essentially of diameter $O(1)$). Moreover, scaled by $1/\sqrt{n}$, the number of such clusters converge weakly to a constant multiple of $\bar{Z}$, defined as $Z$ in~\eqref{e:1.9} only with respect to the tree $\ol{\bbT}$, in place of $\bbT$. To handle phase $B$, we will then show that for such a clustered collection of leaves, the number of clusters not entirely visited by the random walk after $t_n^{\rmB}+s n$ time is asymptotically Poisson with rate proportional to $\rme^{-s}/\sqrt{n}$ times the original number of clusters in the set. This will show that, after phase $B$, the number of clusters of non-visited leaves is asymptotically Poisson with rate $C\bar{Z}\rme^{-s}$, from which the convergence in law of the cover time $T_n^{\rmC}$ will follow immediately.

To make this argument precise, for $n \geq 1$ and $t \geq 0$, we start by thinking of $L_t = L_t(\ol{\bbL}_n)$ as a (random) function on $\ol{\bbL}_n$, and for $u \geq 0$ define the sub-level set of $L_t$ at height $u$ as
\begin{equation}
\cF_{n,t}(u) := \big\{ x \in \ol{\bbL}_n :\: L_t(x) \leq u \big\} \,.
\end{equation}
For $\eta \in (0,1/2)$, whose precise value is immaterial for the argument, we also henceforth set
\begin{equation}\label{e:2.3}
r_n = r_n^\eta:= \lfloor n^{1/2-\eta}\rfloor \,.
\end{equation}
Then, for any $y \in [\cF_{n,t}(u)]_{r_n}$, we shall call the set $\cF_{n,t}(u) \cap \bbT(y)$ an {\em $r_n$-cluster} of $\cF_{n,t}(u)$. 

For each $k \in [r_n, n]$, the subset of leaves in $\cF_{n,t}(u)$ belonging to {\em $k$-rooted} $r_n$-clusters is defined as
\begin{equation}
\label{e:2.4a}
\cW^k_{n,t}(u) := \Big\{ x \in \cF_{n,t}(u) :\: \Big|\bigwedge \Big( \cF_{n,t}(u) \cap \bbT([x]_{r_n})\Big)\Big| = k \Big\} \,.
\end{equation}
This is the set of all leaves that belong to some $r_n$-cluster of $\cF_{n,t}(u)$ of which the deepest common ancestor is found at depth $k$. We will call $z \in [\cW_{n,t}^k(u)]_k$ the {\em root} of the cluster $\cW_{n,t}^k(u) \cap \bbT(z)$. Notice that any $k$-rooted $r_n$-cluster of $\cW^k_{n,t}(u)$ has diameter $2(n-k)$ (in graph distance). Let us also set, for $K \subseteq [r_n,n]$,
\begin{equation}
\label{e:2.4b}
\cW^{K}_{n,t}(u) := \bigcup_{k \in K} \cW^{k}_{n,t}(u)  \,,
\end{equation}
and observe that this is always a union of disjoint sets and is equal to $\cF_{n,t}(u)$ when $K=[r_n, n]$.

In analog to $Z$ from~\eqref{e:1.9}, we define $\bar{Z}$ as the almost-sure limit of the derivative martingale associated with the DGFF on $\ol{\bbT}$:
\begin{equation}
\label{e:1.9a}
\bar{Z} := \lim_{n \to \infty} \bar{Z}_n \  \text{a.s.}
\quad, \qquad
\bar{Z}_n := 2^{-2n+1} \sum_{x \in \ol{\bbL}_n} \big(\sqrt{\log 2}\, n - h(x)\big) \rme^{2 \sqrt{\log 2}\,  h_n(x)} \,,
\end{equation}
where $(h(x) :\: x \in \ol{\bbT})$ is defined exactly as in~\eqref{e:1.8} above, except that now the underlying tree is $\ol{\bbT}$ instead of $\bbT$.
Since $(2Z_n : n \geq 0)$ has the same law as the sum of two independent copies of $(\bar{Z}_n : n \geq 0)$, the existence, positivity and finiteness almost-surely of $\bar{Z}$ follows from the validity of this claim for $Z$ and, moreover, 
\begin{equation}
\label{e:2.8}
2Z \eqd \bar{Z}^l + \bar{Z}^r \,,
\end{equation}
where $\bar{Z}^l$, $\bar{Z}^r$ are two independent copies of $\bar{Z}$.

We can now state the key theorem for phase A, which includes both a sharp description of the clustered structure of non-visited leaves and asymptotics for the number of their $r_n$-clusters.
\begin{thm}[Phase A]
\label{t:2a}
There exists $C_A \in (0,\infty)$ such that as $n \to \infty$,
\begin{equation}
\tfrac{1}{\sqrt{n}}\big|\big[\cF_{n, t_n^{\rmA}}(0)\big]_{r_n}\big| \Longrightarrow 
C_A \bar{Z}\,,
\end{equation}
where $\bar{Z}$ is as in~\eqref{e:1.9a}. In addition, the following holds in probability as $n \to \infty$:
\begin{equation}\label{e:2.9b}
\tfrac{1}{\sqrt{n}} \Big|\cW_{n, t_n^{\rmA}}^{[r_n, n-r_n]}(0)\Big| \lto 0 \,.
\end{equation}
\end{thm}

Turning to phase $B$, for any $0 \leq r \leq R \leq n$ we will say that a set of leaves $\cL_n \subset \ol{\bbL}_n$ is $(r, R)$-{\em clustered} if for all $x, y \in \cL$, we have either $|x \wedge y| \geq R$ or $|x \wedge y| < r$. The next theorem shows that the number of $r_n$-clusters in an $(r_n, n-r_n)$-clustered subset $\cL_n$ of $\ol{\bbL}_n$ which are not entirely visited by the walk after time $t_n^{\rmB}+sn$ is asymptotically Poisson with rate $C \rme^{-s} |\cL_n|/\sqrt{n}$. Moreover, any subset $\cL_n$ of $\ol{\bbL}_n$ which is of size $o(\sqrt{n})$ will be entirely visited with high probability.

\begin{thm}[Phase B]
\label{t:2b}
There exists $C_B \in (0,\infty)$ such that for all $s \in \bbR$, $\lambda \geq 0$, 
\begin{equation}
\label{e:2.9}
\lim_{n \to \infty} \sup_{\cL_n}
\Big| \bbE \exp \Big(-\lambda \Big|\big[\cF_{n,t_n^{\rmB}+sn}(0) \cap \cL_n\big]_{r_n}\Big| \Big) 
- \exp \Big( -C_B \tfrac{\rme^{-s}}{\sqrt{n}} \big|[\cL_n]_{r_n}\big| \big(1-\rme^{-\lambda}\big) \Big) \Big| = 0 \,,
\end{equation}
where the supremum is over all $(r_n, n-r_n)$-clustered subsets $\cL_n$ of $\ol{\bbL}_n$. 
Moreover, for all $n \geq 1$, $s \in \bbR$, any set $\cL_n \subseteq \ol{\bbL}_n$ (not necessarily clustered) and $v > 0$,
\begin{equation}
\label{e:2.10}
\bbP \Big( \big|\cF_{n,t_n^{\rmB}+sn}(0) \cap \cL_n\big| > C_B \tfrac{v \rme^{-s}}{\sqrt{n}}|\cL_n| \Big) \leq v^{-1} \,.
\end{equation}
\end{thm}

Combining Theorem~\ref{t:2a} and Theorem~\ref{t:2b}, the proof of Theorem~\ref{t:1.1} is not difficult.

\begin{proof}[Proof of Theorem~\ref{t:1.1}]
Run the random walk for time $t_n^{\rmA}$ first. Since $\cF_{n,t^{\rmA}_n}(0)$ is the disjoint union of $\cW_{n,t^{\rmA}_n}^{[r_n, n-r_n)}(0)$ and $\cW_{n,t^{\rmA}_n}^{[n-r_n, n]}(0)$, both parts of Theorem~\ref{t:2a} together yield 
\begin{equation}
\tfrac{1}{\sqrt{n}} \big|[\cW_{n,t_n^{\rmA}}^{[n-r_n, n]}(0)]_{r_n}\big| \Longrightarrow C_A \bar{Z} \quad \text{ as } n \to \infty \,,
\end{equation}
where $\bar{Z}$ is as in~\eqref{e:1.9a}. Now run the random walk for additional $t_n^{\rmB} + sn$ time and denote the set of unvisited leaves during this time by $\cF^{\rmB}_{n,t_n^{\rmB}+sn}(0)$. Since $\cW_{n,t_n^{\rmA}}^{[n-r_n, n]}(0)$ is $(r_n, n-r_n)$-clustered by definition, we can apply the first part of Theorem~\ref{t:2b} to conclude that for all $\lambda \geq 0$,
\begin{equation}
\label{e:2.12}
\lim_{n \to \infty} \bbE \exp \Big(-\lambda \Big| \big[\cF^{\rmB}_{n,t_n^{\rmB} +sn}(0)
\cap \cW_{n,t_n^{\rmA}}^{[n-r_n, n]}(0)\big]_{r_n}\Big| \Big) = 
\exp \Big(\! -C_A C_B \rme^{-s} \bar{Z} \big(1-\rme^{-\lambda}\big) \Big) \,.
\end{equation}
Above we first condition on $\cW_{n,t_n^{\rmA}}^{[n-r_n, n]}(0)$ and then take expectation, using the uniformity of the limit in the statement of Theorem~\ref{t:2b}.

Writing,
\begin{equation}
\begin{split}
\cF_{n, t_n^{\rmA}+t_n^{\rmB}+sn}(0) & = \cF_{n, t_n^{\rmA}}(0) \cap \cF^{\rmB}_{n, t_n^{\rmB} + sn}(0) \\
& = \Big(\cW_{n,t_n^{\rmA}}^{[n-r_n, n]}(0) \cap \cF^{\rmB}_{n,t_n^{\rmB} +sn}(0) \Big) \cup \Big(\cW_{n,t_n^{\rmA}}^{[r_n, n-r_n)}(0) \cap \cF^{\rmB}_{n,t_n^{\rmB} +sn}(0) \Big)\,,
\end{split}
\end{equation}
the probability that the second set in the union is not empty is bounded above for any $\delta > 0$ by
\begin{equation}
\bbP \Big(\big|\cW_{n,t_n^{\rmA}}^{[r_n, n-r_n)}(0)\big| > \delta \sqrt{n} \Big)
+
\bbP \Big( \big|\cF_{n,t_n^{\rmB}+sn}^{\rmB}(0) \cap \cW_{n,t_n^{\rmA}}^{[r_n, n-r_n)}(0)\big| > 
C_B \tfrac{v(\delta) \rme^{-s}}{\sqrt{n}}\big|\cW_{n,t_n^{\rmA}}^{[r_n, n-r_n)}(0)\big| \Big) \,,
\end{equation} where $v(\delta) := C_B^{-1} \rme^s \delta^{-1} /2$.
Thanks to the second parts of Theorem~\ref{t:2a} and Theorem~\ref{t:2b}, both probabilities above tend to $0$ when $n \to \infty$ followed by $\delta \to 0$. This shows that 
\begin{equation}
\big|\big[\cF_{n, t_n^{\rmA}+t_n^{\rmB}+sn}(0)\big]_{r_n}\big| - 
\big|\big[\cF^{\rmB}_{n,t_n^{\rmB} +sn}(0) \cap \cW_{n,t_n^{\rmA}}^{[n-r_n, n]}(0)\big]_{r_n} \big| \lto 0 \,,
\end{equation}
as $n \to \infty$ in probability. Combined with~\eqref{e:2.12}, this gives
\begin{equation}
\label{e:2.17a}
\big|\big[\cF_{n, t_n^{\rmA}+t_n^{\rmB}+sn}(0)\big]_{r_n}\big| \wto {\rm Poisson}\big(C_A C_B \rme^{-s} \bar{Z}\big) \,,
\end{equation}
where the right hand side of~\eqref{e:2.17a} is the law of a random variable defined, conditionally on $\bar{Z}$, to have a Poissonian law with rate $C_A C_B \rme^{-s} \bar{Z}$.

Finally, observe that for $0 \in \bbT_n$ the local time fields $L_t(\ol{\bbT}_n^l(0))$ and $L_t(\ol{\bbT}_n^r(0))$ are independent, and each has the same law as $L_t(\ol{\bbT}_n)$. In view of~\eqref{e:2.5}, this implies that with $s' = (2 \sqrt{\log 2})^{-1} s$ as $n \to \infty$,
\begin{equation}
\bbP \Big(\sqrt{T_n^{\rmC}} - \sqrt{t_n^{\rmC}} \leq s' + o(1)\Big)
= \bbP \Big(\big|\big[\cF_{n, t_n^{\rmA}+t_n^{\rmB}+sn}(0)\big]_{r_n}\big| = 0 \Big)^2
\lto \bbE \exp \Big(\!-C_A C_B \rme^{-s} \big(\bar{Z}^l + \bar{Z}^r\big)\Big) \,,
\end{equation}
where $\bar{Z}^l$ and $\bar{Z}^r$ are independent copies of $\bar{Z}$ from~\eqref{e:1.9a}. In view of~\eqref{e:2.8}, this gives the desired statement with $C_\star := 2C_A C_B \in (0,\infty)$.
\end{proof}

\subsection{Proof of Theorem~\ref{t:2.5}}
Next, we present the top-level proof of Theorem~\ref{t:2.5}. In order to compare the cover time measured in terms of the local time at the root and the real cover time, we show that for any $1 \leq k \leq n$ and $s \geq 0$ one can define a random time $\nu_{k,s}$ which is measurable with respect to the walk $X_n$ and an additional independent random variable $\xi \sim \cN(0,1/2)$ such that two statements hold. First, the laws of the local time field on $\bbL_k$ at time $s$ and at time $\nu_{k,s}$, both measured in terms of the local time at the root, are close to each other. Second, when the local time at the root is $\nu_{k,s}$, the total running time of the walk is with high probability close to an explicit expression involving $s$ and $\xi$. This is formalized in the following theorem.
\begin{thm} 
\label{t:1.2.5} 
Given $n \geq 1$ and a random variable $\xi \sim \mathcal{N}(0,1/2)$ independent of the walk $X_n$, for each $1 \leq k \leq n$ and any $s \geq 0$ there exists a $\sigma( X_n , \xi)$-measurable random time $\nu_{k,s}$ such that for any $\epsilon > 0$,
	\begin{equation}
	\label{e:202.23}
	\lim_{k \to \infty} \limsup_{s \to \infty} \sup_{n \geq \sqrt{s}} \mathbb{P} \Big(\Big|\sqrt{2^{-(n+1)} \mathbf{L}^{-1}_{n, \nu_{k,s}}(0)} + \xi - \sqrt{s} \Big| > \epsilon \Big) = 0.
	\end{equation}
Moreover, there exists a coupling $(L'_{\nu_{k,s}}(\bbL_k), L''_s(\bbL_k))$ of the local time fields $L_{\nu_{k,s}}(\bbL_k)$ and $L_s(\bbL_k)$ such that for any $\epsilon > 0$,
	\begin{equation}
	\label{e:2.19b}
	\lim_{k \to \infty} \limsup_{s \to \infty} \mathbb{P} \Big(\Big\| \sqrt{L'_{\nu_{k,s}}}(\bbL_k) - \sqrt{L''_{s}}(\bbL _k) \Big\|_\infty > \epsilon \Big) = 0\,.
	\end{equation} 
\end{thm}

Using Theorem~\ref{t:1.2.5} we can now prove Theorem~\ref{t:2.5}.
\begin{proof}[Proof of Theorem~\ref{t:2.5}]
Fix $s_0 \in \bbR$ and for $n \geq 1$, $\epsilon \in \bbR$, let $s$ and $s^\epsilon$ be defined via the relations,
\begin{equation}
\sqrt{s} := \sqrt{t_n^{\rmC}} + s_0
\quad, \qquad
\sqrt{s^{\epsilon}} = \sqrt{s} + \epsilon = \sqrt{t_n^{\rmC}} + s_0 + \epsilon \,,
\end{equation}
whenever such $s$ and $s^\epsilon$ exist. For $u \in \bbR$, let us also set
\begin{equation}
\nonumber
p_n(u) := \bbP \Big(\sqrt{T^{\rmC}_n} \leq u \Big)
\ , \quad
\wh{p}_n(u) := p_n \Big(\sqrt{t_n^\rmC} + u\Big) = \bbP \Big(\sqrt{T^{\rmC}_n} - \sqrt{t_n^\rmC} \leq u \Big)
\ , \quad
\wh{p}_\infty(u) := \lim_{n \to \infty} \wh{p}_n(u) \,,
\end{equation}
where the last limit exists in light of Theorem~\ref{t:1.1}. Clearly all of the quantities in the last display increase with $u$. Moreover, thanks to the explicit form of $\wh{p}_\infty$, as given by Theorem~\ref{t:1.1}, and a straightforward application of the dominated convergence theorem, the function $u \mapsto \wh{p}_\infty(u)$ is continuous.

Now, we first claim that for any $\epsilon > 0$ and $k \in \bbN$,
\begin{equation}\label{e:2.29c}
\lim_{n \to \infty}
\bbP\Big(\Big\| \sqrt{L_{s^{\pm \epsilon}}}(\bbL_k) - \Big(\sqrt{L_{s}}(\bbL_k) \pm \epsilon\Big)\Big\|_\infty > \epsilon/2 \Big) = 0 \,.
\end{equation} 
Indeed, by the mean value theorem for each $x \in \bbL_k$, we have
\begin{equation}
\label{e:2.23}
\sqrt{L_{s^{\epsilon}}(x)} - \sqrt{L_{s}(x)} \geq \frac{L_{s^{\epsilon}}(x) - L_{s}(x)}{2 \sqrt{L_{s^{\epsilon}}(x)}} = \frac{ \epsilon\sqrt{s}+\epsilon^2/2}{\sqrt{s} + \epsilon} \times 
\frac{\frac{L_{s^{\epsilon}}(x) - L_{s}(x)}{2\epsilon\sqrt{s}+\epsilon^2}}{ \sqrt{L_{s^{\epsilon}}(x)/s^{\epsilon}}}\,.
\end{equation}
Since $L_{s^{\epsilon}}(x) - L_{s}(x) \eqd L_{2\epsilon\sqrt{s}+\epsilon^2}(x)$, using that $L_t(x)/t$ tends to $1$ in probability, as $t \to \infty$ for any fixed $x \in \bbT$, as implied by~\eqref{e:3.1a}, the right hand side above tends to $\epsilon$ in probability, as $n \to \infty$. Bounding $\sqrt{L_{s^{\epsilon}}(x)} - \sqrt{L_{s}(x)}$ from above by $\big(L_{s^{\epsilon}}(x) - L_{s}(x)\big)/\big(2 \sqrt{L_s(x)}\Big)$ and proceeding in a similar way then gives the opposite inequality, and the case $s^{-\epsilon}$ is handled in exactly the same way.

Next, we make use of the coupling from the second part of Theorem~\ref{t:1.2.5}, the Markov property of the local time field and monotonicity of $p_n$ to write,
\begin{equation}
\label{e:2.15}
\begin{split}
\bbP \big(T_n^\rmC \leq \nu_{k,s}\big) & = \bbE \prod_{x \in \bbL_k} p_{n-k} \Big(\sqrt{L_{\nu_{k,s}}(x)}\Big) \\
& \leq \bbE \prod_{x \in \bbL_k} p_{n-k} \Big(\sqrt{L_s(x)} + \epsilon/2\Big) + 
\bbP \Big(\Big\| \sqrt{L'_{\nu_{k,s}}}(\bbL_k) - \sqrt{L''_s}(\bbL _k) \Big\|_\infty > \epsilon/2 \Big) \\
& \leq \bbE \prod_{x \in \bbL_k} p_{n-k} \Big(\sqrt{L_{s^{\epsilon}}(x)}\Big)
+ 
\bbP \Big(\Big\| \sqrt{L'_{\nu_{k,s}}}(\bbL_k) - \sqrt{L''_s}(\bbL _k) \Big\|_\infty > \epsilon/2 \Big) \\
& \qquad \qquad + \bbP\Big(\Big\| \sqrt{L_{s^{\epsilon}}}(\bbL_k) - \Big(\sqrt{L_{s}}(\bbL_k) +\epsilon\Big)\Big\|_\infty > \epsilon/2 \Big)\,.
\end{split}
\end{equation}
Thanks to the second part of Theorem~\ref{t:1.2.5} and~\eqref{e:2.29c}, the last two probabilities tend to $0$ when $n \to \infty$ followed by $k \to \infty$. At the same time, by the Markov property again, the last expectation is equal to $\bbP \big(T_n^{\rmC} \leq s^{\epsilon} \big) = \wh{p}_n(s_0 + \epsilon)$, which tends to $\wh{p}_\infty(s_0)$ when $n \to \infty$ followed by $\epsilon \to 0$, in light of Theorem~\ref{t:1.1} and the continuity of $\wh{p}_\infty$. 
Consequently, if we take $n \to \infty$ followed by $k \to \infty$ and finally $\epsilon \to 0$ in~\eqref{e:2.15}, we obtain
\begin{equation}
\label{e:2.27}
\limsup_{k \to \infty} \limsup_{n \to \infty} \bbP \big(T_n^\rmC \leq \nu_{k,s}\big) \leq \wh{p}_\infty(s_0) \,.
\end{equation}

Repeating the derivation in~\eqref{e:2.15} with $\sqrt{L_s(x)} - \epsilon/2$ and $\sqrt{L_{s^{-\epsilon}}(x)}$ 
in place of $\sqrt{L_s(x)} + \epsilon/2$ and $\sqrt{L_{s^\epsilon}(x)}$ respectively, and with the inequalities reversed, gives
\begin{equation}
\liminf_{k \to \infty} \liminf_{n \to \infty} \bbP \big(T_n^\rmC \leq \nu_{k,s}\big) \geq \wh{p}_\infty(s_0) \,.
\end{equation}
This together with~\eqref{e:2.27} then shows that 
\begin{equation}
\label{e:2.28}
\lim_{k \to \infty} \limsup_{n \to \infty} \Big| \bbP \big(T_n^\rmC \leq \nu_{k,s}\big) - \wh{p}_\infty(s_0) \Big| = 0 \,. 
\end{equation}

Now, let $\xi$ be the random variable from the statement of Theorem~\ref{t:1.2.5} and use the union bound to write,
\begin{equation}
\bbP \Big(\sqrt{2^{-(n+1)}\bfT^{\bfC}_n} + \xi \leq \sqrt{s} \Big)
\leq \bbP  \Big(\sqrt{2^{-(n+1)} \bfL^{-1}_{n, \nu_{k,s^{\epsilon}}(0)}} + \xi < \sqrt{s} \Big)
+ \bbP \Big(T^{\rmC}_n \leq \nu_{k,s^{\epsilon}} \Big) 
\end{equation}
and
\begin{equation}
\bbP \Big(T^{\rmC}_n \leq \nu_{k,s^{-\epsilon}} \Big) 
\leq \bbP \Big(\sqrt{2^{-(n+1)} \bfL^{-1}_{n, \nu_{k,s^{-\epsilon}}}(0)} + \xi > \sqrt{s} \Big)
+
\bbP \Big(\sqrt{2^{-(n+1)} \bfT^{\bfC}_n} + \xi \leq \sqrt{s}  \Big)  \,.
\end{equation}
Letting $n \to \infty$ followed by $k \to \infty$, it then follows from~\eqref{e:202.23} and~\eqref{e:2.28} (with $s_0 \pm \epsilon$ in place of $s_0$), that
\begin{equation}
\label{e:102.31}
\wh{p}_\infty(s_0-\epsilon) \, \leq \,
\varliminf_{n \to \infty} 
\bbP \Big(\sqrt{2^{-(n+1)}\bfT^{\bfC}_n} + \xi \leq \sqrt{s} \Big) 
\, \leq \,
\varlimsup_{n \to \infty} 
\bbP \Big(\sqrt{2^{-(n+1)}\bfT^{\bfC}_n} + \xi \leq \sqrt{s} \Big) 
\, \leq \, \wh{p}_\infty(s_0 + \epsilon)\,.
\end{equation}
Since this is true for all $\epsilon$, the continuity of $\wh{p}_\infty$ again completes the proof.

\end{proof}

\subsection{Proof of Theorem~\ref{t:1.3}}
In order to prove Theorem~\ref{t:1.3}, we first need to address the existence of the field $\bfh$ and derive a distributional relation between $\bfZ$ and $Z$. Both are stated in the following proposition, whose rather elementary proof is given in Section~\ref{s:10}.

\begin{prop}
\label{p:2.6}
There exists a centered Gaussian random field $\bfh = \big(\bfh(x) :\: x \in  \bbT^1 \cup \bbT^2\big)$ satisfying~\eqref{e:301.11} and~\eqref{e:301.12}. Moreover, if $\bfZ$ is defined as in~\eqref{e:301.13} and $Z$ is defined as in~\eqref{e:1.9} then
\begin{equation}
\label{e:301.15}
\bfZ \Lambda \eqd Z \,,
\end{equation}
for $\Lambda \sim \text{Log-normal}(-2 \log 2, 2 \log 2)$ taken to be independent of $\bfZ$. In particular, $\bfZ$ is almost-surely finite and positive.
\end{prop}

Let us now prove Theorem~\ref{t:1.3}.
\begin{proof}[Proof of Theorem~\ref{t:1.3}]
Set $\tau_n := \sqrt{2^{-(n+1)} \bfT^{\bfC}_n} - \sqrt{t_n^{\rmC}}$ and let $\zeta$ be the weak limit as $n \to \infty$ of $\sqrt{T^{\rmC}_n} - \sqrt{t_n^{\rmC}}$, the existence of which is guaranteed by Theorem~\ref{t:1.1}. It follows from Theorem~\ref{t:2.5} that $\tau_n + \xi$ converges weakly to $\zeta$ as $n \to \infty$, where $\xi \sim \cN(0,1/2)$ and independent of $\tau_n$. Denoting by $\wh{\tau}_n$, $\wh{\xi}$ and $\wh{\zeta}$ the characteristic functions of $\tau_n$, $\xi$ and $\zeta$ respectively, the latter implies that 
$\wh{\tau}_n \wh{\xi}$ tends pointwise to $\wh{\zeta}$ under the same limit. Since $\wh{\xi}$ is never zero, we can divide by it and assert the convergence of $\wh{\tau}_n$ to $\wh{\zeta}/\wh{\xi}$. Since $\wh{\zeta}$ and $\wh{\xi}$ are continuous at $0$, so must be $\wh{\zeta}/\wh{\xi}$. But then, by standard theory of weak convergence $\tau_n$ must converge weakly to some random variable $\tau$ whose characteristic function $\wh{\tau}$ is $\wh{\zeta}/\wh{\xi}$. Multiplying by $\wh{\xi}$ we see that $\tau$ must satisfy 
\begin{equation}
\label{e:2.30}
\tau + \xi \eqd \zeta \,,
\end{equation}
where the random variables on the left hand side are taken to be independent. We remark that $\tau$ is uniquely defined via the above relation (since this claim holds for its characteristic function). 

In view of the right hand side in the statement of Theorem~\ref{t:1.1}, we can further write
\begin{equation}
\zeta \eqd \frac{1}{2 \sqrt{\log 2}} \log Z + G \,,
\end{equation}
where $G$ is chosen according to the Gumbel distribution with rate $2\sqrt{\log 2}$ (and a proper shift determined by $C_\star$) and is independent of $Z$. Now, let $\bfZ$ be as in~\eqref{e:301.13}, whose existence, relation to $Z$ via~\eqref{e:301.15} and almost-sure finiteness and positivity are all verified in Proposition~\ref{p:2.6}. Taking the logarithm in~\eqref{e:301.15} and dividing by $2 \sqrt{\log 2}$, we see that $\bfZ$ must satisfy
\begin{equation}
\frac{1}{2 \sqrt{\log 2}} \log \bfZ + \xi - \sqrt{\log 2} \eqd \frac{1}{2 \sqrt{\log 2}} \log Z \,,
\end{equation}
for $\xi$ with the same law as before and independent of $\bfZ$. 
This in turn implies that if we replace $\tau$ in~\eqref{e:2.30} by $G + (2 \sqrt{\log 2})^{-1} \log \bfZ - \sqrt{\log 2}$ with both random variables in this sum taken independent, then the equality in law there will still hold. But since $\tau$ is uniquely defined by~\eqref{e:2.30}, we must have $\tau \eqd G + (2 \sqrt{\log 2})^{-1} \log \bfZ - \sqrt{\log 2}$, which in explicit form is exactly~\eqref{e:201.11}. 
\end{proof}

\section{Preliminaries}
\label{s:3}
In this section we collect some results and tools which will be used in the sequel. Most statements here are taken from existing literature, while a few require proofs which are either straightforward or standard in the subject. Therefore, in order not to divert attention from the main argument, proofs which are not elementary will be relegated to Appendix~\ref{s:A}. Throughout the entire section the underlying graph is assumed to be the tree $\ol{\bbT}$.

\subsection{The DGFF and its relation to the local time field} 
\subsubsection{The Isomorphism Theorem}
Let us recall that $(h(x) :\: x \in \ol{\bbT})$, as defined below~\eqref{e:1.9a}, is the DGFF on $\ol{\bbT}$ with zero imposed at the root as boundary conditions (for more information on the DGFF, see for example~\cite{biskup2017extrema}). We shall often consider the restriction of $h$ to $\ol{\bbT}_n$ which we denote by $h_n = (h_n(x) :\: x \in \ol{\bbT}_n)$. 

As mentioned in the introduction, a crucial tool which we use frequently in the proof is (a version of) the Second Ray-Knight Theorem (or Dynkin's Isomorphism Theorem) due to~\cite{eisenbaum2000ray}. For convenience of use, we present this theorem as an almost-sure equivalence under a coupling of the processes involved.
\begin{thm}[Second Ray-Knight Theorem]
\label{t:103.1}
For all $t \geq 0$ and $n \geq 0$, there exists a coupling of $L_{n,t} = (L_{n,t}(x) :\: x \in \ol{\bbT}_n)$ and two copies of the DGFF, $h_n = (h_n(x) :\: x \in \ol{\bbT}_n)$
and $h'_n = (h'_n(x) :\: x \in \ol{\bbT}_n)$, such that $L_{n,t}$ and $h_n$ are independent of each other and almost-surely,
\begin{equation}
\label{e:3.1}
L_{n,t}(x) + h^2_n(x) = 
\big(h'_n(x) + \sqrt{t})^2 
\quad : \ x \in \ol{\bbT}_n \,.
\end{equation}
\end{thm}
Henceforth, whenever we use~\eqref{e:3.1} we implicitly assume that the process $L_{n,t}$, $h_n$ and $h'_n$ are all defined in our probability space and that they are coupled as in Theorem~\ref{t:103.1}. Furthermore, the coupling above will mostly be used with $t = t_n^{\rmA}$ from~\eqref{e:2.4}, in which case we write the right hand side of~\eqref{e:3.1} as $\wh{h}^2_n(x)$, where
\begin{equation}
\wh{h}_n(x) := h'_n(x) + m_n 
\quad , \qquad m_n = \sqrt{t_n^{\rmA}} \,.
\end{equation}

For $u \geq 0$, the sub-level set of $\wh{h}^2_n$ as a function on $\ol{\bbL}_n$ at height $u$ is defined as
\begin{equation}
\label{e:103.3}
\cG_n(u) := \big \{ x \in \ol{\bbL}_n :\: \wh{h}^2_n(x) \leq u \big \} \,.
\end{equation}
Observe that under the coupling in Theorem~\ref{t:103.1}, we always have
\begin{equation}
\label{e:3.3}
\cG_n(u) \subseteq \cF_{n, t^{\rmA}_n}(u)  \,,
\end{equation}
since the second term on the left hand side of~\eqref{e:3.1} is always positive. 
Accordingly, we shall sometimes say that a leaf $x \in \cF_{n, t^{\rmA}_n}(u)$ {\em survives the isomorphism} to mean that it is also in $\cG_n(u)$.

\subsubsection{Extreme value theory for the DGFF on the tree}
The main reason for considering the local time field at time $\sqrt{t} = \sqrt{t_n^\rmA} = m_n$ is that the min-extreme (near minimum) values of $h_n$ are typically at height $-m_n + O(1)$. This means that sub-level sets near $0$ of $\wh{h}_n$ coincide with the min-extreme level sets of $h_n$ (both as functions on $\ol{\bbL}_n$). Thanks to the isomorphism and the observation in~\eqref{e:3.3}, the structure of leaves with low local time at $t=t_n^\rmA$ can then be inferred from the structure of min-extreme values of $h_n$. Since the theory of extreme values for the DGFF on the tree is fully developed, this offers a way of studying $\cF_{n,t^\rmA_n}$ and in turn to analyze phase $A$. 

In this subsection, we shall therefore survey the results we need from the theory of extreme values for the DGFF on $\ol{\bbT}$. In the literature, such theory is mostly focused on the BRW. This would have posed no problems if one had been interested in the extreme values of the DGFF on $\bbT$, as the latter is a particular instance of the BRW, where branching is deterministically binary and steps are centered Gaussians. Since we have chosen to use the unconventional tree $\ol{\bbT}$, a bit of work is required to convert off-the-shelf statements to our setup. 

This can be done by modifying the proofs of these results to handle an initial generation with branching by one. Alternatively, one can simply observe (as we have done before) that a BRW on $\bbT$ has the same law as two independent copies of a BRW on $\ol{\bbT}$, or that a BRW on $\ol{\bbT}$ has the same law as a BRW on $\bbT$ with the positions of all particles in all generations shifted by a common centered Gaussian. All of these approaches are straightforward and have been used substantially in the past. We shall therefore just cite the needed results from the literature, converted to the case when the underlying tree is $\ol{\bbT}$ and leave the task of verifying the validity of the conversion to the reader.

We start with the following result showing exponential tightness of the minimum/maximum around $\pm m_n$ respectively. This follows, e.g., from Proposition 1.3 in~\cite{aidekon2013convergence}, where it shown also that, centered by $\pm m_n$, the minimum/maximum of $h_n$ converges in law.
\begin{prop}
\label{p:3.1}
There exists $C, C' > 0$, such that for all $u \geq 0$,
\begin{equation}
\limsup_{n \to \infty}
\bbP \big(\max_{x \in \ol{\bbL}_n} |h_n(x)| > m_n  + u \big) \leq C \rme^{-C' u} \,.
\end{equation}
\end{prop}

Other (min-)extreme values of $h_n$ can be recorded via the so-called {\em structured min-extremal process} of $h_n$. To this end, for $1 \leq r \leq n$ and $y \in \ol{\bbL}_{n-r}$, we set:
\begin{equation}
\wh{h}^*_{n,r}(y) := \min_{x \in \bbL_r(y)} \wh{h}_n(x) 
\quad ;
\qquad
\cC_{n,r}(y) := \sum_{x \in \bbL_r(y)} \delta_{(\wh{h}_n(x) - \wh{h}^*_{n,r}(y))} \,.
\end{equation}
The random variable $\wh{h}^*_{n,r}(y)$ is the minimum of $\wh{h}_n$ on the set of leaves in $\ol{\bbL}_n$ whose common ancestor at generation $n-r$ is $y$. The point process $\cC_{n,r}(y)$ records the heights of all such leaves, relative to the minimum. $\cC_{n,r}(y)$ is sometimes referred to as the {\em cluster} around the minimum in $\bbL_r(y)$. The structured min-extremal process of $h_n$ is then the aggregation of all such pairs of minima and clusters in the form of (yet another) point process:
\begin{equation}
\label{e:3.12}
\chi_{n,r} := \sum_{y \in \ol{\bbL}_{n-r}} \delta_{(\wh{h}^*_{n,r}(y) ,\,\, \cC_{n,r}(y))} \,.
\end{equation}

For the topological setup, we view the cluster $\cC_{n,r}(y)$ as taking values in the space $\cM_0(\bbR_+)$ of finitely bounded measures on $[0,\infty)$, and the process $\chi_{n,r}$ as taking values in the space of finitely bounded measures on $\bbR \times \cM_0(\bbR_+)$. Both spaces are equipped with the Vague Metric. The next proposition shows that $\chi_{n,r}$ converges weakly to an explicit point process. It can be deduced from Theorem~1.1 and Corollary 1.2 in~\cite{madaule2017convergence}.
\begin{thm}
\label{t:3.4}
As $n \to \infty$ followed by $r \to \infty$,
\begin{equation}\label{e:c}
\chi_{n,r} \Longrightarrow {\rm PPP} \big(C_\diamond \bar{Z} \rme^{2 \sqrt{\log 2}u} \rmd u \otimes \nu(\omega) \big) \,,
\end{equation}
where $C_\diamond$ a positive and finite constant, $\bar{Z}$ is as in~\eqref{e:1.9a} and $\nu$ is a distribution on $\cM_0(\bbR_+)$ supported on infinite point measures with an atom at $0$.
\end{thm}
\noindent
The notation $\rm{PPP}(\mu)$, as in the right hand side of~\eqref{e:c}, stands for a Cox process with random intensity measure $\mu$, namely a point process defined conditionally on $\mu$ as a Poisson point process with $\mu$ as its intensity measure. 

We shall not use Theorem~\ref{t:3.4} directly, but rather two propositions which follow from it and concern the set $\cG_n(u)$ from~\eqref{e:103.3}. The first concerns the tightness and asymptotic non-triviality of the size of $\cG_n(u)$. This follows rather immediately from the above theorem together with the almost-sure finiteness and positivity of $\bar{Z}$, and can also be found essentially in~\cite{madaule2017convergence}.
\begin{prop}
\label{p:3.2}
For all $u > 0$, the sequence $(|\cG_n(u)| :\: n \geq 1)$ is tight and, moreover,
\begin{equation}
\label{e:3.9a}
\lim_{u \to \infty} \limsup_{n \to \infty} \bbP \big(\cG_n(u) = \emptyset\big) = 0 \,.
\end{equation}
\end{prop}

The second proposition requires a short proof which is rather standard.  It is therefore relegated to Appendix~\ref{a:1}. 
\begin{prop}
\label{l:3.5a}
For all $u > 0$, there exists $C_u > 0$, such that as $n \to \infty$ followed by $r \to \infty$,
\begin{equation}
\label{e:104.21}
\big|\big[\cG_n(u)\big]_{n-r}\big| \Longrightarrow {\rm Poisson}\big(C_u \bar{Z}) \,,
\end{equation}
where $\bar{Z}$ is as in~\eqref{e:1.9a}. 
\end{prop}
\noindent 
As in Theorem~\ref{t:3.4}, the right hand side stands for the distribution of a random variable which is defined, conditionally on $\bar{Z}$, to have a Poissonian law with intensity $C_u \bar{Z}$.

An important fact that we shall use is that, for leaves $x \in \ol{\bbL}_n$ having an $O(1)$ value under $\wh{h}_n$, their {\em trajectory} under $\wh{h}_n$, namely the sequence $\big((k, \wh{h}_n([x]_k)) :\: k=0, \dots, n \big)$, is typically much higher than the linear interpolation of its endpoints. This behavior, which goes by the name of {\em entropic repulsion}, is well known in the theory of extreme values of genealogically/logarithmically correlated fields, and is a key factor in both its phenomenology and the methods used to study it.

To formulate a quantitative statement, we first introduce the following notation, which will be used frequently in the sequel:
\begin{equation}
\wedge_n(k) := k \wedge (n-k) \,.
\end{equation}
For $\eta \in (0,1/2)$ we also set
\begin{equation}
\label{e:3.10}
\frR_k = \frR_k^{\eta} := \big[k^{1/2 - \eta} \,,\,\, k^{1/2 + \eta} \big] \,,
\end{equation}
and for $u \geq 0$, define
\begin{equation}
\cH_n^{k}(u) = \cH_n^{k, \eta}(u) := \Big\{x \in \cG_n(u) :\: 
\wh{h}_n([x]_k) - \sqrt{\log 2}\,(n-k) \notin \frR_{\wedge_n(k)}^\eta \big] \Big\}\,,
\end{equation}
which is extended to all $K \subseteq [0,n]$ via
\begin{equation}
\cH_n^{K}(u) = \cH_n^{K, \eta}(u) := \bigcup_{k \in K} \cH_n^{k,\eta}(u) \,.
\end{equation}

This entropic repulsion of trajectories of min-extreme leaves of the field $\wh{h}_n$ is then given by the following proposition. We recall that $r_n$ was defined in~\eqref{e:2.3}.
\begin{prop}
\label{p:3.3}
Let $\eta \in (0,1/2)$. For all $u \geq 0$,
\begin{equation}
\label{e:3.9}
\lim_{r \to \infty} \limsup_{n \to \infty} 
\bbP \big(\cH_n^{[r_n, n-r], \eta}(u) \neq \emptyset \big) = 0 \,.
\end{equation}
\end{prop}
\noindent Statements such as these appear in various places in the literature (c.f. Theorems~2.2 and~2.3 in~\cite{arguin2011genealogy} in the case of the related model of branching Brownian motion). However, they are usually either implicitly contained in a proof of a different statement or slightly weaker than what we need here. We therefore provide a proof for this proposition in the appendix. We remark that the exponent $1/2-\eta$ (with arbitrarily small $\eta$) in the lower boundary of $\frR_k$ is much stronger than what we actually require. The proofs of all of the main theorems go through without change for any $\eta < 1/4$ and with mild changes also for $C \log \wedge_n(k)$ for sufficiently large $C > 0$ in place of $(\wedge_n(k))^{1/2-\eta}$.

\subsection{Local time preliminaries}
Next we collect some results concerning the local time of a continuous time simple random walk. 

\subsubsection{Some simple estimates for the local time on the tree}
The following is an easy bound on the local time whenever $\sqrt{t} \gg n^2$, which follows easily from the isomorphism theorem.
\begin{lem}
\label{l:3.4a} Given $\eta > 0$, there exist  constants $C,C' > 0$ such that, for all $n \geq 0$ and $t \geq 0$ satisfying $\sqrt{t} \geq n^{2+\eta}$, we have
\begin{equation}\label{e:l.3.18}
\limsup_{n \to \infty} 
\bbP \big(\max_{x \in \ol{\bbL}_n} \big|\sqrt{L_t(x)} - \sqrt{t}\big| > m_n + u \big) \leq C\rme^{-C' u} \,.
\end{equation}
\end{lem}
\begin{proof}
In light of the isomorphism, for $t$ as in the statement of the lemma, on the event that $|h_n(x)| \vee |h'_n(x)| \leq m_n + u-1$ we have, by Taylor expansion, that
\begin{equation}
\sqrt{L_t(x)} = \sqrt{\big(h'_n(x) + \sqrt{t}\big)^2 - h^2_n(x)}
= h'_n(x) + \sqrt{t} + O \bigg( \frac{h^2_n(x)}{\sqrt{t}} \bigg)
\end{equation}
which is in $[\sqrt{t} - m_n - u, \sqrt{t} + m_n + u]$ for all $n$ large enough. Hence, the event in \eqref{e:l.3.18} requires that there exists $x \in \ol{\bbL}_n$ such that $|h_n(x)| \vee |h'_n(x)| > m_n + u - 1$. But the probability of this is exponentially decaying in $u$ thanks to Proposition~\ref{p:3.1}.
\end{proof}

Next, we give a (coarse) upper bound on the size of $|\cF_{n,t}(u)|$.
\begin{lem}
\label{l:3.5}
For any $u \geq 0$, $t \geq 0$ and $n \geq 1$,
\begin{equation}
\label{e:3.20}
\bbE \big|\cF_{n,t}(u)\big| \leq \rme^{-t/n + 2tu/n^2 + n \log 2 + 1} \,.
\end{equation}
In particular, for each $u \geq 0$ there exists $C_u > 0$ such that if $\sqrt{t} = \sqrt{\log 2}\, n + s$ for some $s \in \bbR$ and $n \geq 4u$,
\begin{equation}
\label{e:3.21}
\bbE \big|\cF_{n,t}(u)\big| \leq C_u \rme^{-\sqrt{\log 2}\, s}  
\end{equation}
and also
\begin{equation}
\label{e:3.21b}
\bbP \big(\cF_{n,t}(u) \neq \emptyset) \leq C_u \rme^{-\sqrt{\log 2}\, s}  \,.
\end{equation}
\end{lem}
\begin{proof}
If $t \leq n$ then the bound \eqref{e:3.20} trivially holds, so henceforth we will assume otherwise. 
For any $x \in \ol{\bbL}_n$, the number of visits to $x$ until time $\bfL^{-1}_{n,t}(0)$ is Poisson with parameter $t/n$. Once at $x$, the total time spent there before returning to the root is exponentially distributed with rate $1/n$. Since the accumulated times at different excursions from the root are independent, if we let $X$, $Y$ be independent Poissons with rates $t/n$ and $u/n$ respectively, we have
\begin{equation}
\bbP \big(x \in \cF_{n,t}(u)\big) = \bbP \big(X \leq Y) 
\end{equation}
Since $t/n \geq 1$, by conditioning on $Y$ we can upper bound this probability by
\begin{equation}
\bbE \Big( \sum_{x=0}^Y \bbP(X = x) \Big)
\leq \rme^{-t/n} \bbE (Y+1) \big(\tfrac{t}{n} \big)^Y
= \rme^{-t/n} (1+tu/n^2) \rme^{tu/n^2 - u/n} \,,
\end{equation}
which is at most $\rme^{-t/n + 2tu/n^2}$. Then~\eqref{e:3.20} follows by summation. 

Plugging $\sqrt{t} = \sqrt{\log 2}n + s$, the right hand side of~\eqref{e:3.20} is equal to
\begin{equation}
\exp \Big(-2 \big(1 - \tfrac{2u}{n}\big) \sqrt{\log 2}\, s + (2 \log 2) u - \big(1 - \tfrac{2u}{n} \big) \tfrac{s^2}{n}  + 1 \Big)
\end{equation}
which is bounded by the right hand side of~\eqref{e:3.21} for $C_u := \rme^{(2 \log 2) u + 1}$, 
whenever $n \geq 4u$. Finally,~\eqref{e:3.21b} follows by Markov's inequality.
\end{proof}

\subsubsection{Local time on a branch of the tree}
The local time of a one-dimensional continuous time random walk forms a zero-dimensional squared Bessel process (see, for example Lemma 7.7 in~\cite{belius2017subleading}), as shown in the next lemma.
\begin{lem}
\label{l:3.8}
Let $t \geq 0$, $n \geq 1$ and $x \in \ol{\bbL}_n$. Then
\begin{equation}
\bbP \Big(\big(L_{n,t}([x]_k) :\: k=0, \dots,n) \in \cdot \Big)
= 
\bbP \Big(\big(Y_k :\: k=0, \dots,n) \in \cdot \Big) \,,
\end{equation}
where $Y = (Y_s :\: s \in [0,n])$ is one half times a zero-dimensional squared Bessel process starting from $2t$.
\end{lem}

The next two lemmas are useful in studying the process $Y$. They can be found, e.g., in~\cite{abe2018extremes}.
\begin{lem}
\label{l:3.9}
Let $Y$ be the process from Lemma~\ref{l:3.8} for some $t \geq 0$ and $n \geq 1$. Then for every measurable function $\varphi: \bbR^{n+1} \to \bbR$, we have
\begin{equation}
\bbE \Big(\varphi \big(Y_0, \dots, Y_n) \:;\:  Y_n \neq 0 \Big)
= \bbE \bigg( \varphi \big(B^2_0, \dots, B^2_n) \sqrt{\frac{t^{1/2}}{B_n}}
\exp \Big(-\tfrac{3}{16} \int_0^n B^{-2}_s \rmd s \Big) \;;\; \min_{s \in [0,n]} B_s > 0 \bigg) \,, 
\end{equation}
where $(B_s :\: s \in [0,n])$ is a Brownian motion starting from $\sqrt{t}$ with variance $1/2$.
\end{lem}
\begin{lem}
\label{l:bp}
For each $s \geq 0$, the law $Q_s$ of $Y_s$ is given by  
\begin{equation}
Q_s(A)=\exp\left(-\frac{t}{s}\right) 1_A(0) + \int_A f_s(y)dy\,,
\end{equation} for any Borel set $A \subseteq \bbR_+$, where $f_s(y)$ satisfies
\begin{equation}
\label{e:bp1}
f_s(y) \leq C s^{-1} \sqrt{\frac{t}{y}}\exp\left(-\frac{(\sqrt{t}-\sqrt{y})^2}{s}\right) \,,
\end{equation} 
for all $y \geq 0$ and some absolute constant $C > 0$.
\end{lem}

\subsubsection{Soft entropic repulsion of local time trajectories}
As the last preliminary result, we need a statement concerning the entropic repulsion of the local time trajectory of non-visited leaves. To this end, for $\eta' > 0$ and $K \subset [1,n]$, define
\begin{equation}
\mathcal{O}^K_n = \mathcal{O}^{K, \eta'}_n :=\Big\{ x \in \cF_{n,t^{\rmA}_n}(0)\,:\, \forall k \in K \,,\sqrt{L_{t^{\rmA}_n}([x]_k)} \geq \sqrt{\log 2}(n-k) + n^{\eta'}\Big\}\,.
\end{equation} 
Then,
\begin{prop}\label{p:8.3} 
There exists $\eta' > 0$ such that for any $\delta > 0$,
\begin{equation}\label{e:8.3b}
\lim_{n \to \infty} \bbP\Big( |\cF_{n, t^{\rmA}_n}(0) \setminus \mathcal{O}^{[r_n,n-r_n], \eta'}_{n}| > \delta \sqrt{n}\Big) = 0\,.
\end{equation}
\end{prop}
Similar repulsion statements, usually in a stronger form, exist in the literature (c.f.~\cite{belius2017barrier}), albeit not for the local time at $t=t^\rmA_n$. Our proof follows the usual ``barrier'' approach and introduces no new ideas. In fact, since we only need to exhibit repulsion far from the root and the leaves, the proof is much simpler compared to that of similar results in other works. It can be found in the appendix.

\section{Upper bounds as consequences of the isomorphism}\label{s:upper}
\label{s:4}
In this section we provide upper bounds on the number of $r_n$-clusters of $\cF_{n,t^{\rmA}_n}(u)$ for $u \geq 0$, with and without additional constraints (for the definition of $r_n$ and  $r_n$-clusters, see~\eqref{e:2.3}). These will be used both in Section~\ref{s:5} to study the fine clustering structure of the set $\cF_{n,t^{\rmA}_n}(u)$ and in Section~\ref{s:7}, where we prove the key theorem for phase $A$.
To economize on notation, henceforth we shall write $\cF_{n}(u)$ as a short for $\cF_{n,t^{\rmA}_n}(u)$ and note that whenever we need to consider the set $\cF_{n,t}(u)$ at a time other than $t=t^{\rmA}_n$, we will use the original explicit notation. 

\subsection{The number of $r_n$-clusters}
From the isomorphism theorem we get the following upper bound for the number of $r_n$-clusters:
\begin{lem}
\label{l:3.2}
For each $u \geq 0$, we have
\begin{equation}
\lim_{\delta \to 0}
\limsup_{n \to \infty} \bbP \Big(\big|\big[\cF_n(u)\big]_{r_n}\big| > \delta^{-1} \sqrt{n} \Big) = 0\,.
\end{equation}
\end{lem}

\begin{proof}
Fix $u \geq 0$ and for any $n \geq 1$, define the random set $\cF^*_n(u)$ from $\cF_n(u)$ by keeping one and only one vertex among those sharing the same ancestor at generation $r_n$ according to some predefined but arbitrary rule. Observe that by definition $[\cF^*_n(u)]_{r_n} = [\cF_n(u)]_{r_n}$. 

Now, employing the coupling from~\eqref{e:3.1}, the event
\begin{equation}
\label{e:3.4}
\Big\{\big|\cG_n(u+1)\big| > \delta^{-1} \Big\} \\
\end{equation}
is implied by the intersection:
\begin{equation}
\label{e:3.5}
\Big\{\big|[\cF_n(u)]_{r_n}\big| > \delta^{-1} \sqrt{n} \Big\}
\cap
\Big\{ \max_{x \in \ol{\bbL}_{r_n}} |h_n(x)| < r_n \log r_n \Big\} 
\cap 
\Big\{\big|\big\{x \in \cF^*_n(u) :\: |h_n(x)| \leq 1 \big\}\big| > \delta^{-1} \Big\} \,.
\end{equation}

By the Markov property for $h_n$, conditional on $h_n(\ol{\bbL}_{r_n})$ the law of $h_n(x)$ for $x \in \ol{\bbL}_n$ is Gaussian with mean $h_n([x]_{r_n})$ and variance $(n-r_n)/2$. Moreover, under this conditioning $h_n(x)$ and $h_n(y)$ are independent whenever $x,y \in \ol{\bbL}_n$ and $[x]_{r_n} \neq [y]_{r_n}$. Using this together with the independence of $L_{t^{\rmA}_n}$ and $h_n$, conditional on the first two events in~\eqref{e:3.5}, the law of 
$\big|\big\{x \in \cF^*_n(u) :\: |h_n(x)| \leq 1 \big\}\big|$ stochastically dominates a Binomial distribution with $\lceil\delta^{-1} \sqrt{n}\rceil$ trials and success probability given by
\begin{equation}
\frac{2}{\sqrt{\pi(n-r_n)}} \exp \Big(-\tfrac{(r_n \log r_n+1)^2}{n-r_n} \Big) = \frac{2}{\sqrt{\pi n}}(1+o(1)) \,.
\end{equation}

It therefore follows by Chebyshev's inequality that the probability of the last event in~\eqref{e:3.5} conditional on the first two tends to $1$ as $n \to \infty$ followed by $\delta \to 0$. Thanks to Proposition~\ref{p:3.1}, the probability of the middle event also tends to $1$ under the same limit.
Using the product rule and the independence between $L_{t^{\rmA}_n}$ and $h_n$, it follows that for all $n$ large enough, we have
\begin{equation}
\bbP \Big(\big|[\cF_n(u)]_{r_n}\big| > \delta^{-1} \sqrt{n} \Big)
\leq 2 \bbP \Big(\big|\cG_n(u+1)\big| > \delta^{-1} \Big) \,.
\end{equation}
But then, the tightness of the min-extreme level sets, as given by Proposition~\ref{p:3.2}, completes the proof.
\end{proof}

\subsection{The number of clusters without proper entropic repulsion}
The same line of argument can also yield an upper bound on the number of $r_n$-clusters with leaves whose local time trajectory is not properly repelled. For $\eta \in (0,1/2)$ recalling the definition of $\frR_k^\eta$ from~\eqref{e:3.10}, for $t \geq 0$ and $k \in [r_n, n]$ we define
\begin{equation}
\label{e:6.1}
\cR^{k}_{n,t}(u) = \cR^{k,\eta}_{n,t}(u) := \Big \{ x \in \cF_{n,t}(u) : \: 
\sqrt{L_{t}([x]_{k})} \notin \sqrt{\log 2}(n-k) + \frR_{\wedge_n(k)}^\eta \Big\} \,.
\end{equation}
As usual, we also set for $K \subset [r_n, n]$,
\begin{equation}
\label{e:4.7}
\cR_{n,t}^{K}(u) = \cR_{n,t}^{K,\eta}(u) := \bigcup_{k \in K} \cR_{n,t}^{k,\eta}(u) \,,
\end{equation}
and omit $t$ from the subscript if it is equal to $t_n^{\rmA}$.
We then have:
\begin{lem}
\label{l:3.3}
Let $\eta \in (0,1/2)$. For all $u \geq 0$ and $\delta > 0$, 
\begin{equation}
\label{e:3.11}
\lim_{r \to \infty} \limsup_{n \to \infty} 
\bbP \Big(\big|\big[\cR_n^{[r_n, n-r], \eta}(u)\big]_{r_n}\big| > \delta \sqrt{n} \Big)= 0 \,.
\end{equation}
\end{lem}

\begin{proof}
We fix $u$ and $\eta$ as in the conditions of the lemma, and omit the dependence on $u$ from the notation henceforth. By definition, for each $x \in \cR_n^{[r_n, n-r]}$, there exists a largest $k(x) \in [r_n,n-r]$ such that
\begin{equation}
\sqrt{L_{t_n^{\rmA}}([x]_{k(x)})} \notin \sqrt{\log 2}\big(n-k(x)\big) + \frR_{\wedge_n(k(x))}^\eta  \,.
\end{equation}
Now, for all such $x$ if $\big|h_n([x]_{k(x)})\big| \leq \wedge_n(k(x))^{1/2+\eta'}$ for $\eta' \in (0,1/2)$ then by Taylor expansion,
\begin{equation}
\wh{h}_n([x]_{k(x)}) = \sqrt{L_{t_n^{\rmA}}([x]_{k(x)}) + h_n([x]_{k(x)})^2}
= \sqrt{L_{t_n^{\rmA}}([x]_{k(x)})} + O \big(\!\wedge_n\!(k(x))^{2\eta'}\big) \,,
\end{equation}
whenever $\sqrt{L_{t_n^{\rmA}}([x]_{k(x)})} > \tfrac12 \sqrt{\log 2}(n-k(x))$. Therefore as soon as
$2\eta' < 1/2-\eta$ and $k(x)$ is large enough we will have
\begin{equation}
\wh{h}_n([x]_{k(x)}) \notin \sqrt{\log 2}\big(n-k(x)\big) + \frR_{\wedge_n(k(x))}^{\eta''} \,,
\end{equation}
for any $\eta'' \in (0, \eta)$. The same conclusion also holds when $\sqrt{L_{t_n^{\rmA}}([x]_{k(x)})} \leq \tfrac12 \sqrt{\log 2}(n-k(x))$, since then $\wh{h}_n([x]_{k(x)}) \leq \tfrac34 \sqrt{\log 2}(n-k(x))$. 

We now argue as in the proof of Lemma~\ref{l:3.2}. Let $\cR_n^*$ be a (random) subset of $\cR_n^{[r_n, n-r]}$ obtained by choosing one vertex among all ones in $\cR_n^{[r_n, n-r]}$ sharing the same ancestor in generation $r_n$, so that $[\cR_n^*]_{r_n} = [\cR_n^{[r_n, n-r]}]_{r_n}$. Then for any $v \geq 0$ and $\delta > 0$, on the event
\begin{multline}
\label{e:3.18}
\Big\{\big|[\cR_n^{[r_n, n-r],\eta}]_{r_n}\big| > \delta \sqrt{n} \Big\}
\cap
\Big\{ \max_{x \in \ol{\bbL}_{r_n}} |h_n(x)| < r_n \log r_n \Big\}  \\
\cap 
\Big\{\Big|\big\{x \in \cR_n^* :\: 
\big|h_n(x)\big| \leq \sqrt{v} \,,\,\, \big|h_n([x]_{k(x)})\big| \leq \wedge_n(k(x))^{1/2+\eta'}
\big\} \Big| > \delta v^{1/4} \Big\}  \,,
\end{multline}
we must have 
\begin{equation}
\Big\{ \big| \cH_n^{[r_n, n-r], \eta''}(u+v) \big| > \delta v^{1/4} \Big \} \,.
\end{equation}

To estimate the probability of the third event in~\eqref{e:3.18}, we observe that for any $x \in \ol{\bbL}_n$, $w \in (-r_n \log r_n,\, r_n \log r_n)$ and $k \in [r_n, n]$ we can write
\begin{multline}
\label{e:4.14}
\bbP \Big(\big|h_n(x)\big| \leq \sqrt{v} \,,\,\, \big|h_n([x]_k)\big| \leq \wedge_n(k(x))^{1/2+\eta'} \,\Big|\, h_n([x]_{r_n}) = w \Big) = \\
\int_{-\sqrt{v}}^{\sqrt{v}} \bbP \Big(h_n(x) \in \rmd v' \Big|\, h_n([x]_{r_n}) = w \Big)
\bbP \Big(\big|h_n([x]_k)\big| \leq \wedge_n(k)^{1/2+\eta'} \,\Big|\, h_n([x]_{r_n}) = w ,\, h_n(x) = v' \Big)  \,.
\end{multline}
Now, conditional on $h_n([x]_{r_n}) = w$, the sequence $(h_n([x]_l :\: l=r_n,\dots, n)$ is random walk with Gaussian steps having mean zero and variance $1/2$, starting from $w$. Therefore, the first term in the integrand can be bounded from below by 
\begin{equation}
\frac{1}{\sqrt{\pi (n-r_n)}} \exp \Big(-\tfrac{(\sqrt{v} + r_n \log r_n)^2}{n-r_n} \Big) 
\geq \frac{C}{\sqrt{n}} \,,
\end{equation}
for some $C > 0$ and all $n$ large enough, depending on $v$. 

At the same time, since under the second conditional probability in the integral, $h_n([x]_k)$ is a Gaussian with mean $\mu$ and variance $\sigma^2$ satisfying:
\begin{equation}
|\mu| = \bigg| \frac{w(n-k) + v'(k-r_n)}{n-r_n} \bigg| \leq \sqrt{v} + r_n \log r_n 1_{\{n-k > \sqrt{n}\}} + 1 
, \quad
\sigma^2 = \frac{(k-r_n)(n-k)}{2(n-r_n)} \leq \tfrac12\wedge_n(k) \,. 
\end{equation}
It follows by the Gaussian tail formula, that the second term in the integrand is at least
\begin{equation}
1 - 2\exp \bigg(-\frac{\big( \wedge_n\!(k)^{1/2+\eta'} - \big(\sqrt{v} + r_n \log r_n 1_{\{n-k > \sqrt{n}\}} + 1\big)\big)^2}{\wedge_n(k)} \bigg)
> C' >  0\,,
\end{equation}
whenever $\wedge_n(k)$ is large enough, depending on $v$. Together this shows that the probability on the left hand side of~\eqref{e:4.14} is at least 
$C''\sqrt{v} /\sqrt{n}$ for some $C'' > 0$ whenever $\wedge_n(k)$ is large enough (depending on $v$).

Using this in~\eqref{e:3.18} we see that for any $v \geq 0$, whenever $r$ is large enough, when   we condition on the first two events in~\eqref{e:3.18}, the probability of the third is larger or equal than the probability that a Binomial with $\lceil \delta \sqrt{n} \rceil$ trials each with probability $C''\sqrt{v} / \sqrt{n}$ is at least $\delta v^{1/4}$. It follows by standard arguments that this probability tends to $1$ when $n \to \infty$ followed by $v \to \infty$. Since the middle event in~\eqref{e:3.18} has probability tending to $1$ with $n$ as well, using the independence in the coupling and the product rule this shows that for any $\delta > 0$, we may choose $v$ then $r$ and finally $n$ (all large enough) such that 
\begin{equation}
\bbP \Big(\big|[\cR_n^{[r_n, n-r], \eta}]_{r_n}\big| > \delta \sqrt{n} \Big)
\leq 2 
\bbP \Big( \big|\cH_n^{[r_n, n-r], \eta''}(u+v) \big| > \delta v^{1/4} \Big) \,.
\end{equation}
It remains to observe that when $v$ is such that $\delta \sqrt{v} \geq 1$, the right hand side goes to $0$ when $n \to \infty$ followed by $r \to \infty$, in light of Proposition~\ref{p:3.3}.
\end{proof}

We shall also need a stronger upper bound on the number of $r_n$-clusters of leaves whose local time trajectory is unusually low. To this end, for any $\eta' \in (0,1/2)$ and for $k \in [r_n,n]$ we define the set of vertices in $\ol{\bbL}_k$ whose local time is unusually low.
\begin{equation}
\label{e:5.12}
\cD^k_n
= \cD^{k,\eta'}_n := \Big \{ y \in \ol{\bbL}_k : \: 
\sqrt{L_{t_n^{\rmA}}(y)} \leq \sqrt{\log 2} (n-k) - \wedge_n(k)^{1/2-\eta'} \Big
\} \,.
\end{equation}
Then,
\begin{lem}
\label{l:4.2}
Let $0 < \eta' < \eta'' < 1/2$. Then,
\begin{equation}
\label{e:4.21}
\lim_{r \to \infty} \limsup_{n \to \infty}
\bbP \Big( \exists k \in [r_n, n-r] :\: \big| \big[\cD^{k, \eta'}_n\big]_{r_n} \big| 
> \rme^{-\wedge_n(k)^{1/2-\eta''}} \sqrt{n} \Big) = 0 \,.
\end{equation}
\end{lem}
\begin{proof}
The argument is similar to that in the proof of Lemma~\ref{l:3.2}. We fix $\eta'$ and $\eta''$ as in the statement of the lemma and in the remaining of the proof omit the dependence on these parameters from the notation. Let $k \in [r_n, n-r]$ and set $\cD^{k*}_n$ to be the random set obtained from $\cD_n^k$ by choosing (in a fixed but arbitrary manner) one member from each subset of vertices of $\cD_n^k$ sharing the same ancestor at depth $r_n$. By definition we have $|[\cD_n^k]_{r_n}| = |\cD^{k*}_n|$. Now, for all $y \in \cD_n^{k*}$ if
$|h_n(y)| \leq 1$ then by Taylor expansion,
\begin{equation}
\wh{h}_n(y) = \sqrt{L_{t_n^{\rmA}}(y) + h_n(y)^2}
\leq \sqrt{\log 2} (n-k) - \wedge_n(k)^{1/2-\eta'} + O \big((n-k)^{-1}\big) \,,
\end{equation}
which is at most $m_n - m_k - \tfrac12\! \wedge_n\!(k)^{1/2-\eta'}$ for all $k$ in the range considered whenever $r$ is large enough. Consequently for such $y$ we have $\wh{h}_k(y) \leq -\tfrac12 \! \wedge_n\!(k)^{1/2-\eta'}$.

Therefore, the intersection of
\begin{equation}
\label{e:3.5a}
\Big\{ \big| \cD_n^{k*} \big| 
> \rme^{-\wedge_n(k)^{1/2-\eta''}} \sqrt{n} \Big\}
\cap 
\Big\{ \exists y \in \cD_n^{k*} :\: |h_n(y)| \leq 1 \Big\} \,,
\end{equation}
implies that 
\begin{equation}
\Big\{ \min_{y \in \ol{\bbL}_k} \wh{h}_k(y) \leq -\tfrac12 \! \wedge_n\!(k)^{1/2-\eta'} \Big\} \,.
\end{equation}
Now if $\rme^{-\wedge_n(k)^{1/2-\eta''}} \sqrt{n} < 1$, then conditional on the first event in~\eqref{e:3.5a}, the probability of the second is at least the probability that one $y \in \cD_n^k*$ satisfies $|h_n(y)| \leq 1$. Thanks to the independence between $L_t$ and $h_n$ and since $h_n$ is Gaussian with mean zero and variance $k/2$, this conditional probability is at least $C/\sqrt{n} \geq C \rme^{-\wedge_n(k)^{1/2-\eta''}}$. This gives 
\begin{equation}
\label{e:4.19}
\bbP \Big( \big| \cD_n^{k*} \big| 
> \rme^{-\wedge_n(k)^{1/2-\eta''}} \sqrt{n} \Big)
\leq C \rme^{\wedge_n(k)^{1/2-\eta''}} 
\bbP \Big( \min_{y \in \ol{\bbL}_k} \wh{h}_k(y) \leq -\tfrac12 \! \wedge_n\!(k)^{1/2-\eta'} \Big) \,.
\end{equation}

On the other hand, if $\rme^{-\wedge_n(k)^{1/2-\eta''}} \sqrt{n} \geq 1$, then for $n$ large enough, we must have $k \geq n/2$. In this case, we replace the intersection in~\eqref{e:3.5a} with
\begin{equation}
\label{e:4.17}
\Big\{ \big| \cD_n^{k*} \big| 
> \rme^{-\wedge_n(k)^{1/2-\eta''}} \sqrt{n} \Big\}
\cap 
\Big\{ \max_{x \in \ol{\bbL}_{r_n}} |h_n(x)| < r_n \log r_n \Big\} 
\cap
\Big\{ \exists y \in \cD_n^{k*} :\: |h_n(y)| \leq 1 \Big\} \,.
\end{equation}
Then, conditional on second event in~\eqref{e:4.17} for any $y \in \ol{\bbL}_k$, the probability that $|h_n(y)|\leq 1$ is at least
\begin{equation}
\frac{2}{\sqrt{\pi (k-r_n) }} \exp \Big(-\tfrac{(r_n \log r_n+1)^2}{k-r_n} \Big) \geq \frac{C'}{\sqrt{n}} \,,
\end{equation}
It follows that conditional on the first two events in~\eqref{e:4.17}, the probability of the third is larger or equal than the probability of at least one success in a sequence of 
$\lceil \sqrt{n} \rme^{-\wedge_n(k)^{1/2-\eta''}} \rceil $ independent trials each succeeding with probability
at least $C/\sqrt{n}$. Since the product of the last two quantities tends to $0$ with $n$, standard arguments imply that this probability is at least $C' \rme^{-\wedge_n(k)^{1/2-\eta''}}$. Since the probability of the middle term in~\eqref{e:4.17} tends to $1$ with $n$, as shown by Proposition~\ref{p:3.1}, by the product rule, we again have~\eqref{e:4.19}.

Since $\eta' < \eta''$, invoking Proposition~\ref{p:3.1} again, we get
\begin{equation}
\bbP \Big( \big| \cD_n^{k*} \big| 
> \rme^{-\wedge_n(k)^{1/2-\eta''}} \sqrt{n} \Big)
\leq C \rme^{-\tfrac14 \wedge_n(k)^{1/2-\eta'}}
\end{equation}
Summing the left hand side over $k \in [r_n, \lceil n/2 \rceil]$ gives a quantity which tends to $0$ with $n$, while the same sum over $k \in [\lceil n/2 \rceil, n-r]$ is at most $C \rme^{-C' r^{1/2-\eta'}}$. Since $|[\cD_n^k]_{r_n}| = |\cD^{k*}_n|$, combining the sums and using the union bound, we recover~\eqref{e:4.21}.
\end{proof}

\section{Sharp clustering of leaves with low local time}
\label{s:5}
We continue to the use the notional convention from the previous section, whereby we write $\cF_n(u)$ as a short for $\cF_{n,t^{\rmA}_n}(u)$. This is further extended to the set $\cW_{n,t^{\rmA}_n}(u)$ from~\eqref{e:2.4a}, as well as to other subsets of $\cF_{n,t^{\rmA}_n}(u)$ which appear in the sequel. A key ingredient in the proof of Theorem~\ref{t:2a} is a sharp description of the clustering structure of $\cF_n(u)$. Recall that $\cW_n^{[r_n, n-r]}(u)$ is the set of leaves belonging to $r_n$-clusters whose root is at depth at most $n-r$. The following theorem shows that with high probability, for large $r$ such leaves do not survive the isomorphism and moreover that the number of clusters containing such leaves is $o(1) \sqrt{n}$ with the $o(1)$ term tending to $0$ in probability as $r \to \infty$.

\begin{thm}
\label{t:3.1}
For any $u \geq 0$,
\begin{equation}
\label{e:3.14}
\lim_{r \to \infty}
\limsup_{n \to \infty} \bbP \Big( \cW_n^{[r_n, n-r]}(u) \cap \cG_n(u)  \neq \emptyset \Big) = 0 \,.
\end{equation}
Moreover, for all such $u \geq 0$ and $\delta > 0$,
\begin{equation}\label{e:3.14b}
\lim_{r \to \infty}
\limsup_{n \to \infty} \bbP \Big( \big| \big[\cW_n^{[r_n, n-r]}(u)\big]_{r_n} \big| > \delta \sqrt{n}  \Big) = 0 \,.
\end{equation}
\end{thm}
In order to prove Theorem~\ref{t:3.1} we cover the set $\cW_n^{[r_n, n-r]}(u)$ by three subsets and show that each satisfies~\eqref{e:3.14} for a different reason. The first subset includes leaves whose ancestor at depth $k$ has a large local time. To this end, let $\eta \in (0,1/2)$ and for any $r_n \leq k \leq n$ and $u \geq 0$, define 
\begin{equation}
\label{e:5.1}
\cU^k_n(u) = \cU^{k, \eta}_n(u):= \Big \{ x \in \cW_n^k(u) : \: 
\sqrt{L_{t_n^{\rmA}}([x]_k)} > \sqrt{\log 2} (n-k) + \wedge_n(k)^{1/2-\eta} \Big
\} \,.
\end{equation}
For $x \in \cU_n^k(u)$ the $r_n$-cluster containing $x$, namely $\cF_n(u) \cap \bbT([x]_{r_n})$, will be called a {\em $k$-up-repelled} $r_n$-cluster.

The second subset is that of leaves belonging to clusters which are large:
\begin{equation}
\label{e:5.2}
\cB^k_n(u) =
\cB^{k, \eta}_n(u) := \Big \{ x \in \cW_n^k(u) :\: 
\big| \cW_n^k(u) \cap \bbT\big([x]_{r_n}) \big| > \rme^{\wedge_n(k)^{1/2-\eta}} \Big\} \,,
\end{equation}
where $\eta$, $k$ and $u$ are as in the previous definition. We will call the clusters of such leaves {\em $k$-big}. 
The third subset will simply be $\cW_n^k(u) \setminus \big(\cU^k_n(u)\cup\cB^k_n(u)\big)$.

Setting as usual for $K \subseteq [r_n,n]$,
\begin{equation}
\cU^{K}_{n}(u) = \cU^{K,\eta}_{n}(u) := \bigcup_{k \in K} \cU^{k}_n(u) 
\quad ,\qquad
\cB^K_{n}(u) = \cB^{K, \eta}_{n}(u) := \bigcup_{k \in K} \cB^{k}_n(u) \,,
\end{equation}
we have the following two propositions:
\begin{prop}
\label{p:3.1a}
Let $\eta \in (0,1/2)$. For any $u \geq 0$,
\begin{equation}
\label{e:3.14a}
\lim_{r \to \infty}
\limsup_{n \to \infty} \bbP \Big( \big( \cU_n^{[r_n, n-r], \eta}(u) \cup \cB_n^{[r_n, n-r], \eta}(u)\big) \cap \cG_n(u)  \neq \emptyset \Big) = 0 \,.
\end{equation}
\end{prop}
\begin{prop}
\label{p:3.1c}
Let $\eta \in (0,1/2)$. For any $u \geq 0$,
\begin{equation}
\label{e:3.14c}
\lim_{r \to \infty}
\limsup_{n \to \infty} \bbP \Big(\Big( \cW_n^{[r_n, n-r]}(u) \setminus \big(\cU_n^{[r_n, n-r], \eta}(u) \cup  \cB_n^{[r_n, n-r], \eta}(u)\big) \Big) \cap \cG_n(u)  \neq \emptyset \Big) = 0 \,.
\end{equation}
\end{prop}

These clearly imply Theorem~\ref{t:3.1}. Indeed,
\begin{proof}[Proof of Theorem~\ref{t:3.1}]
~\eqref{e:3.14} follows immediately from Propositions~\ref{p:3.1a} and \ref{p:3.1c} by an application of the union bound. 
As for the proof of~\eqref{e:3.14b}, we will only sketch it as it is very similar to the upper bound arguments already found in Section \ref{s:upper}. 

Fix $u \geq 0$ and, for any $r \geq 1$ and $n \geq r$, let $\cW_n^{*}(u)$ be the random subset of $\cW_n^{[r_n,n-r]}(u)$ obtained by picking exactly one vertex from each and every group of vertices in $\cW_n^{[r_n,n-r]}(u)$ which share the same ancestor in generation $r_n$, according to some predefined but arbitrary rule, so that $[\cW^*_n(u)]_{r_n} = [\cW_n^{[r_n,n-r]}(u)]_{r_n}$. 

By the isomorphism~\eqref{e:3.1}, we see that, for any $\delta>0$, the event
\begin{equation}
\label{e:5.8}
\Big\{ \cW_n^{[r_n, n-r]}(u+1) \cap \cG_n(u+1)  \neq \emptyset \Big\} \\
\end{equation}
is implied by the intersection
\begin{equation}
\label{e:5.9}
\Big\{\big|[\cW_n^{[r_n,n-r]}(u)]_{r_n}\big| > \delta \sqrt{n} \Big\}
\cap
\Big\{ \max_{x \in \ol{\bbL}_{r_n}} |h_n(x)| < r_n \log r_n \Big\} 
\cap 
\Big\{\big|\big\{x \in \cW^*_n(u) :\: |h_n(x)| \leq 1 \big\}\big| > \delta \Big\} \,.
\end{equation}

As in the proof of Lemma~\ref{l:3.2}, the probability of the third event in~\eqref{e:5.9} given the first two is at least the probability of a Binomial random variable with $\lceil\delta \sqrt{n}\rceil$ trials and success probability $\frac{2}{\sqrt{\pi n}}(1+o(1))$ being larger than $\delta$, and is therefore bounded from below by some constant $C > 0$ for all $n$ large enough. Furthermore, again as in Lemma~\ref{l:3.2}, the probability of the second event in~\eqref{e:5.9} given the first tends to $1$ under the same limit. It then follows from the product rule that for all $r \geq 1$ and $n$ large enough,
\begin{equation}
\bbP \Big(\big|[\cW_n^{[r_n,n-r]}(u)]_{r_n}\big| > \delta \sqrt{n} \Big) \leq \frac{2}{C}\bbP \Big(\cW_n^{[r_n, n-r]}(u+1) \cap \cG_n(u+1)  \neq \emptyset \Big) \,,
\end{equation}
so that~\eqref{e:3.14b} now follows from~\eqref{e:3.14}.
\end{proof}

In the remainder of this section we prove the above propositions.

\subsection{Proof of Proposition~\ref{p:3.1a}}
In order to prove Proposition~\ref{p:3.1a} we shall show that
\begin{lem}
\label{l:3.9.5a}
Let $0 < \eta < \eta' < 1/2$. Then for all $u \geq 0$,
\begin{equation}
\label{e:3.47a}
\lim_{r \to \infty} \limsup_{n \to \infty}
\bbP \Big( \exists k \in [r_n, n-r] :\: \big| \big[\cU^{k, \eta}_n(u)\big]_k \big| 
> \rme^{-\wedge_n(k)^{1/2-\eta'}} \sqrt{n} \Big) = 0 \,.
\end{equation}
\end{lem}
and
\begin{lem}
\label{l:3.9.5}
Let $0 < \eta < \eta' < 1/2$. Then for all $u \geq 0$,
\begin{equation}
\label{e:3.47}
\lim_{r \to \infty} \limsup_{n \to \infty}
\bbP \Big( \exists k \in [r_n, n-r] :\: \big| \big[\cB^{k, \eta}_n(u)\big]_k \big| 
> \rme^{-\wedge_n(k)^{1/2-\eta'}} \sqrt{n} \Big) = 0 \,.
\end{equation}
\end{lem}

Let us first see why this proves Proposition~\ref{p:3.1a}.
\begin{proof}[Proof of Proposition~\ref{p:3.1a}]
We fix $u$ and take $r$ large enough. Writing $\cU_n^k$, $\cB^k_n$ and $\cG^k_n$ as shorts for
$\cU_n^{k, \eta}(u)$, $\cB^{k, \eta}_n(u)$ and $\cG^k_n(u)$, for any $r_n \leq k \leq n-r$, by the union bound we have,
\begin{equation}
\label{e:3.40}
\bbP \Big( \big(\cU_n^k \cup \cB_n^k\big) \cap \cG_n \neq \emptyset \, \Big|\, L_{t_n^{\rmA}}(\ol{\bbT}_n) \Big) 
\leq \sum_{y \in [\cU_n^k \cup \cB_n^k]_k} \bbP \big(\exists x \in \bbL_{n-k}(y) :\:  
|h_n(x)| \leq \sqrt{u} \big) \,.
\end{equation}

The probability in the sum can be bounded from above by
\begin{multline}
\bbP \big(|h_n(y)| \leq (n-k)^2 \big) +
\bbP \big(\exists x \in \ol{\bbL}_{n-k} :\: |h_{n-k}(x)| \geq (n-k)^2-\sqrt{u} \big) \\
\leq C \Big(\tfrac{(n-k)^2}{\sqrt{k}} + \rme^{-C' (n-k)^2}\Big) \leq C'' \tfrac{(n-k)^2}{\sqrt{k}} \,,
\end{multline}
where we have used the fact that $h_n(y)$ is a centered Gaussian with variance {$\frac{k}{2}$} for the first term and Proposition~\ref{p:3.1} for the second.

Plugging this in~\eqref{e:3.40} we get for such $k$,
\begin{equation}
\label{e:3.42.1}
\bbP \Big( \big(\cU_n^k \cup \cB_n^k\big) \cap \cG_n \neq \emptyset \, \Big|\, L_{t_n^{\rmA}}(\ol{\bbT}_n) \Big)  \leq C'' \big( \big|[\cU_n^k]_k\big| + \big|[\cB_n^k]_k\big| \big) \tfrac{(n-k)^2}{\sqrt{k}}  \,.
\end{equation}

Therefore, whenever the events in Lemmas~\ref{l:3.9.5a} and~\ref{l:3.9.5} do not occur, by the union bound the probability that there exists $k \in [r_n, n-r]$ such that $(\cU_n^k \cup \cB_n^k) \cap \cG_n \neq \emptyset$ is at most
\begin{equation}\label{eq:5.13}
C'' \sum_{k=r_n}^{n-r} (n-k)^2 \sqrt{\frac{n}{k}} \rme^{-\wedge_n(k)^{1/2-\eta'}} 
\leq C'' \sum_{k=r_n}^{\lceil n/2 \rceil } \rme^{-\tfrac12 k^{1/2-\eta'}} 
+ C'' \sum_{k=\lceil n/2 \rceil +1}^{n-r} (n-k)^2 \rme^{-(n-k)^{1/2-\eta'}}  \,,
\end{equation}
which goes to $0$ when $n \to \infty$ followed by $r \to \infty$. Using the union bound together with Lemmas~\ref{l:3.9.5a} and~\ref{l:3.9.5} then completes the proof.
\end{proof}

\subsubsection{Up-Repelled Clusters}

We now turn to the proof of Lemma~\ref{l:3.9.5a}.

\begin{proof}[Proof of Lemma~\ref{l:3.9.5a}]
We fix $u$, $\eta$ and $\eta'$ as in the statement of the lemma and henceforth omit the dependence on these quantities from the notation. Letting $k \in [r_n, n$, in analog to~\eqref{e:5.1} we first define
\begin{equation}
\label{e:5.1a}
\cV^k_n := \Big \{ x \in \cF_n : \: 
\sqrt{L_{t_n^{\rmA}}([x]_k)} > \sqrt{\log 2}(n-k) + \wedge_n(k)^{1/2-\eta} \Big
\} \,.
\end{equation}
Since $\cU_n^k = \cV_n^k \cap \cW_n^k$, if $y \in [\cU_n^k]_k$ we must have both
$\cF_n \cap \ol{\bbT}^l(y) \neq \emptyset$ and $\cF_n \cap \ol{\bbT}^r(y) \neq \emptyset$. We can then use the Markov property to write:
\begin{equation}
\begin{split}
\label{e:3.21a}
\bbE \Big(\big|[\cU_n^k]_k\big| \,\Big|\, [\cF_n]_k \,,\,
L_{t_n^{\rmA}}(\ol{\bbT}_k) \Big) 
& = \sum_y \frac{\bbP \big( y \in [\cW_n^k]_k \,\big|\, L_{t^{\rmA}_n}(y) \big)}
{\bbP \big(y \in [\cF_n]_k \,\Big|\, L_{t^{\rmA}_n}(y) \big)} \\
& = \sum_y
\frac{\bbP \Big(\cF_n \cap \ol{\bbT}^l(y) \neq \emptyset \text{ and }
	\cF_n \cap \ol{\bbT}^r(y) \neq \emptyset \,\Big|\, L_{t_n^{\rmA}}(y) \Big)}
	{\bbP \Big(\cF_n \cap \ol{\bbT}^l(y) \neq \emptyset \text{ or }
			\cF_n \cap \ol{\bbT}^r(y) \neq \emptyset \,\Big|\, L_{t_n^{\rmA}}(y) \Big)} \\
\end{split}
\end{equation}
where the sums are over the set $[\cV_n^k \cap \cW_n^{[k, n]}]_k$. Observe that this set is measurable with respect to the random variables under the conditioning on the left hand side. 

Abbreviating $t(y) \equiv L_{t_n^{\rmA}}(y)$ and setting for any $n \geq 1$ and $t \geq 0$,
\begin{equation}
\label{e:3.50}
p_{n,t} := \bbP \big(\cF_{n,t} \neq \emptyset \big) \,,
\end{equation}
the ratio in the sum is equal to 
\begin{equation}
\label{e:5.17}
\frac{(p_{n-k, t(y)})^2}{2p_{n-k, t(y)} - (p_{n-k, t(y)})^2} \leq p_{n-k, t(y)} \,,
\end{equation}
whenever $\wedge_n(k)$ is large enough.
By definition, for any $y$ in the sum $\sqrt{t(y)} \geq \sqrt{\log 2} (n-k) + \wedge_n(k)^{1/2-\eta}$ and hence by Lemma~\ref{l:3.5} and monotonicity of $L_t$ in $t$,
\begin{equation}
\label{e:5.16}
p_{n-k, t(y)} \leq C \rme^{-C' \wedge_n(k)^{1/2-\eta}} \,.
\end{equation}

Using the above bounds and since $\big|[\cV_n^k \cap \cW_n^{[k,n]}]_k\big|
\leq \big|[\cF_n]_{r_n}\big|$, the conditional expectation in~\eqref{e:3.21a} can be bounded by
$C\rme^{-C'\wedge_n(k)^{1/2-\eta}} \big|[\cF_n]_{r_n}\big|$. Applying Markov's inequality and then taking expectation, for any $\delta>0$ we then have
\begin{equation}
\bbP \Big(\big|[\cU_n^k]_k\big| > \delta\rme^{-\wedge_n(k)^{1/2-\eta'}} \big|[\cF_n]_{r_n}\big| \Big) 
\leq C \delta^{-1} \rme^{-\tfrac12 C'\wedge_n(k)^{1/2-\eta}} \,,
\end{equation}
where we use that $\eta' > \eta$. 
The sum of the right hand side over $k \in [r_n, \lceil n/2 \rceil]$ gives a quantity which tends to $0$ with $n$, while the same sum over $k \in [\lceil n/2 \rceil, n-r]$ is at most $C \delta^{-1} \rme^{-C' r^{1/2-\eta}}$.  Altogether, by the union bound we therefore get for all $\delta > 0$,
\begin{equation}
\lim_{r \to \infty} \limsup_{n \to \infty}
\bbP \Big( \exists k \in [r_n, n-r] :\: \big| \big[\cU^k_n(u)\big]_k \big| 
> \delta \rme^{-\wedge_n(k)^{1/2-\eta'}} \big|[\cF_n]_{r_n}\big| \Big) = 0 \,.
\end{equation}
It remains to use Lemma~\ref{l:3.2} and the union bound one more time to complete the proof.
\end{proof}

\subsubsection{Big Clusters}
To prove that the number of $k$-big clusters decay with $\wedge_n(k)$, we first show that the root of most of these clusters must have an unusually low local time. Recalling the definition of $\cD_n^{k,\eta'}$ from~\eqref{e:5.12}, we then have
\begin{lem}
\label{l:5.5}
Let $0 < \eta < \eta' < 1/2$. For all $u \geq 0$,
\begin{equation}
\label{e:3.48}
\lim_{r \to \infty} \limsup_{n \to \infty}
\bbP \Big( \exists k \in [r_n, n-r] :\: 
\big| \big[\cB^{k,\eta}_n(u)\big]_k \setminus \cD^{k,\eta'}_n \big| 
	> \rme^{-\wedge_n(k)^{1/2-\eta'}} \sqrt{n}  \Big) = 0 \,.
\end{equation}
\end{lem}

\begin{proof}
We fix $u$, $\eta$ and $\eta'$ as in the statement of the lemma and henceforth omit the dependence on these parameters from the notation. Proceeding as in the proof of Lemma~\ref{l:3.9.5a}, for $k \in [r_n, n]$, we write
\begin{multline}
\label{e:3.49}
\bbE \Big(\big|\big[\cB^k_n\big]_k \setminus \cD^k_n \big| \,\Big|\, 
[\cF_n]_k \,, \, L_{t_n^{\rmA}}(\ol{\bbT}_k)\Big)  
= \sum_y \frac{\bbP \big( y \in [\cB_n^k]_k \,\big|\, L_{t^{\rmA}_n}(y) \big)}
{\bbP \big(y \in [\cF_n]_k \,\Big|\, L_{t^{\rmA}_n}(y) \big)} 
\\
= \sum_y
\frac{\bbP \Big(\cF_n \cap \ol{\bbT}^l(y) \neq \emptyset \,,\,\,
	\cF_n \cap \ol{\bbT}^r(y)  \neq \emptyset \,,\,\,  
\big| \cF_n \cap \ol{\bbT}(y) \big| > \rme^{\wedge_n(k)^{1/2-\eta}} \,\Big|\, L_{t_n^{\rmA}}(y) \Big)}
{\bbP \Big(\cF_n \cap \ol{\bbT}^l(y)  \neq \emptyset \text{ or }
	\cF_n \cap \ol{\bbT}^r(y)  \neq \emptyset
\, \Big| \, L_{t_n^{\rmA}}(y) \Big)} \,,
\end{multline}
where the sums run over $y \in \big[\cW^{[k, n]}_n\big]_k \setminus \cD_n^k$. 

Abbreviating $t(y) \equiv L_{t_n^{\rmA}}(y)$, using $p_{n,t}$ from~\eqref{e:3.50} and setting for $1 \leq k \leq n$ and $t \geq 0$,
\begin{equation}
q_{n,t,v} := \bbP \big(|\cF_{n,t}| > v \big) 
\quad , \qquad
v(n,k) := \tfrac12 \rme^{\wedge_n(k)^{1/2-\eta}} \,, 
\end{equation}
the ratio in~\eqref{e:3.49} is at most
\begin{equation}
\frac{2 p_{n-k,\, t(y)} \, q_{n-k,\,t(y), v(n,k)}}
{2p_{n-k,\,t(y)} - (p_{n-k,\,t(y)})^2} 
\leq 2 q_{n-k,\,t(y), v(n,k)} \,,
\end{equation}

Since for $y$ in the sum in~\eqref{e:3.49} we have $t(y) > \sqrt{\log 2} (n-k)  - \wedge_n(k)^{1/2-\eta'}$, it follows from Lemma~\ref{l:3.5}, monotonicity of the local time in $t$ and Markov's inequality, that 
\begin{equation}
q_{n-k,\,t(y), v(n,k)} \leq C \rme^{-\tfrac12 \wedge_n(k)^{1/2 - \eta}}  
\end{equation}
whenever $\wedge_n(k)$ is large enough. Plugging this back in~\eqref{e:3.49} and using Markov's inequality again, for any $\delta > 0$ we have
\begin{equation}
\bbP \Big(\big|\big[\cB^k_n\big]_k \setminus \cD^k_n\big| \geq 
\delta \rme^{-\wedge_n(k)^{1/2 - \eta'}}  \big|\big[\cW^{[k,n]}_n\big]_k \setminus \cD^k_n \big|
\,\Big|\, [\cF_n]_k \,,\, L_{t_n^{\rmA}}(\ol{\bbT}_k)\Big)
\leq C \delta^{-1} \rme^{-\tfrac14 \wedge_n(k)^{1/2 - \eta}} \,,
\end{equation}
as long as $\wedge_n(k)$ is large enough. Since $|[\cW^{[k,n]}_n]_k \setminus \cD^k_n | \leq |[\cF_n]_{r_n}|$, taking expectation and summing from $k=r_n$ to $k=n-r$, we get
\begin{equation}
\bbP \Big( \exists k \in [r_n, n-r] :\: 
\big| \big[\cB^k_n\big]_k \setminus \cD^k_n  \big| 
	> \delta \rme^{-\wedge_n(k)^{1/2 - \eta'}}  \big|[\cF_n]_{r_n}\big| \Big) \leq 
	C \delta^{-1} \rme^{-C' r^{1/2 - \eta}} + o(1) \,,
\end{equation}
where the $o(1)$ term tends to $0$ with $n$. Together with Lemma~\ref{l:3.2} and the union bound, this shows~\eqref{e:3.48}.
\end{proof}

The proof of Lemma~\ref{l:3.9.5} is now straightforward.
\begin{proof}[Proof of Lemma~\ref{l:3.9.5}]
Let $u \geq 0$ and $0 < \eta < \eta''' < 1/2$. Choose any $\eta', \eta'' \in (\eta, \eta''')$ and write
\begin{equation}
\Big|\big[\cB_n^{k, \eta}(u)\big]_k\Big|  = 
\Big|\big[\cB_n^{k,\eta}(u)\big]_k \setminus \cD_n^{k,\eta'} \Big| +
\Big|\big[\cB_n^{k,\eta}(u)\big]_k \cap \cD_n^{k,\eta'} \Big| 
\leq \Big|\big[\cB_n^{k,\eta}(u)\big]_k \setminus \cD_n^{k,\eta'} \Big|
+ \Big|\big[\cD_n^{k,\eta'}]_{r_n}\Big|\,.
\end{equation}
Then for all $k$ and $n$ large enough, the event
\begin{equation}
\Big\{\Big|\big[\cB_n^{k,  \eta}(u)\big]_k\Big| > \rme^{-\wedge_n(k)^{1/2-\eta'''}} \sqrt{n} \Big\}
\end{equation}
is included in
\begin{equation}
\Big\{\Big|\big[\cB_n^{k,\eta}(u)\big]_k \setminus \cD_n^{k,\eta'} \Big| 
> \rme^{-\wedge_n(k)^{1/2-\eta'}} \sqrt{n}\Big\}
\cup
\Big\{\Big|\big[\cD_n^{k,\eta'}]_{r_n}\Big| > \rme^{-\wedge_n(k)^{1/2-\eta''}} \sqrt{n} \Big\}
\end{equation}
Now using Lemma~\ref{l:5.5} for the first event, Lemma~\ref{l:4.2} for the second and the union bound, we obtain~\eqref{e:3.47} with $\eta'''$ in place of $\eta'$.
\end{proof}

\subsection{Proof of Proposition~\ref{p:3.1c}}

\begin{proof}[Proof of Proposition~\ref{p:3.1c}]
Let $\eta \in (0,1/2)$ and choose $\eta'$ and $\eta''$ such that $0 < \eta' < \eta'' < \eta$. We let also $u \geq 0$ and as usual omit the dependence on this parameter from the notation. 
Thanks to Proposition~\ref{p:3.3} and the union bound, it suffices to show
\begin{equation}
\label{e:3.14d}
\lim_{r \to \infty}
\limsup_{n \to \infty} \bbP \Big(\Big( \cW_n^{[r_n, n-r]} \setminus \big(\cU_n^{[r_n, n-r], \eta} \cup  \cB_n^{[r_n, n-r], \eta}\big)\Big) \cap \big(\cG_n \setminus \cH_n^{[r_n, n-r], \eta'}\big)  \neq \emptyset \Big) = 0 \,.
\end{equation}
For all $k \in [r_n, n]$, we first claim that if $y \in \big[(\cW^k_n \setminus \cU_n^{k, \eta} \cap (\cG_n \setminus \cH_n^{k,\eta'})\big]_k$ then we must have $|h_n(y)| \geq \wedge_n(k)^{3/4-\eta''/2}$. Otherwise, since
$y \in [\cW^k_n \setminus \cU_n^{k,\eta}]_k$, we have $\sqrt{L_{t_n^{\rmA}}(y)} \leq \sqrt{\log 2} (n-k) + \wedge_n(k)^{1/2-\eta}$ and therefore by Taylor expansion, since $\eta'' < \eta$, 
\begin{equation}
\sqrt{L_{t_n^{\rmA}}(y) + h^2_n(y)}
\leq \sqrt{\log 2}(n-k) +  O \big(\!\wedge_n\!(k)^{1/2-\eta''} \big) \,,
\end{equation}
which is a contradiction to $y \in [\cG_n \setminus \cH_n^{k, \eta'}]_k$ whenever $\wedge_n(k)$ is large enough, since $\eta' < \eta''$.

Therefore, conditional on $L_{t_n^{\rmA}}(\ol{\bbT}_n)$, the probability in~\eqref{e:3.14d} is at most
\begin{equation}
\label{e:5.33}
\sum_{k=r_n}^{n-r} \  \sum_{y \in [\cW_n^k \setminus \cB_n^k]_k}\  \sum_{x \in \cF_n \cap \bbT(y)}  
\bbP \Big(\big|h_n([x]_k)\big| \geq \wedge_n(k)^{3/4 - \eta''/2} \,,\,\,
|h_n(x)| \leq \sqrt{u} \Big) \,.
\end{equation}
As in~\eqref{e:4.14}, conditional on $h_n(x)$, the law of $h_n([x]_k)$ is that of a Gaussian with mean $h_n(x) k/n$ and variance $k(n-k)/(2n) \leq \wedge_n(k)/2$. Therefore, the probability in the last display is at most 
\begin{equation}
C \frac{\sqrt{u}}{\sqrt{n}} \exp \big(\!- \tfrac12 \wedge_n(k)^{1/2 - \eta''} \big)
\end{equation}
whenever $\wedge_n(k)$ is large enough.  
At the same time the number of summands in the inner sum is at most $\rme^{\wedge_n(k)^{1/2-\eta}}$ since $y \in [\cW_n^k \setminus \cB_n^{k,\eta}]_k$. Since $\eta' < \eta$ and using that 
$|[\cW_n^k]_k| \leq |[\cF_n]_{r_n}|$, it follows that the last sum is at most
\begin{equation}
C \frac{\sqrt{u}}{\sqrt{n}} \big|[\cF_n]_{r_n} \big| 
\sum_{k=r_n}^{n-r} \rme^{-\tfrac14 \wedge_n(k)^{1/2-\eta''}}
\leq C \frac{\sqrt{u}}{\sqrt{n}} \big|[\cF_n]_{r_n}\big| \big( \rme^{-\tfrac18 r^{1/2 - \eta''}} + o(1) \big) \,,
\end{equation}
as long as $r$ is large enough and with $o(1)$ term tending to $0$ with $n$. 

Consequently, on $\{|[\cF_n]_{r_n}| \leq \delta^{-1} \sqrt{n}\}$ for $\delta > 0$ the entire sum in~\eqref{e:5.33} tends to $0$ as $n \to \infty$ followed by $r \to \infty$.  Since Lemma~\ref{l:3.2} implies that $\bbP(|[\cF_n]_{r_n}| > \delta^{-1} \sqrt{n})$ can be made arbitrarily small by choosing $\delta$ small enough and taking $n$ large, the result follows from the union bound.
\end{proof}

\section{The IID nature of clusters}
\label{s:6}
In this section we study the i.i.d. structure of the law of the local time field restricted to clusters containing leaves with $O(1)$ local time, as well as the i.i.d. nature of the law of the DGFF restricted to such clustered sets.

\subsection{Local time clusters}

Our first task is to show that clusters containing leaves with low local time, follow an i.i.d. law, which is, most importantly, independent of $n$. To this end, we will show that if $x \in \ol{\bbL}_{n-r}$ is a common ancestor of all leaves in an $r_n$-cluster of $\cF_{n}$, then the law of its local time is insensitive to the local time of its ancestor $[x]_k$ whenever $k \ll n-r$. For a precise formulation, we define for $\eta \in (0,1/2)$, $t \geq 0$, $K \subseteq [0, n]$ and $u \geq 0$, the set
\begin{equation}
\label{e:7.1}
\cQ_{n,t}^{K}(u) = \cQ_{n,t}^{K,\eta}(u) := \Big \{ x \in \cF_{n,t}(u) : \: \forall k \in K ,\,
\sqrt{L_{t}([x]_{k})} \in \sqrt{\log 2}(n-k) + \frR_{n-k}^\eta \Big\} \,,
\end{equation}
where $\frR_{n-k}^\eta$ is as in~\eqref{e:3.10}.
(Notice the slight difference between $\cQ_{n,t}^{K,\eta}(u)$ and $\cF_{n,t}(u) \setminus \cR_{n,t}^{K,\eta}(u)$, where the latter was defined in~\eqref{e:4.7}.)

%recall the definition of $\cQ_{n,t}^{K,\eta}(u)$ from~\eqref{e:6.17} and 

Now, given $\eta \in (0,1/2)$, $t \geq 0$, $k \in [0, n]$ and $x \in \ol{\bbL}_n$, define the probability measure $\rmP_n^k(t \;; \cdot)$ on $\bbR$ via,
\begin{equation}
\label{e:6.11}
\rmP_n^k(t \;; \cdot) := \Big(L_t([x]_{k}) \in \cdot \,\Big|\, 
	x \in \cQ_{n, t}^{[0, k],\eta}(u) \,\,,\, [\cF_{n, t}(u)]_{k+1} = [x]_{k+1} \Big) \,.
\end{equation}
We remark that, although the above quantity depends on $u$ and $\eta$, we do not express this in the notation of $\rmP_n^k(t \;; \cdot)$ to avoid clutter. Observe that by definition $\rmP_n^k(t ; \cdot)$ is supported on $\{ v \geq 0 :\: \sqrt{v} \in \sqrt{\log 2}(n-k) + \frR_{n-k}^\eta\}$.

We then have,
\begin{lem}
\label{l:6.2}
If $\eta > 0$ is chosen small enough, then for all $u \geq 0$, any $n \geq 1$ and $k \in [r_n, n]$, there exists a probability measure $\mu_n^k$ on $\bbR_+$ (which depends on $u$ and $\eta$), such that for all $r \geq 1$,
\begin{equation}
\label{e:6.12}
\lim_{n \to \infty} \sup_{\sqrt{t} \in \sqrt{\log 2}n + \frR_n} \sup_{\|\varphi\|_\infty \leq 1}
\bigg|\int \varphi(v) \rmP_n^{n-r}\big(t \;; \rmd v \big)  - \int \varphi(v) \mu_n^{n-r}(\rmd v) \bigg| = 0 \,.
\end{equation} 
\end{lem} 
\begin{proof}
Fix $u \geq 0$ and $r \geq 1$. As usual we omit the dependency on $u$ for brevity. The statement of the lemma will follow if we show that
\begin{equation}
\label{e:6.13a}
\lim_{n \to \infty} \sup_{t,t'} \sup_{\|\varphi\|_\infty \leq 1}
\bigg|\int \varphi(v) \rmP_n^{n-r}\big(t \;; \rmd v \big)  - \int \varphi(v) \rmP_n^{n-r}
\big(t' \;; \rmd v \big) \bigg| = 0 \,,
\end{equation}
with the outer supremum taken over all $t,t'$ such that $\sqrt{t}, \sqrt{t'} \in \sqrt{\log 2}n + \frR_n$. Indeed, we may then set $\mu_n^k(\rmd v) := \rmP_n^k\big(t' \;;\; \rmd v \big)$ for $\sqrt{t'} = \sqrt{\log 2}n + n^{1/2}$ and claim~\eqref{e:6.12}.

The first integral in~\eqref{e:6.13a} is
\begin{equation}
\label{e:6.20}
\frac{\bbE \Big( \varphi \big(L_t([x]_{n-r})\big) \,;\,
	x \in \cQ_{n, t}^{[0, n-r]} \,\,,\, [\cF_{n, t}]_{n-r+1} = [x]_{n-r+1} \Big)}
{\bbP \Big(x \in \cQ_{n, t}^{[0, n-r]} \,\,,\, [\cF_{n, t}]_{n-r+1} = [x]_{n-r+1} \Big)} \,.
\end{equation}
Abbreviating $n' \equiv n-n^{3/4}$ and setting
\begin{equation}
p_{n}(t ;\; \rmd w) :=  \bbP \Big(L_t([x]_{n'}) \in \rmd w \,,\, x \in \cQ_{n,t}^{[0,n']} \,,\,
\cF_{n,t} \setminus \bbT([x]_{n'+1}) = \emptyset \Big) \,,
\end{equation}
we now condition on $L_t \big(\ol{\bbT} \setminus \bbT([x]_{n'+1}) \big)$ and use the Markov property to write~\eqref{e:6.20} as 
\begin{equation} 
\label{e:6.14}
\frac{\int
	\bbE \Big( \varphi\big(L_w([x]_{n^{3/4}-r})\big) \;;\; x \in \cQ_{n^{3/4}, w}^{[0, n^{3/4}-r]} \,,\,\, 
	\big[\cF_{n^{3/4}, w}\big]_{n^{3/4}-r+1} = [x]_{n^{3/4} -r+1} \Big) p_n(t \;;\; \rmd w)}
{\int
		\bbP \Big( x \in \cQ_{n^{3/4}, w}^{[0,n^{3/4}-r]} ,\, 
		\big[\cF_{n^{3/4}, w}\big]_{n^{3/4}-r+1} = [x]_{n^{3/4} -r+1} \Big) p_n(t \;;\; \rmd w)} \,,
\end{equation}
where both integrals are over $w$ such that $\sqrt{w} \in \sqrt{\log 2}\, n^{3/4} + \frR_{n^{3/4}}$.
It will therefore be sufficient to show that $p_{n}(t ;\; \rmd w)$ depends on $t$ via a multiplicative factor.

To this end, we further write $p_n(t; \rmd w)$ as
\begin{equation}
\label{e:6.15a}
\bbP \Big(L_t([x]_{n'}) \in \rmd w \,,\, x \in \cQ_{n,t}^{[0,n']} \Big) 
\bbP \Big(\cF_{n,t} \setminus \bbT([x]_{n'+1}) = \emptyset \,\Big|\, 
L_t([x]_{n'}) = w \,,\, x \in \cQ_{n,t}^{[0,n']} \Big) \,.
\end{equation}
We first claim that the conditional probability above tends to $1$ as $n \to \infty$, uniformly in $\sqrt{t} \in \sqrt{\log 2}\,n + \frR_n$. This is because under the conditioning we have $\sqrt{L_t([x]_k)} > \sqrt{\log 2} (n-k) + (n-k)^{1/2-\eta}$ for all $k=1, \dots, n'$ and consequently, the probability in question is at least $1$ minus  
\begin{equation}
\sum_{k=1}^{n'} \bbP \big(\cF_{n-k, (\sqrt{\log 2}(n-k) + (n-k)^{1/2-\eta})^2} \neq \emptyset\big)
\leq 
\sum_{k=1}^{n'} C \rme^{-C' (n-k)^{1/2-\eta}} \leq \rme^{-n^{1/4}} \,.
\end{equation}
Above we have used Lemma~\ref{l:3.5} to bound the terms in the first sum.

Turning to the first term in~\eqref{e:6.15a}, by Lemma~\ref{l:3.8} and Lemma~\ref{l:3.9} it can be written as $\bbP (B_{n'}^2 \in \rmd w\,,\,B_{n'}>0\, | B_0 = \sqrt{t})$ times
\begin{equation}
\begin{split}
\label{e:6.18a}
\Big(\frac{\sqrt{t}}{\sqrt{w}}\Big)^{1/2} \bbE \Big( 
\exp \Big(-\tfrac{3}{16} \small\int_0^{n'} B^{-2}_s \rmd s \Big) \ ; \ & 
B_k \in \sqrt{\log 2}(n-k) + \frR_{n-k} \,,\forall k \in [1, n'] \,,
\\ & B_s > 0 \,, \forall s \in [0,n'] 
\,\Big|\, B_0 = \sqrt{t} ,\, B_{n'} = \sqrt{w} \Big) \,,
\end{split}
\end{equation}
where $(B_s :\: s \in [0,n'])$ is a Brownian motion with variance $1/2$ starting from $\sqrt{t}$ at time $0$ (which we write formally as conditioning).

Since $\sqrt{t} \wedge \sqrt{w} > \tfrac12 n^{3/4}$, it follows by standard arguments (e.g. the Reflection Principle for Brownian motion) that 
\begin{equation}
\bbP \Big(\min_{s \in [0,n']} B_s < \tfrac14 n^{3/4} \,\Big|\, B_0=\sqrt{t} ,\, B_{n'} = \sqrt{\omega} \Big) \leq C\rme^{-C'\sqrt{n}} \,,
\end{equation}
uniformly in $t$ and $w$. As on the complement of this event, the integral in~\eqref{e:6.18a} is bounded by $C/\sqrt{n}$, display~\eqref{e:6.18a} is equal to
\begin{equation}\label{e:6.12abc}
n^{1/8} \bbP \Big( B_k \in \sqrt{\log 2}(n-k) + \frR_{n-k} \,,\forall k \in[1, n'] \,\Big|\, B_0 = \sqrt{t} \,, B_{n'} = \sqrt{w} \Big) \big(1+o(1)\big) + O \big (\rme^{-C \sqrt{n}} \big) \,.
\end{equation}
Tilting by $s \mapsto -\sqrt{\log 2}(n-s)$, the last probability is further equal to
\begin{equation}
\bbP \Big(B_k \in \frR_{n-k} \,, \forall k \in [1, \dots, n'] \,\Big|\, B_0 = \sqrt{\hat{t}}, B_{n'} = \sqrt{\hat{w}} \Big) \,,
\end{equation}
where $\sqrt{\hat{t}} := \sqrt{t} - \sqrt{\log 2}\, n \in \frR_n$ and $\sqrt{\hat{w}} := \sqrt{w} - \sqrt{\log 2}\,n^{3/4} \in \frR_{n^{3/4}}$. Standard Brownian motion estimates (c.f.~\cite{CHL17Supp}) then show that the last probability has the same asymptotics as that of the probability for the same Brownian motion to stay positive on $[1,n']$, namely
\begin{equation}
\frac{2\sqrt{\hat{t}}\sqrt{\hat{w}}}{n'} \big(1+o(1)\big)= \frac{2\sqrt{\hat{t}\hat{w}}}{n} \big(1+o(1)\big)  \,,
\end{equation} 
uniformly in the ranges of $t$ and $w$, so that~\eqref{e:6.12abc} is then equal to $2n^{-\tfrac 78}\sqrt{\hat{t}\hat{w}}(1+o(1))$, also uniformly in $t$ and $w$.

At the same time, if $\eta$ is chosen small enough,
\begin{equation}
\begin{split}
\bbP (B_{n'}^2 \in \rmd w \,,\,B_{n'}>0\,| B_0 = \sqrt{t}) & = \frac{1}{\sqrt{2\pi n' w}} \rme^{-\frac{(\sqrt{t} - \sqrt{w})^2}{n'}} \\
& = \frac{1}{\sqrt{2 \pi w n}} 2^{-n'}
\exp \Big( \!-2\sqrt{\log 2} \big(\sqrt{\hat{t}} - \sqrt{\hat{w}}\big) 
	-\frac{\hat{t}}{n'} \Big) (1+o(1)) \,.
\end{split}
\end{equation}

Combining all of the above we get
\begin{equation}
p_{n}(t ;\; \rmd w) = \sqrt{\frac{2}{\pi}} 2^{-n'} n^{-\frac{11}{8}} 
\, \frac{\sqrt{\hat{w}}}{\sqrt{w}}  
\rme^{2 \sqrt{\log 2} \sqrt{\hat{w}}} \,
\sqrt{\hat{t}} \rme^{-2\sqrt{\log 2}\sqrt{\hat{t}} -\frac{\hat{t}}{n'}}  \big(1+o(1)\big) \,.
\end{equation}
where the $o(1)$ goes to $0$ as $n \to \infty$, uniformly in the ranges of $t$ and $w$.
Plugging this in~\eqref{e:6.14} shows that the ratio there is asymptotically equivalent as $n \to \infty$ to a quantity which does not depend on $t$. Since the ratio is always bounded by one, this implies that the difference in~\eqref{e:6.13a} goes to $0$ as $n \to \infty$ uniformly in $t,t'$ and any $\varphi$ with $\|\varphi\|_\infty \leq 1$ and completes the proof.
\end{proof}

\subsection{DGFF clusters}
\label{ss:6.2}
Next, we address the i.i.d. structure of the DGFF clusters. In order to simplify the notation, for any $y \in \bbT$ and $k \geq 1$ we shall identify $\ol{\bbL}_k(y)$ with $\ol{\bbL}_k$, so that for any $\varphi:\bbR^{\ol{\bbL}_k} \to \bbR$ we may write $\varphi\big(h_n(\ol{\bbL}_k(y)\big)$ to denote its value on $h_n(\ol{\bbL}_k(y)\big)$ as an element of $\bbR^{\ol{\bbL}_k}$. With this in mind, we now prove,
\begin{lem}
\label{l:6.4}
For each $r \geq 1$ there exists a measure $\nu_r$ on $\bbR^{\bbL_r}$ with Radon marginals such that, for all $\lambda \geq 0$ and $M \in (0,\infty)$, 
\begin{multline}
\lim_{n \to \infty} \sup_{\cL_{n-r}} \sup_{(\varphi_{n,y})_{y \in \cL_{n-r}}} \bigg| \bbE \exp \Big(-\lambda \sum_{y \in \cL_{n-r}} \varphi_{n,y} \big(h_n(\bbL_r(y))\big) \Big)  \\
- \exp \Big(- \tfrac{1}{\sqrt{n}} \sum_{y \in \cL_{n-r}} \int_{\bbR^{\bbL_r}} \big(1-\rme^{-\lambda \varphi_{n,y}(\omega)}\big) \nu_r(\rmd \omega) \Big) \bigg| = 0 \,,
\end{multline} where the first supremum is taken over all sets $\cL_{n-r} \subseteq \ol{\bbL}_{n-r}$ such that $|y \wedge y'| < r_n$ for $y \neq y' \in \cL_{n-r}$, and the second is over all families $(\varphi_{n,y}:y \in \cL_{n-r})$ of measurable functions $\varphi_{n,y}: \bbR^{\bbL_{r}} \to \bbR_+$ satisfying  
\begin{equation}
\|\varphi_{n,y}\|_\infty \leq M 
\quad ; \quad
\supp\big(\varphi_{n,y}\big) \subseteq \big\{ \omega \in \bbR^{\bbL_r} :\:
\min_{x \in \bbL_r}	|\omega(x)| \leq M \big\} \,.
\end{equation}\end{lem}

\begin{proof}
Conditional on $h_n(\ol{\bbL}_{r_n})$, the first expectation in the statement of the lemma is equal to
\begin{equation}
\label{e:6.18}
\prod_{y \in \cL_{n-r}} \Big( 1 - \int \big(1-\rme^{-\lambda \varphi_{n,y}(\omega)}\big) \nu_{n,r}\big(h_n([y]_{r_n}) \;;\; \rmd \omega\big)\Big) \,,
\end{equation}
where the probability measure $\nu_{n,r}(v \;;\; \rmd \omega)$ above is given by
\begin{equation}
\nu_{n,r}(v \;;\; \rmd \omega) = \bbP \big(W_v + h_r(\bbL_{r}) \in \rmd \omega \big) \,,
\end{equation}
where $h_r$ is a DGFF on $\bbT_r$ and $W_v$ is scalar Gaussian with mean $v$ and variance $(n-r_n-r)/2$, independent of $h_r$.

Now for any $\psi$ which obeys the same bounds as $\varphi_{n,y}$, we have
\begin{multline}
\Big| \int \psi(\omega) \nu_{n,r}(v \;; \rmd \omega) -
\int_{|u| \leq r_n \log r_n} \int \psi(u+\omega) \bbP(W_v \in \rmd u) \bbP(h_r(\bbL_{r}) \in \rmd \omega) \Big| \\
\leq M \sum_{x \in \bbL_{r}} \bbP(|h_r(x)| > r_n \log r_n - M\big) \,,
\end{multline}
Since $h_r(x)$ is Gaussian with mean $0$ and variance $r/2$, the right hand side above is at most $C \rme^{-C' (r_n\log r_n)^2} = o(n^{-1/2})$ for all $n$ large enough, for some $C, C' > 0$ depending on $r$ and $M$. Similarly, we have that $\bbP(W_v \in \rmd u)/\rmd u = \pi^{-1/2} n^{-1/2} (1+o(1))$ uniformly in $|u|,|v| \leq r_n \log r_n$. Since also
\begin{multline}
\Bigg|\int_{|u| > r_n \log r_n} \int \psi(u+\omega) \bbP(h_r(\bbL_{r}) \in \rmd \omega) \rmd u  \Bigg|
\\ \leq M \sum_{x \in \bbL_{r}} \int_{|u| > r_n \log r_n} \bbP(|h_r(x)| > |u| - M\big) \rmd u= o(n^{-1/2}) \,,
\end{multline}
if we set,
\begin{equation}
\nu_r(\rmd \omega) := \tfrac{1}{\sqrt{\pi}} \int \bbP\big(h_r(\bbL_{r}) + u \in \rmd \omega\big) \rmd u \,,
\end{equation}
then
\begin{equation}
\label{e:6.24a}
\int \psi(\omega) \nu_{n,r}(v \;; \rmd \omega) = \big(1+o(1)\big) \tfrac{1}{\sqrt{n}} \int \psi(\omega) \nu_{r}(\rmd \omega) 
\quad \text{as } n \to \infty \,,
\end{equation}
uniformly in $|v|<r_n \log r_n$ and all measurable functions $\psi$ satisfying the same bounds as $\varphi_{n,y}$. 
Moreover, since $h_r(x)$ is Gaussian with mean $0$ and variance $r/2$ for each $x \in \bbL_r$, it is easy to check that all marginals of $\nu_r$ are a constant multiple of the Lebesgue measure and therefore Radon as desired. In particular, the second integral in~\eqref{e:6.24a} is bounded by $M \sum_{x \in \bbL_r} \nu_r^{(x)}([-M, M]) := C(M,r) < \infty$ for all such functions $\psi$, where $\nu^{(x)}_r$ denotes the $x$-th marginal of $\nu_r$.

Since $1-\rme^{-\lambda \varphi_{n,y}}$ has the same support as $\varphi_{n,y}$ and $\|1-\rme^{-\lambda \varphi_{n,y}}\|_\infty \leq \lambda M$, it follows that on the event
\begin{equation}
\label{e:6.24}
\Big\{ \max_{y \in \cL_{n-r}} |h_n([y]_{r_n})| \leq r_n \log r_n \Big\} \,,
\end{equation}
the product in~\eqref{e:6.18} is equal to
\begin{multline}
\exp \Big( \sum_{y \in \cL_{n-r}} \log 
\Big( 1 - \tfrac{1+o(1)}{\sqrt{n}}\int \big(1-\rme^{-\lambda \varphi_{n,y}(\omega)}\big) \nu_r(\rmd \omega) \Big)\Big) \\
=
\exp \Big( -\tfrac{1}{\sqrt{n}} \sum_{y \in \cL_{n-r}} \int \big(1-\rme^{-\lambda \varphi_{n,y}(\omega)}\big) \nu_r(\rmd \omega) \Big)\Big) + o(1) \,.
\end{multline}
Above we have used Taylor expansion for the log together with the fact that the first integral is uniformly bounded by $C(\lambda M,r)$, as well as the inequality $|\rme^{-\alpha}-\rme^{-(1+\epsilon)\alpha}| \leq 2\rme^{-1}|\epsilon|$ valid for all $\alpha \geq 0$ and $|\epsilon| \leq \tfrac{1}{2}$. Since the probability of~\eqref{e:6.24} goes to $1$ as $n \to \infty$, thanks to Proposition~\ref{p:3.1}, the lemma now follows by taking expectation.
\end{proof}

\section{Phase A: Proof of Theorem~\ref{t:2a}}
\label{s:7}
In this section we prove Theorem~\ref{t:2a}, which is the key result for phase $A$.
Thanks to Theorem~\ref{t:3.1}, we can restrict our attention to $\cW_n^{[n-r, n]}$ instead of $\cF_n(u)$, namely to those leaves which belong to $r_n$-clusters whose root is essentially at a finite distance from the leaves. This is because the remaining leaves do not survive the isomorphism. For the sake of proving Theorem~\ref{t:2a}, it will be useful to take into consideration only such clusters where, in addition, the local time trajectory of their root is properly repelled.

To this end, recall the definition of $\cQ_{n,t}^{K,\eta}(u)$ from~\eqref{e:7.1} and for $r \in [0,n/2]$ let us set,
\begin{equation}
\cE_{n,t}^{r, \eta}(u) := \cW_{n,t}^{[n-r, n]}(u) \cap \cQ_{n,t}^{[n/2, n-r], \eta}(u)  \,.
\end{equation}
As usual, whenever $t$ is omitted from the notation, its value is assumed to be $t_n^{\rmA}$. Our first task is therefore to extend Theorem~\ref{t:3.1} to the set $\cE_n^{r,\eta}(u)$.
\begin{prop}
\label{p:6.1}
Let $\eta \in (0,1/2)$. For any $u \geq 0$,
\begin{equation}
\label{e:5.3a}
\lim_{r \to \infty}
\limsup_{n \to \infty} \bbP \Big( \cG_n(u) \setminus \cE_n^{r,\eta}(u) \neq \emptyset \Big) = 0 \,.
\end{equation}
and, in addition, for any $\delta > 0$, 
\begin{equation}
\label{e:5.3}
\lim_{r \to \infty}
\limsup_{n \to \infty} \bbP \Big( \big| \big[\cF_n(u) \setminus \cE_n^{r,\eta}(u) \big]_{r_n} \big| > \delta \sqrt{n}  \Big) = 0 \,.
\end{equation}
\end{prop}

Next, we will need the following result to assert the non-triviality of the limit in Theorem~\ref{t:2a}. This result will also be used in the proofs of Propositions~\ref{p:6.3} and Proposition~\ref{p:6.5} below.
\begin{prop}
\label{p:6.2}
Let $\eta \in (0,1/2)$. For any $u \geq 0$,
\begin{equation}
\label{e:6.5a}
\lim_{r \to \infty}
\lim_{\delta \to 0} 
\limsup_{n \to \infty} \bbP \Big( \big|\big[\cE_n^{r,\eta}(u) \cap \cF_n(0) \big]_{n-r}\big| \leq \delta \sqrt{n} \Big) = 0 \,. 
\end{equation}
\end{prop}

The main two ingredients in the proof of Theorem~\ref{t:2a} are the following two propositions:
\begin{prop}
\label{p:6.3}
Fix $\eta > 0$ to be small enough and let $u \geq 0$. For all $r \geq 1$, there exists $C_{A_1}^{r}  = C_{A_1}^{r}(u) \in [0,\infty)$ such that for any $\delta > 0$,
\begin{equation}
\label{e:6.13}
\lim_{r \to \infty}
\limsup_{n \to \infty}
\bbP \Bigg(\Bigg|
\frac{\big|\big[\cE_n^{r,\eta}(u) \cap \cF_n(0)\big]_{n-r}\big|}
{\big|\big[\cE_n^{r,\eta}(u) \big]_{n-r} \big|} - C_{A_1}^{r} \Bigg| > \delta \Bigg) = 0 \,.
\end{equation}
\end{prop}

\begin{prop}
\label{p:6.5}
Fix $\eta > 0$ to be small enough and let $u \geq 0$. For all $r \geq 1$, there exists $C_{A_2}^{r}  = C_{A_2}^{r}(u) \in [0,\infty)$ such that for all $\lambda \geq 0$ 
\begin{equation}
\label{e:6.17}
\lim_{r \to \infty}
\limsup_{n \to \infty}
\Big| \bbE \exp \Big(-\lambda 
\big|\big[\cE_n^{r,\eta}(u) \cap \cG_n(u) \big]_{n-r} \big| \Big) 
- \bbE \exp \Big(-C_{A_2}^{r} \tfrac{1}{\sqrt{n}} \big|\big[\cE_n^{r,\eta}(u)\big]_{n-r}\big| \big(1-\rme^{-\lambda}\big) \Big) \Big| = 0 \,.
\end{equation}
\end{prop}

Let us first prove the theorem.
\begin{proof}[Proof of Theorem~\ref{t:2a}]
We fix $\eta > 0$ small enough and an arbitrary $u > 0$, omitting as usual the dependency on these parameters in the sequel. Thanks to the first part of Proposition~\ref{p:6.1}, 
\begin{equation}
\big|\big[\cG_n\big]_{n-r}\big| -  
\big|\big[\cE_n^r \cap \cG_n \big]_{n-r} \big| \lto 0 
\end{equation}
in probability as $n \to \infty$ followed by $r \to \infty$. 
Using that $|\rme^{-\alpha} - \rme^{-\beta}| \leq |\alpha-\beta|$ for $\alpha, \beta \geq 0$ and the bounded convergence theorem, it then follows that for any $\lambda \geq 0$, under the same limits we have,
\begin{equation}
\label{e:6.8}
\bbE \exp \Big(-\lambda 
\big|\big[\cE_n^r \cap \cG_n \big]_{n-r} \big| \Big) 
- \bbE \exp \Big(-\lambda \big|\big[\cG_n\big]_{n-r}\big| \Big)
\lto 0 \,.
\end{equation}
On the other hand, thanks to the second part of Proposition~\ref{p:6.1} and the fact that $\cE^r_n$ is $(r_n,n-r)$-clustered by definition, under the same limits we have
\begin{equation}
\label{e:6.53}
\tfrac{1}{\sqrt{n}} \Big(\big|\big[\cF_n\big]_{r_n}\big| - 
\big|\big[\cE_n^r\big]_{n-r}\big| \Big) \lto 0 \,,
\end{equation}
in probability, which implies as before that for any $\lambda \geq 0$, and $C_{A_2}^r \geq 0$,
\begin{equation}
\label{e:6.10}
\bbE \exp \Big(-C_{A_2}^{r} \tfrac{1}{\sqrt{n}} \big|\big[\cE_n^r\big]_{n-r}\big| \big(1-\rme^{-\lambda}\big) \Big) - 
\bbE \exp \Big(-C_{A_2}^{r} \tfrac{1}{\sqrt{n}} \big|\big[\cF_n\big]_{r_n}\big| \big(1-\rme^{-\lambda}\big) \Big) \lto 0 \,,
\end{equation}
under the same limits.

Using now Proposition~\ref{p:6.5} and Proposition~\ref{l:3.5a} for the first, resp. second expectation in~\eqref{e:6.8} and combining with~\eqref{e:6.10}, we get 
\begin{equation}
\label{e:6.56}
\lim_{r \to \infty} \limsup_{n \to \infty} \Big|
\bbE \exp \Big(-C_{A_2}^{r} \tfrac{1}{\sqrt{n}} \big|\big[\cF_n\big]_{r_n}\big| \big(1-\rme^{-\lambda}\big) \Big) -
\bbE \exp \big(-C_u \bar{Z} (1-\rme^{-\lambda}) \big) \Big| = 0 \,.
\end{equation}
Thanks to Lemma~\ref{l:3.2}, we know that the sequence $(\tfrac{1}{\sqrt{n}} \big|\big[\cF_n\big]_{r_n}\big| : n \in \bbN)$ is tight and hence admits a subsequential weak limit, which we can denote temporarily by $|\wh{\cF}|$. Taking a limit along this subsequence, for any $r \geq 0$, the first expectation therefore converges to $\bbE \exp \big(-C_{A_2}^{r} |\wh{\cF}| (1-\rme^{-\lambda} \big)\big)$. Plugging this back in~\eqref{e:6.56}, we obtain
\begin{equation}
\lim_{r \to \infty} 
\bbE \exp \Big(-C_{A_2}^{r} |\wh{\cF}| \big(1-\rme^{-\lambda}\big) \Big) =
\bbE \exp \big(-C_u \bar{Z} (1-\rme^{-\lambda}) \big) \,.
\end{equation}
Since this holds for all $\lambda \geq 0$, it follows that $C_{A_2}^r |\wh{\cF}| \wto C_u \bar{Z}$ as $r \to \infty$. Since $\bar{Z}$ is positive and finite with positive probability and $C_u > 0$, we must have that $C_{A_2}^r \to C_{A_2}$ as $r \to \infty$ for some $C_{A_2} \in (0,\infty)$ and that $|\wh{\cF}| = C_u C_{A_2}^{-1} \bar{Z}$ in law. Since this holds for any choice of a sub-sequence, we get
that as $n \to \infty$,
\begin{equation}
\label{e:6.15b}
\tfrac{1}{\sqrt{n}} \big|\big[\cF_n\big]_{r_n}\big| 
\wto C_u C_{A_2}^{-1} \bar{Z} \,.
\end{equation}

Since $\bar{Z}$ is positive almost surely the above convergence together with~\eqref{e:6.53} yields \begin{equation}\label{e:6.14b}
\frac{\big|\big[\cF_n\big]_{r_n}\big|}{\big|\big[\cE_n^r\big]_{n-r}\big|} \lto 1
\quad, \qquad
\frac{\big|\big[\cF_n(0)\big]_{r_n}\big| - \big|\big[\cE_n^r \cap \cF_n(0)\big]_{n-r}\big|}
{\big|\big[\cE_n^r \big]_{n-r} \big|}	
\lto 0\,,
\end{equation}
both in probability as $n \to \infty$ followed by $r \to \infty$, where we have also used that the numerator in the second ratio is bounded by $|[\cF_n]_{r_n}| - |[\cE_n^r]_{n-r}|$. Using this together with Proposition~\ref{p:6.3} then gives
\begin{equation}\label{e:7.15}
\frac{\tfrac{1}{\sqrt{n}}\big|\big[\cF_n(0)\big]_{r_n}\big|}
{\tfrac{1}{\sqrt{n}} \big|\big[\cF_n\big]_{r_n}\big|} - C_{A_1}^{r} \lto 0 \,,
\end{equation}
in probability under the same limits. In particular, the sequence $r \mapsto C_{A_1}^r$ is Cauchy and hence must converge to a limit, which we denote by $C_{A_1} \in [0,\infty)$. 

Taking this into consideration, using~\eqref{e:7.15} together with~\eqref{e:6.15b} then gives that, as $n \to \infty$,
\begin{equation}\label{e:6.17b}
\tfrac{1}{\sqrt{n}}\big|\big[\cF_n(0)\big]_{r_n}\big|
\Longrightarrow C_{A_1} C_{A_2}^{-1} C_u \bar{Z}\,.
\end{equation} 
Setting $C_A$ to be the product of the constants on the right hand side, noting that thanks to Proposition~\ref{p:6.2} it cannot be zero, this completes the proof of the first part of the theorem.

Turning to the second part of the theorem, we need to prove that for any $\delta > 0$,
\begin{equation} \label{e:c1}
\lim_{n \to \infty} \bbP\Big( |\cW^{[r_n,n-r_n]}_{n,t^{\rmA}_n}(0)| > \delta \sqrt{n}\Big)=0\,.
\end{equation}
To this end, we use the soft entropic repulsion of local time trajectories of non-visited leaves, as stated in Proposition~\ref{p:8.3}. Accordingly, we fix $\eta' > 0$ small enough for the condition in the proposition to hold, and notice that by the union bound and Markov's inequality, we can bound the probability in~\eqref{e:c1}
from above by
\begin{equation}
\label{e:d1}
\bbP\bigg( |\cF_{n}\setminus \mathcal{O}^{[r_n,n-r_n], \eta'}_{n}| > \frac{\delta}{2} \sqrt{n}\bigg) + \frac{2}{\delta \sqrt{n}}\bbE \Big|\cW^{[r_n,n-r_n]}_{n,t^{\rmA}_n}(0)\cap \mathcal{O}^{[r_n,n-r_n], \eta'}_{n}\Big|\,.
\end{equation} Since the first term in~\eqref{e:d1} tends to zero as $n \to \infty$ by Proposition~\ref{p:8.3}, in order to obtain~\eqref{e:c1} it will suffice to show that the second one does as well. 

To this end, observe that any $x \in \cW^{[r_n,n-r_n]}_{n,t^{\rmA}_n}(0)$ belongs to some $r_n$-cluster with root $[x]_{k(x)}$ having depth $k(x) \in [r_n,n-r_n]$. Furthermore, that the root has depth exactly $k(x)$ means that, if $\ol{\bbT}^{x}([x]_k)$ denotes the sub-tree rooted at $[x]_{k(x)}$ not containing $x$, the intersection $\cF_n \cap \ol{\bbT}^x([x]_{k(x)})$ cannot be empty. Finally, if $x$ also belongs to $\mathcal{O}^{[r_n,n-r_n], \eta'}_{n}$ then we must have 
\begin{equation}
\sqrt{L_{t^{\rmA}_n}([x]_{k(x)})} \geq \sqrt{\log 2}(n-k(x)) + n^{\eta'}\,.
\end{equation}
Therefore, if we set $\sqrt{s_n(k)}:=\sqrt{\log 2}(n-k) + n^{\eta'}$, by gathering all these facts we can bound the expectation in~\eqref{e:d1} from above by 
\begin{equation}\label{e:8.7}
\sum_{x \in \ol{\bbL}_{n}} \sum_{k=r_n}^{n-r_n} \bbP (x \in \cF_n)\bbP\Big(\cF_n \cap \ol{\bbT}^x([x]_{k})\neq \emptyset \,\Big|\, L_{t^{\rmA}_n}([x]_k) \geq s_n(k)\,,\,x \in \cF_n\Big)\,.  
\end{equation} Furthermore, since the local time fields $L_{t^{\rmA}_n}(\ol{\bbT}^l_{n-k}([x]_k))$ and $L_{t^{\rmA}_n}(\ol{\bbT}^r_{n-k}([x]_k))$ are independent given $L_{t^{\rmA}_n}([x]_k)$, one can further bound~\eqref{e:8.7} by 
\begin{equation}\label{e:8.8}
\sum_{x \in \ol{\bbL}_{n}} \sum_{k=r_n}^{n-r_n} \bbP (x \in \cF_n)\bbP\big(\cF_{n-k,s_n(k)}\neq \emptyset\big)\,.  
\end{equation} A straightforward computation using Lemma~\ref{l:3.5} then shows that, for all $n$ sufficiently large, ~\eqref{e:8.8} is at most
\begin{equation} C n \rme^{-\sqrt{\log 2} n^{\eta'}} \bbE|\cF_n| \leq C^2 n^{\frac{7}{4}} \rme^{-\sqrt{\log 2} n^{\eta'}}
\end{equation} for some constant $C > 0$, which implies that the second term in~\eqref{e:d1} also tends to $0$ as $n \to \infty$, thus yielding~\eqref{e:c1}.
\end{proof}

\subsection{Proof of Proposition~\ref{p:6.1} and Proposition~\ref{p:6.2}}
\begin{proof}[Proof of Proposition~\ref{p:6.1}]
Fixing $u, \eta$ as in the statement of the proposition, we henceforth omit them from the notation as usual. Since $\cF_n \setminus \cE_n^r \subseteq \cR_n^{[r_n, n-r]} \cup \cW_n^{[r_n, n-r)}$, thanks to Lemma~\ref{l:3.3}, the second part of Theorem~\ref{t:3.1} and the union bound, we readily 
have the second claim in the proposition.

Turning to the first, for any $n/2 > r > r' \geq 1$, we can write,
\begin{equation}
\label{e:6.5}
\cG_n \setminus \cE_n^r \, \subseteq \, \big(\cG_n \cap \cW_n^{[r_n, n-r')}\big)
\cup \big(\cG_n \cap \cR^{[r_n, n-r]}_n \cap \cW_n^{[n-r', n]}\big) \,.
\end{equation}
For the second set on the right hand side, conditioning on $L_{t^{\rmA}_n}(\ol{\bbT}_n)$, we can write
\begin{equation}
\bbE \Big(\big|\big[\cG_n \cap \cR^{[r_n, n-r]}_n \cap \cW_n^{[n-r', n]}\big]_{r_n}\big|
\, \Big| \, L_{t^{\rmA}_n}(\ol{\bbT}_n)\Big) 
\leq \sum_y \sum_{x \in \bbT(y) \cap \cF_n} \bbP \big(|h_n(x)| \leq \sqrt{u} \big)  
\,,
\end{equation}
where $y \in [\cR^{[r_n, n-r]}_n \cap \cW_n^{[n-r', n]}]_{n-r'}$ in the outer sum.

By definition, there are at most $2^{r'}$ terms in the inner sum and since $h_n(x)$ is Gaussian with variance $n/2$, each of the terms is at most $C \sqrt{u}/\sqrt{n}$. It follows that the right hand side above is bounded by $C_{r',u} |[\cR_n^{[r_n, n-r]}]_{r_n}| / \sqrt{n}$ for some $C_{r', u} > 0$.
Therefore, for any $\delta > 0$,  on the event that $|[\cR_n^{[r_n, n-r]}]_{r_n}| \leq \delta \sqrt{n}$, by Markov's inequality we shall have
\begin{equation}
\bbP \Big(\big[\cG_n \cap \cR^{[r_n, n-r]}_n \cap \cW_n^{[n-r', n]}\big]_{r_n} \neq \emptyset \, \Big|\, L_{t^{\rmA}_n}(\ol{\bbT}_n) \Big) \leq C_{r', u} \delta \,.
\end{equation}
Taking expectation and using the union bound, the probability that the second set on the right hand side of~\eqref{e:6.5} is not empty is at most
\begin{equation}
\label{e:6.48}
C_{r', u} \delta + \bbP \big(\big|\big[\cR_n^{[r_n, n-r]}\big]_{r_n}\big| > \delta \sqrt{n} \big) \,.
\end{equation}
But then, thanks to Lemma~\ref{l:3.3}, for any $r'$ the above will go to zero as $n \to \infty$ followed by $r \to \infty$ and then $\delta \to 0$. 

At the same time, by Theorem~\ref{t:3.1} the probability that the first set on the right hand side of~\eqref{e:6.5} is not empty tends to $0$ when $n \to \infty$ followed by $r' \to \infty$.
Combining the two and using the union bound, the result follows.
\end{proof}

\begin{proof}[Proof of Proposition~\ref{p:6.2}]
Fixing $\eta$ as in the statement of the proposition, we suppress the dependency on this parameter in the sequel. For any $u > 0$, $r \in [0, n/2]$ the event $\{\cG_n(u) \neq \emptyset\}$ is always included in 
\begin{equation}
\label{e:6.6}
\big\{ \cG_n(u) \setminus \cE_n^r(u) \neq \emptyset \big\}
\cup
\big\{ \big[ \cG_n(u) \cap \cE_n^r(u) \big]_{n-r} \neq \emptyset \big\} \,.
\end{equation}
As in the proof of Proposition~\ref{p:6.1} (display~\eqref{e:6.48}), the probability of the second event is at most 
\begin{equation}
C_{r, u} \delta + \bbP \big(\big|\big[\cE_n^r(u)\big]_{n-r}\big| > \delta \sqrt{n} \big) \,,
\end{equation}
for some $C_{r,u} > 0$ and any $\delta > 0$. 

At the same time, by Proposition~\ref{p:6.1}, the probability of the first event in~\eqref{e:6.6} can be made arbitrarily small by choosing $r$ large enough and then $n$ large enough. It follows that
\begin{equation}
\limsup_{r \to \infty} \lim_{\delta \to 0} \limsup_{n \to \infty}
\Big( \bbP \big(\cG_n(u) \neq \emptyset\big) - 
\bbP \big(\big|\big[\cE_n^r(u)\big]_{n-r}\big| > \delta \sqrt{n} \big) \Big)
\leq 0 \,.
\end{equation}
Thanks to the second part of Proposition~\ref{p:3.2} the first probability aboves tends to $1$ as $n \to \infty$ followed by $u \to \infty$. Combined with the above statement, this gives
\begin{equation}
\lim_{u \to \infty}
\limsup_{r \to \infty}
\lim_{\delta \to 0} 
\limsup_{n \to \infty} \bbP \Big( \big|\big[\cE_n^{r}(u) \big]_{n-r}\big| \leq \delta \sqrt{n} \Big) = 0 \,,
\end{equation}
but then also,
\begin{equation}
\label{e:6.10a}
\lim_{u \to \infty}
\lim_{\delta \to 0} 
\limsup_{n \to \infty} \bbP \Big( \big|\big[\cF_n(u) \big]_{r_n}\big| \leq \delta \sqrt{n} \Big) = 0 \,,
\end{equation}

Now, by the Markov property of local time, for any $u > 0$ and conditional on $L_{t^{\rmA}_{n-1}}(\bbT_{n-1})$, the distribution of $|[\cF_{n, t^{\rmA}_{n-1}}(0)]_{r_n}|$ dominates a Binomial distribution with $|[\cF_{n-1, t^{\rmA}_{n-1}}(u)]_{r_{n-1}}|$ trials, each having success probability $2\rme^{-u} - \rme^{-2u}$. Therefore, by Chebyshev's inequality, 
on $\{|[\cF_{n-1, t^{\rmA}_{n-1}}(u)]_{r_{n-1}}| > \delta \rme^u \sqrt{n-1}\}$ the conditional probability,
\begin{equation}
\bbP \Big(\big[\big|\cF_{n,t^{\rmA}_{n-1}}(0)\big]_{r_n}\big| \leq \delta \sqrt{n} ,\Big|\, L_{t^{\rmA}_n}(\ol{\bbT}_{n-1}) \Big) 
\end{equation}
goes to $0$ as $n \to \infty$ uniformly in $L_{t^{\rmA}_{n-1}}(\ol{\bbT}_{n-1})$. 
Taking expectation and using the union bound, in light of~\eqref{e:6.10a} we get
\begin{equation}
\label{e:6.12ab}
\lim_{\delta \to 0} \limsup_{n \to \infty}
\bbP \Big(\big|\big[\cF_{n,t^{\rmA}_{n-1}}(0)\big]_{r_n}\big| \leq \delta \sqrt{n} \Big) = 0 \,.
\end{equation}

To replace $t_{n-1}^{\rmA}$ by $t_n^{\rmA}$ in~\eqref{e:6.12ab}, we observe that $\Delta_n = t_n^{\rmA} - t_{n-1}^{\rmA} \leq 2n$ for all $n$ large enough. Consequently it is enough to argue that within such additional local time, a uniformly positive fraction of the clusters of $[\cF_{n,t^{\rmA}_{n-1}}(0)]_{r_n}$ will not be entirely visited with probability tending to $1$ with $n$. To this end, we first observe that the probability that $x \in \ol{\bbL}_n$ is not visited within $\Delta_n$ time is equal to the probability that a Poisson random variable with rate $\Delta_n/n \leq 2$ is equal to zero. This probability is at least $\rme^{-2}$. To see how this computation implies that at least a $\rme^{-6}$ fraction of the clusters in $[\cF_{n,t^{\rmA}_{n-1}}(0)]_{r_n}$ are not entirely visited with high probability, we proceed as in many of the proofs in Section~\ref{s:upper}. We pick one leaf from each of the clusters, condition on the local time field (for a random walk run up to time $\Delta_n$) on $\ol{\bbT}_{r_n}$ and exclude the event that $\max_{x \in \ol{\bbL}_{r_n}} L_{\Delta_n}(x) > 2\Delta_n$ as in Lemma~\ref{l:3.4a}. Then the events of no-visit become (conditionally) independent and occur with probability at least $\rme^{-2\Delta_n/(n-r_n)} \geq \rme^{-5}$ (this bound can be obtained using the same argument leading to the bound $\rme^{-2}$ from before). Chebyshev's inequality then completes the argument. We omit the details as this argument was used many times in the past.

Turning to the statement of the proposition, for any $u \geq 0$ and $\delta > 0$, the probability in~\eqref{e:6.5a} is bounded above by
\begin{equation}
\bbP \Big(\big|\big[\cF_{n,t^{\rmA}_n}(0)\big]_{r_n}\big| \leq 2\delta  \sqrt{n} \Big) 
+ 
\bbP \Big(\big|\big[\cF_{n,t^{\rmA}_n}(u) \setminus \cE_n^r(u) \big]_{r_n}\big| > \delta  \sqrt{n} \Big) \,.
\end{equation} 
Using~\eqref{e:6.12ab} for the first term (with $t^{\rmA}_n$ instead of $t^{\rmA}_{n-1}$) and Proposition~\ref{p:6.1} for the second one then completes the proof.
\end{proof}

\subsection{Proof of Proposition~\ref{p:6.3} and Proposition~\ref{p:6.5}}
Using the results in the previous subsection, we can now give proofs for Proposition~\ref{p:6.3} and Proposition~\ref{p:6.5}.

\begin{proof}[Proof of Proposition~\ref{p:6.3}]
We fix $u \geq 0$ and let $\eta > 0$ to be as small as needed for Lemma~\ref{l:6.2} to hold, omitting the dependence on both parameters from the notation henceforth. We let also $r \geq 1$ and for any $l \in (r, n/2)$, observe that conditional on $L_{t^\rmA_n}(\ol{\bbT}_{n-l})$ and $\big[\cE_n^{r}\big]_{n-r}$, the clusters
\begin{equation}
\Big\{ L_t(\bbT(y)) :\: y \in [\cE_n^{r}\big]_{n-r-1} \Big\}
\end{equation}
are independent, with the law of $L_{t^\rmA_n}(y)$ given by $P_l^{l-r-1}(L_{t^\rmA_n}([y]_{n-l}) ; \cdot)$ from~\eqref{e:6.11}. In particular, the events $\{\cF_n(0) \cap \bbT(y) \neq \emptyset\}$ for $y$ as above are conditionally independent and each occur with probability
\begin{equation}
\label{e:6.15}
\int \bbP \big(\cF_{r+1, v}(0) \neq \emptyset \,\big|\, \cF_{r+1,v}(u) \neq \emptyset\big) 
\rmP_l^{l-r-1}(L_{t^\rmA_n}([y]_{n-l}) \;; \rmd v) \,.
\end{equation}

Since $[y]_{n-l} \in [\cQ_n^{[n/2, n-r]}\big]_{n-l}$, we must have $\sqrt{L_{t^\rmA_n}([y]_{n-l})} \in \sqrt{\log 2} \,l + \frR_l$. Then, thanks to Lemma~\ref{l:6.2}, the difference between the integral in~\eqref{e:6.15} and
\begin{equation}
C_{A_1}^{r,l} := \int \bbP \big(\cF_{r+1, v}(0) \neq \emptyset \,\big|\, \cF_{r+1,v}(u) \neq \emptyset\big) 
\mu_l^{l-r-1} (\rmd v) \,.
\end{equation}
can be made smaller than $\delta / 2$, for any $\delta > 0$, uniformly in $n$ by choosing $l$ large enough. It follows that under the conditioning, the law of 
$\big|\big[\cE_n^{r} \cap \cF_n(0)\big]_{n-r}$
is stochastically smaller, resp. larger than that of a Binomial with $|[\cE_n^{r}\big]_{n-r}|$ trials and probability $(C_{A_1}^{r,l} + \delta/2) \wedge 1$, resp. $(C_{A_1}^{r,l} - \delta/2) \vee 0$ for success. It follows by standard arguments (e.g. Chebyshev's inequality), that for any arbitrary fixed $\delta' > 0$, on
\begin{equation}
\label{e:6.37}
\Big\{\big|\big[\cE_n^{r}\big]_{n-r}\big| > \delta' \sqrt{n} \Big\} \,,
\end{equation}
the conditional probability of
\begin{equation}
\label{e:6.38}
\Bigg\{\Bigg|
\frac{\big|\big[\cE_n^{r}(u) \cap \cF_n(0)\big]_{n-r}\big|}
{\big|\big[\cE_n^{r}(u) \big]_{n-r} \big|} - C_{A_1}^{r,l} \Bigg| > \delta \Bigg\}  \,,
\end{equation}
will go to $0$ as $n \to \infty$, uniformly in the random variables on which we condition.

Taking expectation to get rid of the conditioning, using Proposition~\ref{p:6.2} to bound the probability of the complement of the event in~\eqref{e:6.37} this shows that by choosing $r$ large enough and then $\delta'$ small enough, for any $\delta > 0$ there exists $l_0 > 0$ such that whenever $l > l_0$ and $n$ is chosen large enough, the (unconditional) probability of~\eqref{e:6.38} is less than $1/2$.
Fixing any such $r$ and using the union bound, this shows that for all $\delta > 0$, there exists $l_0$ such that if $l \wedge l' > l_0$ and $n$ is large enough, the ratio above is $\delta$-close to both $C_{A_1}^{r,l}$ and $C_{A_1}^{r,l'}$ with positive probability. This can only happen if
$\big|C_{A_1}^{r,l} - C_{A_1}^{r,l'}\big| < 2\delta$. Therefore the sequence $l \mapsto C_{A_1}^{r, l}$ is Cauchy and hence must converge to a finite limit, which we denote by $C_{A_1}^{r}$. 

Going back to~\eqref{e:6.38}, we can now assert that for any $\delta > 0$, if $l$ is chosen large enough, then on~\eqref{e:6.37}, the probability of~\eqref{e:6.38} with $C_{A_1}^{r,l}$ replaced by $C_{A_1}^{r}$ will also go to $0$ as $n \to \infty$ uniformly as before. Rerunning the previous argument, this shows 
\begin{equation}
\lim_{r \to \infty} \limsup_{n \to \infty}
\bbP \Bigg(\Bigg|
\frac{\big|\big[\cE_n^{r}(u) \cap \cF_n(0)\big]_{n-r}\big|}
{\big|\big[\cE_n^{r}(u) \big]_{n-r} \big|} - C_{A_1}^{r} \Bigg| > \delta \Bigg) = 0  \,,
\end{equation}
which is what we wanted to prove.
\end{proof}

\begin{proof}[Proof of Proposition~\ref{p:6.5}]
Let $u \geq 0$ and $\lambda \geq 0$. Omitting the dependence on $u$ as usual, for any $1 \leq r \leq n$, we write the first exponent in~\eqref{e:6.17} as
\begin{equation}
-\lambda \sum_{y \in [\cE_n^r]_{n-r}} 1_{A_{r,t_n^{\rmA}}(y)}\big(h_n(\bbL_r(y)) \big) \,,
\end{equation}
with 
\begin{equation}
A_{r,t}(y) := \Big\{ \omega \in \bbR^{\bbL_r(y)} :\: \exists x \in \bbL_r(y) :\: L_t(x) + \omega(x)^2 \in [0,u] \Big\} \,.
\end{equation}
Identifying $\bbL_r(y)$ with $\bbL_r$ as in Subsection~\ref{ss:6.2}, we note that $1_{A_{r,t^\rmA_n}(y)}$, which is a random function, is supported on $\{\omega \in \bbR^{\bbL_r} :\: \min_{x \in \ol{\bbL}_r} |\omega(x)| \leq \sqrt{u}\}$ for any $n$, $y$ and any realization of $L_{t^\rmA_n}(\ol{\bbT}_n)$.

Thanks to Lemma~\ref{l:6.4}, conditional on $[\cE_n^r]_{n-r}$ and $L_{t^\rmA_n}(\bbL_r(y))$ for all $y \in [\cE_n^r]_{n-r}$, the first expectation in~\eqref{e:6.17} is then equal to
\begin{equation}
\label{e:6.31}
\exp \bigg(- \tfrac{1}{\sqrt{n}} (1-\rme^{-\lambda}) \sum_{y \in [\cE_n^r]_{n-r}} \nu_r\big(A_{r,t_n^{\rmA}}(y)) \bigg) + o(1) \,,
\end{equation}
with the $o(1)$ tending to $0$ as $n \to \infty$ uniformly in $L_{t^\rmA_n}(\ol{\bbT}_n)$.
Moreover since the marginals of $\nu_r$ are Radon, in view of the support of $1_{A_{r,t}(y)}$, we have
\begin{equation}
\label{e:6.32.1}
\nu_r\big(A_{r,t_n^{\rmA}}(y)\big) \leq \sum_{x \in \bbL_r} \nu_r \Big( \omega \in \bbR^{\bbL_r} :\:   \omega(x) \in [-\sqrt{u}, \sqrt{u}]\Big) =: C(r,u) < \infty \,,
\end{equation}
for all $n$ and $y$ and any realization of $L_{t^\rmA_n}(\ol{\bbT}_n)$.

As in the proof of Proposition~\ref{p:6.3}, for any $l \in (r, n/2)$, if we now just condition on $L_{t^\rmA_n}(\ol{\bbT}_{n-l})$ and $[\cE_n^r]_{n-r}$, 
then the clusters
\begin{equation}
\Big\{ L_{t^\rmA_n}(\bbT(y)) :\: y \in [\cE_n^r\big]_{n-r-1} \Big\}
\end{equation}
are independent, with the law of $L_{t^\rmA_n}(y)$ given by $P_l^{l-r-1}(L_{t^\rmA_n}([y]_{n-l}) ; \cdot)$. 
In particular, under this conditioning, the random variables $\nu_r(A_{r,t_n^{\rmA}}(y))$ for $y \in [\cE_n^r]_{n-r}$ are independent and have mean
\begin{equation}
\label{e:6.32}
\int \bbE \Big(\nu_r \big(A_{r,v}(1)\big)  \,\Big|\, \cF_{r+1,v} \neq \emptyset
	\Big)\rmP_l^{l-r-1}(L_{t^\rmA_n}([y]_{n-l}) \;; \rmd v) \,,
\end{equation}
where we write '$1$' to denote the (single) child of the root of $\ol{\bbT}_{r+1}$.

Therefore, if we set
\begin{equation}
\label{e:6.33b}
C_{A_2}^{r,l} := 
\int \bbE \Big(\nu_r \big(A_{r,v}(0)\big)  \,\Big|\, \cF_{r+1,v} \neq \emptyset \Big) \mu_l^{l-r-1} (\rmd v) \,
\end{equation}
then, since the expectation above %in~\eqref{e:6.33} 
is bounded uniformly in $v$ in view of~\eqref{e:6.32.1} and also $\sqrt{L_{t^\rmA_n}([y]_{n-l})} \in \sqrt{\log 2}l + \frR_l$ because $[y]_{n-l} \in \big[\cQ_n^{[n/2, n-r]}\big]_{n-l}$, we can use Lemma~\ref{l:6.2} to claim that the integrals in~\eqref{e:6.32} and~\eqref{e:6.33b} can be made arbitrarily close to each other, uniformly in $n$ and $L_{t^\rmA_n}(\bbT_{n-l})$, by choosing $l$ large enough.

Using again that the $\nu_r(A_{r,t_n^{\rmA}}(y))$ are bounded uniformly, it then follows from Chebyshev's inequality that for any $\delta, \delta' > 0$ on
\begin{equation}
\label{e:6.37a}
\Big\{\big|\big[\cE_n^{r}\big]_{n-r}\big| > \delta' \sqrt{n} \Big\} \,,
\end{equation}
the conditional probability of
\begin{equation}
\Bigg\{ \bigg|
\frac{\sum_{y \in [\cE_n^r]_{n-r}} \nu_r\big(A_{r,t_n^{\rmA}}(y)\big)}{\big|\big[\cE_n^r\big]_{n-r}\big|} - C_{A_2}^{r,l}  \bigg| > \delta \Bigg\} \,,
\end{equation}
goes to $0$ when $n \to \infty$, as long as $l$ is chosen large enough. Proceeding exactly as in the proof of Proposition~\ref{p:6.3}, we assert that for any $r$ large enough, the limit as $l \to \infty$ of $C_{A_2}^{r,l}$ must exist, and that if we denote it by $C_{A_2}^r \in [0,\infty)$ then
\begin{equation}\label{e:7.50}
\lim_{r \to \infty} \limsup_{n \to \infty}
\bbP \Bigg( \Bigg|
\frac{\sum_{y \in [\cE_n^r]_{n-r}} \nu_r\big(A_{r,t_n^{\rmA}}(y)\big)}{\big|\big[\cE_n^r\big]_{n-r}\big|} - C_{A_2}^{r}  \Bigg| > \delta \Bigg) = 0 \,,
\end{equation}

Using that $|\rme^{-\alpha} - \rme^{-\beta}| \leq |\alpha - \beta|$ for all $\alpha,\beta \geq 0$, on the complement of the event in~\eqref{e:7.50} we have that the difference between the first term in~\eqref{e:6.31} and
\begin{equation}
\exp \Big(- \tfrac{1}{\sqrt{n}} (1-\rme^{-\lambda}) C_{A_2}^r \big| \big[\cE_n^r\big]_{n-r}\big| \Big)
\end{equation} is less or equal than 
\begin{equation}
\delta(1-\rme^{-\lambda})\frac{\big| \big[\cE_n^r\big]_{n-r}\big|}{\sqrt{n}} \leq \delta(1-\rme^{-\lambda})\frac{\big| \big[\cF_n\big]_{r_n}\big|}{\sqrt{n}}\,.
\end{equation} In particular, by the tightness of the sequence $(\tfrac{1}{\sqrt{n}}\big| \big[\cF_n\big]_{r_n}\big| : n \geq 1)$ as given by Lemma~\ref{l:3.2}, we see that this difference tends to $0$ in probability as $n \to \infty$ followed by $r \to \infty$. Since this difference is also always bounded by $2$, taking expectation and using the bounded convergence theorem then yields the desired claim.
\end{proof}

\section{Phase B: Proof of Theorem~\ref{t:2b}}
\label{s:8}
The goal in this section is to prove Theorem~\ref{t:2b}, which is the key result for phase $B$.

\begin{proof}[Proof of Theorem~\ref{t:2b}]
Let $y \in [\cL_n]_{n-r_n}$ and set $\cL_n(y) := \bbL_{r_n}(y) \cap \cL_n$.
Writing $t_n(s)$ as a short for $t_n^{\rmB}+sn$, we first claim that 
\begin{equation}
\label{e:2.11}
q_{n,t_n(s)}(y) := 
\bbP \Big(\cF_{n,t_n(s)}(0) \cap \cL_n(y) \neq \emptyset \Big)
= \frac{1}{\sqrt{n}}\rme^{-s} (1+o(1)) \,,
\end{equation}
where the $o(1)$ tends to $0$ as $n \to \infty$ uniformly in $\cL_n(y)$.

Indeed, for a lower bound, observe that a random walk starting at $[y]_1$, has probability $1/(n-r_n)$ of hitting $y$ before hitting the root. Therefore, since the number of visits from $0$ to $[y]_1$ until time $\bfL^{-1}_{t_n(s)}(0)$ is Poisson with rate $t_n(s)$, the probability that $\{L_{t_n(s)}(y)=0\}$ is the same as the probability that a Poisson with rate $t_n(s)/(n-r_n) = (\log n)/2 + s + o(1)$ is zero. This gives
\begin{equation}
\bbP \Big(\cF_{n,t_n(s)}(0) \cap \cL_n(y) \neq \emptyset \Big) \geq 
\bbP \big(L_{t(s)}(y) = 0\big) = \frac{\rme^{-s}}{\sqrt{n}} (1+o(1)) \,.
\end{equation} 

For a matching upper bound, let $p_n$ be the probability that a random walk starting at $[y]_1$ visits all sites in $\bbL_{r_n}(y)$ before returning to the root. Such an event occurs for example if, given some fixed $\eta' \in (0,1)$, the random walk spends at least  $n^{1-\eta'}$ time at $y$ before returning to the root and while accumulating this time also visits all sites in $\bbL_{r_n}(y)$. We have already seen that the probability of reaching $y$ before returning to the root is $1/(n-r_n)$. By the same reasoning, upon reaching $y$ the local time accumulated at $y$ before returning to the root is Exponential with rate $1/(n-r_n) < 2/n$ for all $n$ large enough. Consequently the time accumulated will be at least $n^{1-\eta'}$ with probability at least $1-2n^{-\eta'}$. Then, recalling that $\bbT_{r_n}(y)$ is the union of two trees isomorphic to $\ol{\bbT}_{r_n}$, by Lemma~\ref{l:3.5} the probability of not visiting all of $\bbL_{r_n}(y)$ is, for all $n$ large enough, bounded from above by 
\begin{equation}
2\bbP \big(\cF_{r_n, n^{1-\eta'}}(0) \neq \emptyset \big) \leq C \rme^{-C' n^{1/2 - \eta'/2}} \,,
\end{equation}
whenever $\frac{\eta'}{2} < \eta$, with $\eta$ as in~\eqref{e:2.3}. Combining all of the above, we get
\begin{equation}
p_n \geq \frac{1}{n-r_n} \big(1-2n^{-\eta'}\big) \big( 1- C \rme^{-C' n^{1/2 - \eta'/2}} \big)
\geq \frac{1}{n} \big(1-O(n^{-\eta'}) \big) \,.
\end{equation}
Arguing as for the lower bound, the probability that $\{\bbL_{r_n}(y) \cap \cF_{n,t_n(s)}(0) \neq \emptyset\}$ is therefore at most the probability that a Poisson with rate $t_n(s) p_n \geq (\log n)/2 + s + o(1)$ is equal to zero. This gives
\begin{equation}
\bbP \Big(\cF_{n,t_n(s)}(0) \cap \cL_n(y) \neq \emptyset \Big) \leq 
\bbP \big(\bbL_{r_n}(y) \cap \cF_{n,t_n(s)}(0) \neq \emptyset \Big)
\leq \rme^{-t_n(s) p_n} = \frac{\rme^{-s}}{\sqrt{n}} (1+o(1)) \,,
\end{equation}
and proves the claim in~\eqref{e:2.11}.

Now, conditional on $L_{t_n(s)}(\ol{\bbT}_{r_n})$, by the Markov property and the fact that $\cL_n$ is $(r_n, n-r_n)$-clustered, for any $\lambda \geq 0$, the expectation in~\eqref{e:2.9} is equal to
\begin{equation}
\label{e:2.16}
\exp \Big( \sum_{y \in [\cL_n]_{n-r_n}} \log \Big( 1 - \big(q_{n-r_n, L_{t_n(s)}([y]_{r_n})}(y)\big)
\big(1-\rme^{-\lambda}\big) \Big) \Big)
\end{equation}
Therefore, for any given $\eta'' \in (0,\eta)$, on the event
\begin{equation}
\label{e:2.17}
\big\{ \max_{z \in \ol{\bbL}_{r_n}} |L_{t_n(s)}(z) - t_n(s)| < n^{1-\eta''} \big\}
\end{equation}
we have $L_{t_n(s)}([y]_{r_n}) = t_{n-r_n}(s + O(n^{-\eta''}))$, so that by~\eqref{e:2.11} 
\begin{equation}
q_{n-r_n, L_{t_n(s)}}([y]_{r_n})=\frac{1}{\sqrt{n-r_n}}\rme^{-(s+O(n^{-\eta''}))}(1+o(1))=\frac{\rme^{-s}}{\sqrt{n}}(1+o(1)),
\end{equation} where the $o(1)$-term tends to $0$ as $n \to \infty$ uniformly in $\cL_n(y)$.
Plugging this in~\eqref{e:2.16} and using Taylor expansion for the log shows that, on the event in~\eqref{e:2.17},~\eqref{e:2.16} is equal to
\begin{equation}\label{e:7.8}
\exp \Big(\!- \tfrac{\rme^{-s}}{\sqrt{n}}|[\cL_n]_{r_n}| \big(1-\rme^{-\lambda}\big) (1+o(1)) \Big) \,,
\end{equation} where we have also used that $|[\mathcal{L}_n]_{n-r_n}|=|[\mathcal{L}_n]_{r_n}|$ since $\mathcal{L}_n$ is $(r_n,n-r_n)$-clustered. Moreover, since $|\rme^{-\alpha}-\rme^{-\alpha(1+\epsilon)}| \leq 2\rme^{-1}|\epsilon|$ for all $\alpha \geq 0$ and $|\epsilon|< \tfrac{1}{2}$, we have that~\eqref{e:7.8} is in fact
\begin{equation}
\exp \Big(\!- \tfrac{\rme^{-s}}{\sqrt{n}}|[\cL_n]_{r_n}| \big(1-\rme^{-\lambda}\big) \Big)+o(1) \,.
\end{equation}

Lastly, we note that the probability of~\eqref{e:2.17} tends to $1$ as $n \to \infty$. Indeed, by the isomorphism we have that for all $z \in \ol{\bbL}_{r_n}$
\begin{equation}
|L_{t_n(s)}(z)-t_n(s)| \leq 2\sqrt{t_n(s)}|h'_{r_n}(z)| + |h'_{r_n}(z)|^2 +  |h_{r_n}(z)|^2,
\end{equation} so that, for $n$ large enough,~\eqref{e:2.17} is implied by the intersection 
\begin{equation}
\{\max_{z \in \ol{\bbL}_{r_n}}|h_{r_n}(z)| \leq r_n \log r_n \} \cap \{\max_{z \in \ol{\bbL}_{r_n}}|h'_{r_n}(z)| \leq r_n \log r_n \}\,,
\end{equation} whose probability goes to $1$ as $n \to \infty$ thanks to Proposition \ref{p:3.1}. Hence, by taking expectation in~\eqref{e:2.16}, combining all of the above we arrive at~\eqref{e:2.9} with $C_B := 1$.

The proof of~\eqref{e:2.10} is even easier. As we have argued above, for any $x \in \ol{\bbL}_n$, the probability that $\{L_{t_n(s)}(x) = 0\}$ is the probability that a Poisson with rate $t_n(s)/n = (\log n)/2 + s$ is $0$. This probability is equal to $\rme^{-t_n(s)/n} = \rme^{-s}/\sqrt{n}$. Consequently,
\begin{equation}
\bbE \big|\cL_n \cap \cF_{n,t_s(n)}(0) \big| = |\cL_n| \tfrac{\rme^{-s}}{\sqrt{n}} \,,
\end{equation}
and~\eqref{e:2.10} now follows by Markov's Inequality.
\end{proof}

\section{From local time at the root to real time: Proof of Theorem~\ref{t:1.2.5}}
\label{s:9}

In this section we prove Theorem~\ref{t:1.2.5}, which is the key statement needed for the proof of Theorem~\ref{t:2.5}. To this end, we shall need several preparatory results and a few new definitions. Recall that the over-lined tree notations, such as $\bbL_k$ or $\bbT_k$, concern the ``normal'' binary tree $\bbT$. For any $n \geq 1$ and $t \geq 0$, we set
\begin{equation}
\label{e:1.10}
S_{n,t} := \sum_{x \in \bbL_n} L_t(x)
\  , \quad
\wh{S}_{n,t} := 2^{-n} S_{n,t} 
\qquad ;  \qquad
R_{n,t} := \sum_{k=0}^n S_{k,t} 
\ , \quad
\wh{R}_{n,t} := 2^{-n} R_{n,t} \,.
\end{equation}
Observe that, with these definitions, the total real time spent by the walk when the local time at the root is $t$ is therefore,
\begin{equation}
\label{e:203.1}
R_{n,t} = \bfL^{-1}_{n,t}(0)\,.
\end{equation}

In the next lemma we compute the mean and covariance structure of the above quantities. These are easy computations whose result will be needed in the near sequel.
\begin{lem}
\label{l:1.4}
For all $t \geq 0$, $n \geq 1$ and $x,y \in \bbL_n$,
\begin{align}
\bbE S_{n,t} & = 2^n t  \,  & \Var\, S_{n,t} & = 2^{2n+1} t (1+o(1)) \,, \label{e:9.3}\\
\bbE R_{n,t} & = (2^{n+1}-1) t \, & \Var\,R_{n,t} & = 2^{2n+3} t (1+o(1)) \label{e:9.4}\
\end{align}
and
\begin{equation}\label{e:9.5}
\Cov \big(L_t(x) - \wh{S}_{n,t} ,\, L_t(y) - \wh{S}_{n,t}\big) = 2\big(|x\wedge y| - 1 + o(1)\big)t \,,
\end{equation}
where all the $o(1)$-terms depend only on $n$ and tend to $0$ as $n \to \infty$.
\end{lem}
\begin{proof}
For each $n \geq 1$ and $x \in \bbL_n$ we have that
\begin{equation}
L_t(x) \eqd \sum_{i=1}^N E_i\,,
\end{equation} where the $E_i$ are i.i.d. Exponentials  with rate $1/n$ and $N$ is an independent Poisson of rate $t/n$. It then follows that $\bbE L_t(x)= (t/n) n = t$ and $\Var\, L_t(x) =2(t/n) n^2 = 2nt$. 
Now, set $z := x \wedge y$. 
Then, conditional on $L_t(z) = u$, $L_t(x)$ and $L_t(y)$ are independent and equal in law to $L_u(x')$ and $L_u(y')$ respectively, where $x'$, $y'$ are any  vertices in $\bbT$ satisfying $|x'| = |x|-|z|$ and $|y'| = |y|-|z|$. From this it follows that $\bbE (L_t(x) L_t(y)) = \bbE L_t(z)^2 = 2t|z| + t^2$ which, by subtracting the product of the means $\bbE L_t(x) \bbE L_t(y)=t^2$, yields that
\begin{equation}
\label{e:102.6}
\Cov \big(L_t(x), L_t(y) \big) = 2|x \wedge y|t \,,
\end{equation}
for any $x,y \in \bbT$.

Summing the means of $L_t(x)$ over all $x \in \bbL_n$ gives $\bbE S_{n,t} = 2^n t$. For the variance of $S_{n,t}$, by grouping all pairs of vertices according to the depth of their common ancestor, we can write $\Var\, S_{n,t}$ as 
\begin{equation}
\label{e:1.12}
\sum_{x,y \in \bbL_n} \Cov \big(L_t(x), L_t(y) \big) = \sum_{k=0}^n 2^k 2^{(2(n-k)-1) \vee 0} 2t k = 
2^{2n} t \sum_{k=0}^n 2^{-k + 1_{\{k=n\}}} k = 2^{2n+1} t (1+o(1)) \,,
\end{equation}
in accordance with~\eqref{e:9.3}.

Now, to check~\eqref{e:9.4}, by summing the means of $S_{k,t}$ over $k=0,\dots,n$, we immediately obtain $\bbE R_{n,t} = (2^{n+1} - 1)t$. To compute the variance of $R_{n,t}$, we again sum all covariances as in~\eqref{e:1.12}, only that the term $2^{(2(n-k)-1) \vee 0}$ is now replaced by $2(2^{n-k} - 1)^2 + 1_{\{k=n\}}$, to account for leaves at different depths. This gives an additional factor of $4$ in the asymptotic expression for $\Var\,R_{n,t}$ and completes the proof of~\eqref{e:9.4}.

Finally, to show~\eqref{e:9.5} we first observe that for all $n \geq 1$ and $x \in \bbL_n$,
\begin{equation}
\label{e:102.8}
\Cov \big(L_t(x) - \wh{S}_{n,t} ,\, \wh{S}_{n,t}\big) = 0\,.
\end{equation}
Indeed, summing $L_t(x) - \wh{S}_{n,t}$ over all $x \in \bbL_n$ gives $0$, which implies that the sum of the covariance in~\eqref{e:102.8} over all such $x$-s is $0$ as well. Then, by symmetry, this must also be the case for each $x$ individually. Writing $\Cov \big(L_t(x) - \wh{S}_{n,t} ,\, L_t(y) - \wh{S}_{n,t}\big)$ as 
\begin{equation}
\Cov \big(L_t(x) ,\, L_t(y)\big) - 
\Cov \big(L_t(x) - \wh{S}_{n,t},\, \wh{S}_{n,t}\big) 
- \Cov \big(\wh{S}_{n,t},\, L_t(y) - \wh{S}_{n,t}\big) -
\Var\, \wh{S}_{n,t} \,,
\end{equation}
we can use~\eqref{e:102.6},~\eqref{e:1.12} and~\eqref{e:102.8} to evaluate the terms in the last display and thus obtain~\eqref{e:9.5}.
\end{proof}

The next step is to derive analogous formulas for $\bbE S_{n,\tau}$ and $\bbE R_{n,\tau}$ when $\tau$ is a stopping time of the local time field $(L_t(\bbL_n) : t \geq 0)$.

\begin{lem}\label{l:os} For $k \geq 1$, let $(\mathfrak{F}_{k,t} : t \geq 0)$ be the filtration given by $\mathfrak{F}_{k,t}:= \sigma(L_s(\bbT_k) : 0\leq s \leq t)$. Then, for any (not necessarily integrable) stopping time $\tau$ with respect to $(\mathfrak{F}_{k,t} : t \geq 0)$ which is finite almost-surely,
\begin{equation}\label{e:2.10z}
\bbE S_{k,\tau} = 2^k \bbE \tau \hspace{1cm}\text{ and }\hspace{1cm}\bbE R_{k,\tau} = (2^{k+1}-1)\bbE \tau\,.
\end{equation}
\end{lem}

\begin{proof} Since by the Markov property of $X_n$ and Lemma~\ref{l:1.4} we have that the processes $(S_{k,t} - 2^k t : t \geq 0)$ and $(R_{k,t} - (2^{k+1}-1)t : t \geq 0)$ are martingales with respect to $(\mathfrak{F}_{k,t} : t \geq 0)$, it then follows from the optional stopping theorem that for any $M>0$ 
\begin{equation}\label{e:2.11z}
\bbE S_{k,\tau \wedge M} = 2^k \bbE (\tau \wedge M) 
\hspace{1cm}\text{ and }\hspace{1cm}
\bbE R_{k,\tau \wedge M} = (2^{k+1}-1)\bbE (\tau \wedge M).
\end{equation}
Thus, since $\tau \wedge M \nearrow \tau$ almost surely as $M \to \infty$ and the trajectory $u \mapsto L_u(x)$ is nonnegative and increasing for all $x \in \bbT$, by taking $M \to \infty$ in~\eqref{e:2.11z} the monotone convergence theorem then yields~\eqref{e:2.10z}.	
\end{proof}

Next, we claim that, for all $n \geq 1$, $\wh{R}_{n,t}$ can be well approximated by $2\wh{S}_{k,t}$ if $k$ is large enough.
\begin{lem}
\label{l:101.5}
There exists $C > 0$ such that for all $1 \leq k \leq n$, $\epsilon > 0$ and any stopping time $\tau$ with respect to $(\mathfrak{F}_{k,t}: t \geq 0)$ (where $\mathfrak{F}_{k,t}$ is as in Lemma~\ref{l:os}), which is finite almost-surely,
\begin{equation}
\label{e:201.14}
\bbP \Big(\big|\wh{R}_{n,\tau} - 2\wh{S}_{k,\tau}|  > \epsilon \Big)
\leq C  \big(\epsilon^{-2}2^{-k} +\epsilon^{-1} 2^{-(n-k)})\bbE \wh{S}_{k,\tau}  
\,.
\end{equation}
\end{lem}
\begin{proof}
Conditional on $L_t(\bbT_k)$, the law of $\wh{R}_{n,t}$ is the same as that of
\begin{equation}
2^{-n} \Big(R_{k-1,t} + \sum_{x \in \bbL_k} R^{(x)}_{n-k, L_t(x)} \Big) \,,
\end{equation}
where $(R^{(x)}_{n-k, L_t(x)}: x \in \bbL_k)$ are independent and satisfy $R^{(x)}_{n-k, L_t(x)} \eqd R_{n-k, L_t(x)}$ for each $x \in \bbL_k$. It then follows from Lemma~\ref{l:1.4}
that the conditional mean and variance of $\wh{R}_{n,t}$ are respectively
\begin{equation}
2^{-n} R_{k-1, t} + 2 \wh{S}_{k,t} \big(1 - 2^{-(n-k+1)}\big)
\quad  , \qquad
2^{-k+3} \wh{S}_{k,t} (1+o(1))\,,
\end{equation}
with the $o(1)$-term tending to $0$ as $n-k \to \infty$. Therefore. for any $\epsilon > 0$, whenever $2^{-n} R_{k,t} \leq  \epsilon/2$ we can use Chebyshev's inequality to conclude that
\begin{equation}
\label{e:204.3}
\bbP \Big( \big|\wh{R}_{n,t} - 2\wh{S}_{k,t}| > \epsilon\,\Big|\, L_t(\bbT_k) \Big) 
\leq C' \epsilon^{-2} 2^{-k} \wh{S}_{k,t} \,
\end{equation}
for some properly chosen $C' > 0$. Now, by the Markov property of local times,~\eqref{e:204.3} also holds if $t$ is replaced by any stopping time as in the statement of the lemma. By taking expectation and then using the union bound together with Markov's inequality, we obtain that 
\begin{equation}
\bbP \Big(\big|\wh{R}_{n,\tau} - 2\wh{S}_{k,\tau}|  > \epsilon \Big)
\leq C' \epsilon^{-2} 2^{-k} \bbE \wh{S}_{k,\tau} +  2 \epsilon^{-1} 2^{-n}  \bbE R_{k,\tau} \,,
\end{equation}
which, by Lemma~\ref{l:os}, is at most the right hand side of~\eqref{e:201.14} if $C$ is chosen correctly.
\end{proof}

When $n,n-k$ are large, the previous lemma shows that $2S_{k,t}$ acts as a good approximation of $R_{n,t} = \bfL^{-1}_{n,t}(0)$. Consequently, instead of running the walk until real time $s \geq 0$ we can run it until the local time at the root is $\tau_{k,s}$, where $\tau_{k,s}$ is the stopping time
\begin{equation}\label{e:9.18}
\tau_{k,s} := \inf \big\{t \geq 0 :\: 2S_{k,t} > s \,\big\} \,.
\end{equation}
The advantage of this is that $\tau_{k,s}$ depends only the random walk restricted to $\bbT_k$ and, moreover, that the asymptotic law of $L_{\tau_{k,s}}(\bbL_k)$ can be easily derived. In order to do this, we first need to obtain a suitable control on the mean of $\wh{S}_{k,\tau_{k,s}}$. 

\begin{lem}\label{l:9.4} For any $k \geq 1$ and $s \geq 0$, we have $\wh{S}_{k,\tau_{k,2^{k+1}s}} \geq s$ and
\begin{equation}\label{e:o1}
\bbE \wh{S}_{k,\tau_{k,2^{k+1}s}} \leq s+ 1\,.
\end{equation} 
In particular, for any fixed $k \geq 1$, as $s \to \infty$, in probability,
\begin{equation}\label{e:o2}
\frac{\wh{S}_{k,\tau_{k,2^{k+1}s}} -s }{\sqrt{s}} \longrightarrow 0\,.
\end{equation} 
\end{lem}

\begin{proof} First, notice that, since $t \mapsto L_t(x)$ is right-continuous for any $x \in \bbT$ by definition of $L_t$, one has that $t \mapsto S_{k,t}$ is right-continuous as well. In particular, by definition of $\tau_{k,2^{k+1}s}$ it follows that $S_{k,\tau_{k,2^{k+1}s}} \geq 2^k s$ and consequently that $\wh{S}_{k,\tau_{k,2^{k+1}s}} \geq s$. Moreover, $S_{k,\tau_{k,2^{k+1}s}} - 2^ks$ is precisely the additional local time collected at the leaves in $\bbL_k$ from the moment the sum of their local times reached $2^ks$ and until the walk returns to the root for the first time after. That is, if we set
\begin{equation}
\mathcal{T}^1_{k,s}:= \inf \Big\{ t \geq 0 : \sum_{x \in \bbL_k} \bfL_{k,t}(x) \geq 2^k s\Big\} \hspace{1cm}\text{ and }\hspace{1cm}\mathcal{T}^2_{k,s}:= \inf \Big\{ t \geq \mathcal{T}^1_{k,s} : X_{k,t}=0\Big\}\,,
\end{equation} then 
\begin{equation}\label{e:9.22}
S_{k,\tau_{k,2^{k+1}s}} - 2^ks = \sum_{x \in \bbL_k} \Big( \bfL_{k,\mathcal{T}^2_{k,s}}(x) - \bfL_{k,\mathcal{T}^1_{k,s}}(x)\Big)\,. 
\end{equation} 

Now, by the strong Markov property of the random walk on $\bbT_k$, the right hand side of~\eqref{e:9.22} is equal in law to the total local time accumulated at $\bbL_k$ for a random walk starting at one of these leaves, and run until reaching the root. Since the depth of such a walk forms a standard Gambler's ruin problem on $\{0, \dots, k\}$ with probability $2/3$ of winning a game, it follows by standard computations that the mean number of returns to the leaves before getting to the root is $2^k-1$. Since at each visit to the leaves, the sum of their local times increases by an Exponential with mean $1$, it follows that 
\begin{equation}
\bbE( S_{k,\tau_{k,2^{k+1}s}} - 2^ks) = 2^k - 1 \,,
\end{equation} from where~\eqref{e:o1} readily follows. To conclude the proof we observe that, since  $\wh{S}_{k,\tau_{k,2^{k+1}s}} \geq s$, Markov's inequality and~\eqref{e:o1} together yield that for any $\epsilon>0$
\begin{equation}
\bbP( |\wh{S}_{k,\tau_{k,2^{k+1}s}}-s|> \sqrt{s}\epsilon) \leq \frac{\bbE (\wh{S}_{k,\tau_{k,2^{k+1}s}} -s)}{\sqrt{s}\epsilon} \leq \frac{1}{\sqrt{s}\epsilon} \longrightarrow 0
\end{equation} as $s \to \infty$, from where~\eqref{e:o2} now follows. 
\end{proof}

In order to derive the asymptotic law of $L_{\tau_{k,s}}(\bbL_k)$ we shall need a generalization of the central limit theorem for a c\`adl\`ag martingale stopped at a random (not necessarily stopping) time. Such a statement can be found, e.g. in~\cite[Theorem 7.3.2]{chung2001course}, however in a discrete time setting. In order to obtain a version of this result in continuous time, it is enough to replace a key step in the proof there with the following lemma.
\begin{lem} Let $(W_s : s \geq 0)$ be a c\`adl\`ag martingale with respect to some filtration $(\mathfrak{F}_s : s \geq 0)$, satisfying $\bbE W_s^2 = C s$ for some fixed $C>0$ and all $s \geq 0$. Then, for any collection $(\tau_s : s \geq 0)$ of finite random times such that 
\begin{equation}
\frac{\tau_s}{s} \longrightarrow 1\,
\end{equation} in probability as $s \to \infty$, we have that 
\begin{equation}
\frac{W_{\tau_s} - W_s}{\sqrt{s}} \longrightarrow 0
\end{equation} in probability as $s \to \infty$.
\end{lem}

\begin{proof} To prove this result we can essentially mimic the proof of ~\cite[Theorem 7.3.2]{chung2001course}, substituting the use of Kolmogorov's inequality there by Doob's maximal inequality for c\`adl\`ag martingales. The straightforward details are left to the reader. 
\end{proof}
We are now ready to derive the asymptotic law of $L_{\tau_{k,s}}(\bbL_k)$.

\begin{prop}
\label{l:102.4}
Let $k \geq 1$. Then as $s \to \infty$,
\begin{equation}
\label{e:101.23}
\bigg(\frac{L_{\tau_{k,2^{k+1} s}}(x) - s}{\sqrt{s}} :\: x \in \bbL_k \bigg)
\Longrightarrow \cN \Big(0, \big(2(|x\wedge y| - 1) + o(1)\big)_{x,y \in \bbL_k} \Big) \,,
\end{equation}
where the $o(1)$-term tends to $0$ as $k \to \infty$.
\end{prop}

\begin{proof} Abbreviating $s':=2^{k+1}s$, by~\eqref{e:o2} in Lemma~\ref{l:9.4} it will suffice to show that as $s \to \infty$
\begin{equation}\label{e:9.30}
\bigg(\frac{L_{\tau_{k,s'}}(x) - \wh{S}_{k,\tau_{k,s'}}}{\sqrt{s}} :\: x \in \bbL_k \bigg)
\Longrightarrow \cN \Big(0, \big(2(|x\wedge y| - 1) + o(1)\big)_{x,y \in \bbL_k} \Big) \,.
\end{equation}

To this end, we notice that, by definition of $\tau_{k,s}$, for all $\epsilon > 0$ we have
\begin{equation}\label{e:2.19z}
\bbP \big(\tau_{k,s'} > s(1+\epsilon)\big)  = 
\bbP \big(\wh{S}_{k,s(1+\epsilon)} < s \big) = 
\bbP \big(\wh{S}_{k,s(1+\epsilon)} < s(1+\epsilon) - \epsilon s \big) \,.
\end{equation}
By Lemma~\ref{l:1.4}, the mean and variance of $\wh{S}_{k,s(1+\epsilon)}$ are $s(1+\epsilon)$ and $2s(1+\epsilon)(1+o(1))$, respectively. It then follows from Chebyshev's inequality that the probabilities in~\eqref{e:2.19z} go to $0$ as $s \to \infty$. Since the same holds for 
$\bbP \big(\tau_{k,s'} < s(1-\epsilon)\big)$ via a similar argument, it follows that $\tau_{k,s'}/s \to 1$ in probability as $s \to \infty$. 

Now, for each $x \in \bbL_k$ the process $(W_s(x) : s \geq 0)$ given by $W_s(x):=L_s(x)-\wh{S}_{k,s}$ is a c\`adl\`ag martingale which, by Lemma~\ref{l:1.4}, satisfies $\bbE (W_s(x))^2 = 2(k-1 +o(1))s$. We may then use Lemma~\ref{l:9.4} (applied to each $x$ separately) to conclude that showing~\eqref{e:9.30} is equivalent to proving that
\begin{equation}\label{e:9.32}
\frac{L_{s}(\bbL_k) - \wh{S}_{k,s}}{\sqrt{s}}:=	\bigg(\frac{L_{s}(x) - \wh{S}_{k,s}}{\sqrt{s}} :\: x \in \bbL_k \bigg)
\Longrightarrow \cN \Big(0, \big(2(|x\wedge y| - 1) + o(1)\big)_{x,y \in \bbL_k} \Big) \,,
\end{equation} as $s \to \infty$. But this is now a straightforward consequence of the multivariate central limit theorem. Indeed, by splitting time into intervals of length $1$ we can write 
\begin{equation}\label{e:9.33}
\frac{L_{s}(\bbL_k) - \wh{S}_{k,s}}{\sqrt{s}} = \sqrt{\frac{\lfloor s \rfloor}{s}} \bigg( \frac{1}{\sqrt{\lfloor s \rfloor}} \sum_{u=1}^{\lfloor s \rfloor} W^{(u)}_1(\bbL_k) \bigg)+ \frac{W^{(\lceil s \rceil)}_{\lceil s \rceil -s}(\bbL_k)}{\sqrt{s}}\,, 
\end{equation} where, for each $u \geq 1$ and any $\delta \in [0,1]$, we define $W^{(u)}_\delta(\bbL_k) := (W_{u}(x) - W_{u-\delta}(x) : x \in \bbL_k)$. Since $\|W^{(\lceil s \rceil)}_{\lceil s \rceil -s}(\bbL_k)\|_\infty$ is stochastically dominated by $S_{k,1}$, the second term in the right hand side of~\eqref{e:9.33} tends to zero in probability as $s \to \infty$. Therefore, upon noticing that $( W^{(u)}_1(\bbL_k) : u \geq 1)$ are i.i.d. centered random vectors with covariance matrix equal to that on the right hand side of~\eqref{e:9.32}, the limit in~\eqref{e:9.32} now follows from~\eqref{e:9.33} by the multivariate central limit theorem. 
\end{proof}

With all the results above, we can now show Theorem~\ref{t:1.2.5}.

\begin{proof}[Proof of Theorem~\ref{t:1.2.5}]

Given any $k \geq 1$ and $s \geq 0$, from~\eqref{e:3.1} applied to $L_t$, we have 
\begin{equation}
\bigg\| \frac{L_s(\bbL_{k}) - s}{\sqrt{s}} - 2h'_{k}(\bbL_{k}) \bigg\|_\infty \leq \frac{\big\| h'^2_{k}(\bbL_{k})\big\|_\infty + \big\|h^2_{k}(\bbL_{k})\big\|_\infty}{\sqrt{s}}\,.
\end{equation}
By letting $s \to \infty$ on the right hand side, we see that the left hand side tends to $0$ in probability as $s \to \infty$, and therefore that under the same limit,
\begin{equation}
\label{e:102.26}
\frac{L_{s}(\bbL_k) - s}{\sqrt{s}} \Longrightarrow 2h'_{k}(\bbL_{k}) \,.
\end{equation}
Since $h'_{k}$ is the DGFF on $\bbT_k$, the right hand side of~\eqref{e:102.26} has a Gaussian law with mean $0$ and covariance matrix $(2|x \wedge y|)_{x,y \in \bbL_k}$.

On the other hand, consider a random variable $\xi \sim \cN(0,1/2)$ independent of the walk $X_n$ (without loss of generality, we may assume that our current probability space is large enough to support such a $\xi$), as in the statement of the theorem, and for $s \geq 0$ set $\theta_s := s - 2\sqrt{s} \xi \sim \cN(s, 2s)$. If we write $\nu_{k,s}:=\tau_{k,2^{k+1} \theta_s}$ then, conditional on $\xi$, by Proposition~\ref{l:102.4} we have as $s \to \infty$,
\begin{equation}
\label{e:102.27}
\frac{L_{\nu_{k,s}}\big(\bbL_k \big) - s}{\sqrt{s}} = 
\frac{L_{\nu_{k,s}}\big(\bbL_k \big) - \theta_s}{\sqrt{\theta_s}} \sqrt{\frac{\theta_s}{s}} - 2\xi
\Longrightarrow \cN \Big(\!-2\xi, \big(2(|x\wedge y| - 1) + o(1)\big)_{x,y \in \bbL_k} \Big) \,.
\end{equation}
It then follows from the bounded convergence theorem that unconditionally,
\begin{equation}
\frac{L_{\nu_{k,s}}\big(\bbL_k \big) - s}{\sqrt{s}} 
\Longrightarrow \cN \Big(0, \big(2(|x\wedge y|) + o(1)\big)_{x,y \in \bbL_k} \Big) \,.
\end{equation}
Since the distribution on the right hand side above tends, as $k \to \infty$, to the limiting law in~\eqref{e:102.26}, for any $1 \leq k \leq n$ and $s \geq 0$ we can find a coupling $(L'_{\nu_{k,s}}(\bbL_k),L''_{s}(\bbL_k))$ of $L_{\nu_{k,s}}(\bbL_k)$ and $L_{s}(\bbL_k)$ such that 
\begin{equation}
\bigg\|\frac{L'_{\nu_{k,s}}(\bbL_k) - s}{\sqrt{s}}
- \frac{L''_{s}(\bbL_k) - s}{\sqrt{s}} \bigg\|_\infty \lto 0 \,,
\end{equation}
in probability as $s \to \infty$ followed by $k \to \infty$. But then, by doing a Taylor expansion around $s$, under the same limits we have
\begin{equation}
\Big\| \sqrt{L'_{\nu_{k,s}}}(\bbL_k) - \sqrt{L''_{s}}(\bbL_k) \Big\|_\infty  
\lto 0 \,,
\end{equation} 
in probability, so that~\eqref{e:2.19b} holds.

At the same time, for $\nu_{k,s}$ chosen in this way we have, in light of Lemma~\ref{l:101.5} and~\eqref{e:203.1}, that for any $\epsilon \in (0,1)$ and $s$ large,
\begin{equation}\label{e:9.40}
\begin{split}
\bbP \Big(\Big|\sqrt{2^{-(n+1)} \bfL^{-1}_{n, \nu_{k,s}}(0)} - \sqrt{\theta_s}\Big| > \epsilon \Big) & = 
\bbP \Big(\Big|\sqrt{\wh{R}_{n, \nu_{k,s}}} - \sqrt{2\theta_s}\Big| > \sqrt{2} \epsilon \Big) \\
& \leq
\bbP (\theta_s \leq s/2) + 
\bbP \Big(\big|\wh{R}_{n,\nu_{k,s}} - 2\theta_s\big|  > \sqrt{s} \epsilon \Big)\,.% \\
%& \leq \bbP (\theta_s \leq s/2) + 
%\bbP \Big(\big|R_{n,\nu_{k,s}} - 2\wh{S}_{k,\nu_{k,s}}\big|  > \sqrt{s} \epsilon \Big) + \bbP \Big(\big|2\wh{S}_{k,\nu_{k,s}}-2\theta_s\big|  > \sqrt{s} \epsilon \Big)\\
%\leq C \big(\rme^{-C' s} + \epsilon^{-2} 2^{-k}\big) \,,
\end{split}
\end{equation} By the standard Gaussian tail estimate, we have that $\bbP (\theta_s \leq s/2) \leq \rme^{-C's}$ for some $C'>0$ and all sufficiently large $s$. On the other hand, by the union bound and Lemmas~\ref{l:101.5} and~\ref{l:9.4} we have for all large enough $s$,
\begin{equation}\label{e:9.41}
\begin{split}
\bbP \Big(\big|\wh{R}_{n,\nu_{k,s}} - 2\theta_s\big|  > \sqrt{s} \epsilon \Big)& \leq 
\bbP \Big(\big|\wh{R}_{n,\nu_{k,s}} - 2\wh{S}_{k,\nu_{k,s}}\big|  > \frac{\sqrt{s} \epsilon}{2} \Big)+
\bbP \Big(\big|\wh{S}_{k,\nu_{k,s}} - \theta_s\big|  > \frac{\sqrt{s} \epsilon}{4} \Big)\\
& \leq C\epsilon^{-2}\Big(\Big(\frac{2^{-k}}{s} + \frac{2^{-(n-k)}}{\sqrt{s}}\Big)\bbE \wh{S}_{k,\nu_{k,s}} + \frac{1}{\sqrt{s}}\bbE(\wh{S}_{k,\nu_{k,s}} - \theta_s)\Big)\\
& \leq C\epsilon^{-2}\Big(\Big(\frac{2^{-k}}{s} + \frac{2^{-(n-k)}}{\sqrt{s}}\Big) (s+1) + \frac{1}{\sqrt{s}}\Big)
\end{split}
\end{equation}
which goes to $0$ as $s \to \infty$ followed by $k \to \infty$ uniformly in $n \geq \sqrt{s}$, so that the same limit holds for the probability on the left hand side of~\eqref{e:9.40}. Notice that to obtain the upper bound on $\bbE \wh{S}_{k,\nu_{k,s}}$ in~\eqref{e:9.41} we  first conditioned on the value of $\xi$ and then used the bound in Lemma~\ref{l:9.4}. 

Finally, observe that since $\sqrt{\theta_s} - \sqrt{s} + \xi \lto 0$ almost surely as $s \to \infty$, we can replace $\sqrt{\theta_s}$ in the left hand side of~\eqref{e:9.40} by $\sqrt{s} - \xi$, with the limit still holding. This shows~\eqref{e:202.23} and completes the proof.
\end{proof}

\section{Construction of the negatively correlated DGFFs: Proof of Proposition~\ref{p:2.6}} 
\label{s:10}
This section includes the proof of Proposition~\ref{p:2.6}, which is the main ingredient in the proof of Theorem~\ref{t:1.3}. 
\begin{proof}[Proof of Proposition~\ref{p:2.6}]
We will show that $\bfh$ exists by constructing it ``by hand''. To describe the construction, let $\cE$ be the set of edges in $\bbT$ and for $n \geq 1$, let also $\cE_n$ be the subset of $\cE$, which includes all edges leading from a vertex in $\bbL_{n-1}$ to a vertex in $\bbL_n$. Adapting this notation to $\bbT^1$ and $\bbT^2$ in the obvious way, suppose that we can construct a centered Gaussian field $\omega = (\omega(e) :\: e \in \cE^1 \cup \cE^2)$ whose covariances satisfy
\begin{equation}
\label{e:401.16}
\bbE\, \omega(e) \omega(f) =
\begin{cases}
\tfrac12	& \text{if }  e = f \,, \\
-2^{-(n+1)}  & \text{if }  e \in \cE^1_n ,\, f \in \cE^2_n \,, n \geq 1 \\
0 			& \text{otherwise.}
\end{cases}
\end{equation}
We can then define $\bfh(x^1)$ on $x^1 \in \bbT^1$ to be the sum of the values under $\omega$ of the edges leading from the root of $\bbT^1$ to $x^1$ and similarly for $x^2 \in \bbT^2$. Then, it is not difficult to see that $\bfh$ as constructed satisfies both~\eqref{e:301.11} and~\eqref{e:301.12}.

This leaves the task of showing that $\omega$ exists. To do this, we define yet another centered Gaussian field $\sigma = (\sigma(e) :\: e \in \cE)$. We set the variance of $\sigma(e)$ for $e \in \cE_k$ and $k \geq 1$ to be $2^{(k-2) \vee 0}$. Furthermore, if $e, e' \in \cE_k$ are attached to the same vertex, then $\sigma(e) = -\sigma(e')$, so that $\bbE \sigma(e) \sigma(e') = -2^{(k-2) \vee 0}$. For all other pairs of edges $e$, $e'$ we assume that $\sigma(e)$ and $\sigma(e')$ are independent of each other. Now, for $x,y \in \bbL_n$ with $n \geq 1$, the reader can easily verify that the covariance between the sum of $\sigma$ along the edges on the path from the root to $x$ with the corresponding sum along the edges on the path from the root to $y$ is equal to
\begin{equation}
\label{e:301.17}
\begin{cases}
2^{n-1} & \text{if } |x \wedge y| = n \,, \\
0 & \text{if } 0 < |x \wedge y| < n \,,\\
-1 & \text{if } |x \wedge y| = 0 \,.
\end{cases}
\end{equation}

Now letting $z^1$, $z^2$ be the two children of the root of $\bbT$, we identify $\bbT(z^1)$ and $\bbT(z^2)$ with $\bbT^1$ and $\bbT^2$ respectively, so that for all $n \geq 2$, we have $\cE_n = \cE^1_{n-1} \cup \cE^2_{n-1}$. We also take an i.i.d. sequence of fields $(\sigma_n :\: n \geq 1)$ all having the same law as $\sigma$. Then we define $\omega(e)$ for $e \in \cE_n$ with $n \geq 2$, as $2^{-n/2}$ times the sum under $\sigma_n$ of the path leading from the root of $\bbT$ to the vertex in $\bbL_n$ to which $e$ leads. It is not difficult to verify using~\eqref{e:301.17} that such $\omega$ satisfies~\eqref{e:401.16} as required.

To show~\eqref{e:301.15} we let $Z^1_n$ and $Z^2_n$ be the derivative martingales at depth $n \geq 0$ for the restrictions of $\bfh$ to $\bbT_n^1$ and $\bbT_n^2$ respectively.
We also let $\xi, \xi^1, \xi^2 \sim \cN(0,1/2)$ be independent of each other and of $\bfh$ and set $\xi^1_n := \sqrt{1-2^{-n}} \xi + \sqrt{2^{-n}} \xi^1$ and $\xi^2_n := \sqrt{1-2^{-n}} \xi + \sqrt{2^{-n}} \xi^2$. Now,
\begin{multline}
\label{e:301.16}
2^{-2n} \sum_{x \in \bbL^1_n}
\Big(\sqrt{\log 2}\, n - \big(\bfh(x) + \xi^1_n\big)\Big) \rme^{2 \sqrt{\log 2}\, (\bfh(x) + \xi^1_n)}  \\
= \rme^{2 \sqrt{\log 2} \xi^1_n} \Big(
2^{-2n} \sum_{x \in \bbL^1_n}
\big(\sqrt{\log 2}\, n - \bfh(x) \big) \rme^{2 \sqrt{\log 2}\, \bfh(x)}
- \xi^1_n 2^{-2n} \sum_{x \in \bbL^1_n}
\rme^{2 \sqrt{\log 2}\, \bfh(x)} \Big) \,.
\end{multline}
Since $\bfh$ restricted to $\bbT^1$ has the same law as that of $h$ on $\bbT$, it follows from~\eqref{e:1.8} and~\eqref{e:1.9}, that the first term in the parenthesis above converges almost surely to a random variable $Z^1$, which has the same law as $Z$. 

At the same time, it is well known that the critical exponential martingale (without the ``derivative term'') vanishes in the limit (c.f.~\cite{lyons1997simple}), namely
\begin{equation}
\lim_{n \to \infty} 2^{-2n} \sum_{x \in \bbL_n} \rme^{2 \sqrt{\log 2}\, h(x)} = 0
\quad \text{a.s.}
\end{equation}
It follows that the second term in the parenthesis of~\eqref{e:301.16} tends to $0$ almost surely.
Lastly, since $\xi^1_n \to \xi$ as $n \to \infty$, we obtain altogether that the left hand side of~\eqref{e:301.16} tends almost-surely to $\rme^{2\sqrt{\log 2} \xi} Z^1$.
Replacing $\bbL^1_n$ and $\xi^1_n$ with $\bbL^2_n$ and $\xi^2_n$ in~\eqref{e:301.16}, we may apply the same reasoning to obtain the almost convergence to $\rme^{2\sqrt{\log 2} \xi} Z^2$, where $Z^2$ is the limit of the derivative martingale with respect to $\bfh$ restricted to $\bbT^2$. 
Combining both claims we get
\begin{multline}
\label{e:301.18}
2^{-2n} \Big(\sum_{x \in \bbL^1_n}
\Big(\sqrt{\log 2}\, (n + 1) - \big(\bfh(x) + \xi^1_n\big)\Big) \rme^{2 \sqrt{\log 2}\, (\bfh(x) + \xi^1_n)} + \\
\sum_{x \in \bbL^2_n}
\Big(\sqrt{\log 2}\, (n + 1) - \big(\bfh(x) + \xi^2_n\big)\Big) \rme^{2 \sqrt{\log 2}\, (\bfh(x) + \xi^2_n)}\Big)
\lto \rme^{2\sqrt{\log 2} \xi} \bfZ \,,
\end{multline}
almost surely as $n \to \infty$, where $\bfZ = Z^1 + Z^2$ as in~\eqref{e:301.13}. 

Next, in view of~\eqref{e:301.11} and~\eqref{e:301.12}, we observe that 
\begin{equation}
\label{e:301.19}
\Big(\bfh(x) +\xi^1_n 1_{\bbL^1_n}(x) + \xi^2_n 1_{\bbL^2_n}(x) :\: x \in \bbL^1_n \cup \bbL^2_n \Big)
\end{equation}
is a centered Gaussian vector with the covariance between its values at $x$ and $y$ given by
\begin{equation}
\begin{cases}
\bbE \bfh(x^1) \bfh(y^1) + \bbE (\xi_n^1)^2 = 
\tfrac12\big(|x^1 \wedge y^1| + 1) & \text{if } x=x^1 \in \bbL^1_n ,\, y=y^1 \in \bbL^1_n \,, \\
\bbE \bfh(x^2) \bfh(y^2) + \bbE (\xi_n^2)^2 = 
\tfrac12\big(|x^2 \wedge y^2| + 1) & \text{if } x=x^2 \in \bbL^2_n ,\, y=y^2 \in \bbL^2_n \,, \\
\bbE \bfh(x^1) \bfh(y^2) + \bbE \xi_n^1 \xi_n^2 =
0 & \text{if } x=x^1 \in \bbL^1_n ,\, y=y^2 \in \bbL^2_n \,,
\end{cases}
\end{equation}
where $|x^1 \wedge y^1|$ is the depth of the common ancestor of $x^1$ and $y^1$ in $\bbT^1$ and similarly for $|x^2 \wedge y^2|$. 

We now identify as before $\bbL_{n}(z^1)$ and $\bbL_{n}(z^2)$ in $\bbT$ with
$\bbL^1_n$ in $\bbT^1$ and $\bbL^2_n$ in $\bbT^2$ respectively, so that $\bbL_{n}(z^1) \cup \bbL_{n}(z^2) =  \bbL_{n+1}$. Then, the covariances above are equal in all cases to $|x \wedge y|/2$, where $|x \wedge y|$ is the depth of the common ancestor of $x,y \in \bbL_{n+1}$ in $\bbT$. This shows that the Gaussian vector in~\eqref{e:301.19} has exactly the same law as that of $(h(x) :\: x \in \bbL_{n+1})$. In view of~\eqref{e:1.8} and~\eqref{e:1.9}, we see that the left hand side in~\eqref{e:301.18} tends in law to $4 Z$. But then in conjunction with the full statement in~\eqref{e:301.18} this gives
\begin{equation}
\rme^{-2 \log 2 + 2\sqrt{\log 2} \xi} \, \bfZ \eqd Z \,.
\end{equation}
Since the law of $\Lambda: = \rme^{-2 \log 2 + 2\sqrt{\log 2} \xi}$ is $\text{Log-normal}(-2 \log 2, 2 \log 2)$ and $\xi$ is independent of $\bfZ$, this shows~\eqref{e:301.15}. Then, positivity and almost-sure finiteness of $\bfZ$ follows immediately, since the same applies to $\Lambda$ and $Z$.
\end{proof}

\appendix
\section{Appendix: Proofs of preliminary statements} 
\label{s:A}
This appendix includes proofs for various preliminary statements from Section~\ref{s:3}. As these proofs are rather standard in the subject, we allow ourselves to be brief.

\subsection{DGFF preliminaries}
\label{a:1}
\begin{proof}[Proof of Proposition~\ref{p:3.3}]
Fix $u \geq 0$ and $\eta \in (0,1/2)$ and set $\wh{h}_n^* := \min \wh{h}_n(\ol{\bbL}_n)$.
For any $v > \sqrt{u}$, we can write
\begin{multline}
\label{e:4.16}
\bbE \big(\big|\cG_n(u)\big| :\: \wh{h}_n^* \geq -v \big)
\\ = \sum_{x \in \ol{\bbL}_n} \int_{-\sqrt{u}}^{\sqrt{u}}
\bbP \big(\wh{h}_n(x) \in \rmd w \big) 
\bbP \Big(\min_{k \in (0,n)} \big(\wh{h}_n([x]_k) + \wh{h}_{n-k}^{(n-k)*} - m_{n-k}\big) \geq -v \,\Big|\, \wh{h}_n(x) = w \Big)\,,
\end{multline}
where $\wh{h}_k^{(k)}$ are independent copies of the DGFF on $\ol{\bbT}_k$ for $k \in (0,n)$.

Recall that $\wh{h}_n(x) = h_n(x) + m_n$ and that $\big(h_n([x]_k) :\: k \in [0,n] \big)$ is a random walk with centered Gaussian steps. Tilting by $k \mapsto -m_n (k/n)$ and setting $m_{n,k} := m_n(k/n) - m_k$, the second probability above is therefore equal to
\begin{equation}
\label{e:4.17a}
\bbP \Big( \min_{k \in (0,n)} \big(h_n([x]_k) + \wh{h}_{n-k}^{(n-k)*} + m_{n,n-k} \big) \geq -v \,\Big|\, h_n(x) = w\Big)\,.
\end{equation}
Now $\wh{h}_k^{(k)*}$ are exponentially tight for all $k \geq 1$~\cite[Lemma 2.1]{roy2018branching} and 
$m_{n,k} \leq 1+ \log \wedge_n(k)$~\cite[Lemma~3.3]{CHL17}.  It then follows from considerations similar to that in the proof of Lemma 2.7 in~\cite{CHL17Supp} and Proposition 1.5 in~\cite{CHL17} that 
\begin{equation}
\label{e:4.18}
\lim_{r \to \infty} \limsup_{n \to \infty}
\bbP \Big( \min_{k \in [r,n-r]} h_n([x]_k) \leq 0
\,\Big|\, h_n(x) = w , \min_{k \in (0,n)} \big(h_n([x]_k) + \wh{h}_{n-k}^{(n-k)*} + m_{n,n-k} \big) \geq -v \Big) = 0
\end{equation}
uniformly in $|w| \leq \sqrt{u}$.

At the same time, it follows from Theorem~2 in~\cite{ritter1981growth} and stochastic monotonicity w.r.t. boundary conditions of a random walk conditioned to stay positive~\cite{bramsonBBM} that
\begin{equation}
\label{e:4.20}
\lim_{r' \to \infty} \limsup_{n \to \infty}
\bbP \Big( \min_{k \in [r', n/2]} \big(h_n([x]_k) - k^{1/2-\eta}\big) \leq 0
\,\Big|\, \min_{k \in [r, n/2]} h_n([x]_k) > 0 \Big) = 0\,.
\end{equation}
To convert this into a bridge estimate, for any $\epsilon > 0$, we can use stochastic monotonicity again to lower bound the last probability by  
\begin{multline}
\bbP \Big( \min_{k \in [r', n/2]} \big(h_n([x]_k) - k^{1/2-\eta}\big) \leq 0
\,\Big|\, h_{n}([x]_{n/2}) = \epsilon \sqrt{n},  \min_{k \in [r, n/2]} h_n([x]_k) > 0 \Big)  \\
\times \bbP \Big( h_n([x]_{n/2}) \leq \epsilon \sqrt{n}  \,\Big|\,  \min_{k \in [r, n/2]} h_n([x]_k) > 0 \Big) \,.
\end{multline}
Since standard random walk estimates show that the second probability stays uniformly bounded away from $0$ for all $n \geq 1$, it follow from~\eqref{e:4.20} that
the first probability in the last display must also go to zero in the limit when $n \to \infty$ followed by $r' \to \infty$ for any fixed $\epsilon > 0$.

But then by monotonicity again,
\begin{multline}
\bbP \Big( \min_{k \in [r', n-r']} \big(h_n([x]_k) - \wedge_n(k)^{1/2-\eta}\big) \leq 0
\,\Big|\, h_n(x) = w\,, \min_{k \in [r,n-r]} h_n([x]_k) > 0  \Big)
\leq \\
\bigg(\bbP \Big( \min_{k \in [r', n]} \big(h_n([x]_k) - k^{1/2-\eta}\big) \leq 0
\,\Big|\, h_{n}([x]_{n/2}) = \epsilon \sqrt{n},  \min_{k \in [r, n/2]} h_n([x]_k) > 0 \Big) \bigg)^2 \\
+
\bbP \Big( h_n([x]_{n/2}) \leq \epsilon \sqrt{n} \,\Big|\, h_n(x) = w \,,\, \min_{k \in [r,n-r]} h_n([x]_k) > 0 \, \Big) \,.
\end{multline} 
Now, similar random walk estimates show the the second probability on the right hand side tend to $0$ as $n \to \infty$ followed by $\epsilon \to 0$. Consequently, if we take $n \to \infty$ followed by $r' \to \infty$ and finally $\epsilon \to 0$, the right hand side will vanish, yielding that 
\begin{equation}
\label{e:4.22}
\lim_{r' \to \infty} \limsup_{n \to \infty} \bbP \Big( \min_{k \in [r', n-r']} \big(h_n([x]_k) - \wedge_n(k)^{1/2-\eta}\big) \leq 0
\,\Big| h_n(x) = w\,, \min_{k \in [r,n-r]} h_n([x]_k) > 0  \Big) = 0 \,,
\end{equation}
uniformly in $|w| \leq \sqrt{u}$.

To handle deviations above $\wedge_n(k)^{1/2+\eta}$, we observe that conditional on $h_n(x) = w$, 
the law of $h_n([x]_l)$ is Gaussian with mean $w l/n$ and variance $l(n-l)/(2n) \leq \wedge_n(l)/2$. Therefore for $l \in [r, n/2]$,
\begin{multline}
\bbP \Big( h_n([x]_l) > l^{1/2+\eta}, \min_{k \in [r,n-r]} h_n([x]_k) > 0 \,\Big|\, h_n(x) = w\Big) \\
\leq \bbP \Big( h_n([x]_l) > l^{1/2+\eta} \,\big|\, h_n(x) = w \Big)
\bbP \Big(\min_{k \in [l,n-r]} h_n([x]_k) > 0 \,\Big|\ h_n([x]_l) = l^{1/2+\eta}, h_n(x) = w\Big) 
\end{multline}
which is at most $C \rme^{-l^{2\eta}} l^{1/2+\eta} n^{-1}$ for $C > 0$ depending on $u$ and $r$. Above we have used monotonicity, the Gaussian tail formula and standard estimates for a random walk conditioned to stay positive. Since 
\begin{equation}
\bbP \big(\min_{k \in [r,n-r]} h_n([x]_k) > 0 \,\Big|\, h_n(x) = w\Big) > C' n^{-1} \,,
\end{equation}
for some $C' >0$, we get for all $l$ large,
\begin{equation}
\bbP \Big( h_n([x]_l) > l^{1/2+\eta} \,\Big|\,  \min_{k \in [r,n-r]} h_n([x]_k) > 0 ,\, h_n(x) = w\Big) \leq \rme^{-l^{\eta}} \,.
\end{equation}
A symmetric argument gives the same bound for the conditional probability of $\{h_n([x]_{n-l}) > l^{1/2+\eta}\}$. 
Invoking the union bound, we then get
\begin{equation}
\bbP \Big( \max_{k \in [r', n-r']} \big(h_n([x]_k) - \wedge_n(k)^{1/2+\eta}\big) \geq 0 \,\Big|\, 
\min_{k \in [r,n-r]} h_n([x]_k) > 0 ,\, h_n(x) = w\Big) 
\leq \sum_{l=r'}^{n-r'} \rme^{-\wedge_n(l)^{\eta}}  \,,
\end{equation}
which is at most $C \rme^{-r'^{\eta/2}}$ and therefore the probability on the left hand side goes to $0$ as $n \to \infty$ followed by $r' \to \infty$ uniformly in $|w| \leq \sqrt{u}$ as well.

Combining this with~\eqref{e:4.22} and~\eqref{e:4.18} and using monotonicity and the union bound, the conditional probability
\begin{equation}
\label{e:4.26}
\bbP \Big( \exists k \in [r, n-r] :\: h_n([x])_k \notin \frR_{\wedge_n(k)}^{\eta}\big)  
\,\Big|\, h_n(x) = w \,, \min_{k \in (0,n)} \big(h_n([x]_k) + \wh{h}_{n-k}^{(n-k)*} + m_{n,n-k} \big) \geq -v \Big)\,,
\end{equation}
must go to $0$ in the limit when $n \to \infty$ followed by $r \to \infty$. Tilting back by $k \mapsto m_n k/n = \sqrt{\log 2} k + O(\log k)$, the last probability is at least
\begin{multline}
\bbP \Big( \exists k \in [r, n-r] :\: \wh{h}_n([x])_k \notin \sqrt{\log 2}(n-k) + \frR_{\wedge_n(k)}^{2\eta}\big)  \\
\,\Big|\, 
\wh{h}_n(x) = w \,, \min_{k \in (0,n)} \big(\wh{h}_n([x]_k) + \wh{h}_{n-k}^{(n-k)*} - m_{n-k} \big) \geq -v \Big)\,.
\end{multline}
Assuming that we started with $\eta/2$ in the first place, we can then plug this back in~\eqref{e:4.16} and get
\begin{equation}
\label{e:4.29}
\bbE \big(\cH_n^{[r, n-r], \eta}(u) :\: \wh{h}_n^* \geq -v \big) = 
\bbE \big(\cG_n(u) :\: \wh{h}_n^* \geq -v \big) o(1) \,,
\end{equation}
with the $o(1)$ term tending to $0$ when $n \to \infty$ followed by $r \to \infty$.

Finally, we use Proposition 1.1 in~\cite{CHL17Supp}, to bound the probability in~\eqref{e:4.17a} by
$C v(\sqrt{u})/n$. Plugging this in~\eqref{e:4.16} and using that $\wh{h}_n(x)$ is a Gaussian with mean $m_n$ and variance $n/2$, we get 
\begin{equation}
\bbE \big(\cG_n(u) :\: \wh{h}_n^* \geq -v \big) \leq C' 2^n v^2 n^{-3/2} \rme^{-m_n^2/n} \leq C'' v^2 \,,
\end{equation}
with all constants depending on $u$. We then use Markov's inequality and the union bound to write
\begin{equation}
\bbP \Big(\cH_n^{[r, n-r], \eta}(u) \neq \emptyset \Big) \,\leq\, 
\bbP \big(\wh{h}^*_n < -v\big) + \bbE \big(\cH_n^{[r, n-r], \eta}(u) :\: \wh{h}_n^* \geq -v \big) 
\,\leq\, \bbP \big(\wh{h}^*_n < -v\big) + C'' v^2 o(1) \,,
\end{equation}
with the $o(1)$ term as in~\eqref{e:4.29}.
In view of Proposition~\ref{p:3.1}, the right hand side vanishes as $n \to \infty$ followed by $r \to \infty$ and finally $v \to \infty$. This implies the same for the left hand side.
\end{proof}

\begin{proof}[Proof of Proposition~\ref{l:3.5a}]
Let $\cM_0(\bbR)$ denote the space of all boundedly finite measures on $\bbR$ and for $u > 0$, define
\begin{equation}
A_u := \Big\{ (x, \omega) \in \bbR \times \cM_0(\bbR) :\: \omega\big(-x + \big[-\sqrt{u}, \sqrt{u}\big]\big) > 0 \Big\}
\end{equation}
Then $|[\cG_n(u)]_{n-r}| = \chi_{n,r}(A_u)$.
Denoting by $\chi_\infty$ the limiting process in~\eqref{e:c}, conditional on $\bar{Z}$ the intensity measure of $A_u$ is given by
\begin{equation}
\label{e:104.23}
C_\diamond \bar{Z} \int_{-\infty}^{\sqrt{u}} \rmd x \rme^{2 \sqrt{\log 2}x} \int 1_{\{\omega(-x + [-\sqrt{u}, \sqrt{u}]) > 0\}} \nu(\rmd \omega) \,.
\end{equation}
Setting $C_u$ to be the integral above, which is finite as the inner integral is bounded by $1$, 
it follows that $\chi_\infty(A_u)$ has a law as in the right hand side of~\eqref{e:104.21}. It is therefore sufficient to show the convergence in law of $\chi_{n,r}(A_u)$ to $\chi_\infty(A_u)$ under the stated limits. 

To this end, it is not difficult to see that
\begin{equation}
\partial A_u := \Big\{ (x, \omega) \in \bbR \times \cM_0(\bbR) :\: \omega\big(-x + \{\sqrt{u}, \sqrt{u}\} \big) \geq 1 \Big\} \,.
\end{equation}
An expression for the conditional intensity measure of $\partial A_u$ can be written as in~\eqref{e:104.23}. Using Fubini to exchange the order of integrals and observing that $\omega$ charges countably many points, $\nu$-almost surely, we get that the double integral must be zero and hence that $\partial A_u$ is a stochastic continuity set of $\chi_\infty$.

The trouble is that $A_u$ is not a bounded set. To remedy this, for any $\delta > 0$, we consider the processes $\chi_{n,r}^\delta$ and $\chi_\infty^\delta$, obtained from $\chi_{n,r}$ and $\chi_\infty$ respectively, by restricting the first coordinate to $[-\delta^{-1}, \delta^{-1}]$ and the second to finite measures on $[0, 2\delta^{-1}]$. For $\epsilon > 0$ we then define also 
\begin{equation}
B^\delta_{\epsilon} := \Big\{ (x, \omega) :\: |x| \leq \delta^{-1} , \   \omega\big([0, 2\delta^{-1}]\big) \leq \epsilon^{-1} \Big\}
\end{equation}
Then $A_u \cap B^\delta_{\epsilon}$ is a bounded stochastic continuity set under $\chi_\infty^\delta$
and therefore the law of its mass under $\chi^\delta_{n,r}$ converges to that under $\chi^\delta_\infty$. It therefore remains to verify that the random variables
\begin{equation}
\label{e:104.27}
\chi_{n,r} (A_u \setminus B_{\epsilon}^{\delta})
\quad , \qquad
\chi_{\infty} (A_u \setminus B_{\epsilon}^{\delta}) 
\end{equation}
tend to $0$ in probability as $\epsilon \to 0$ followed by $\delta \to 0$, and in the case of 
$\chi_{n,r} (A_u \setminus B_{\epsilon}^{\delta})$, that this limit is uniform in $n \in [1,\infty)$ and $r \geq 1$.

Indeed, the probability that the first quantity in~\eqref{e:104.27} is not zero is bounded from above by
\begin{equation}
\bbP \big(\max_{x \in \ol{\bbL}_n} |\wh{h}_n(x)| > \delta^{-1}\big)
+ \bbP \big(\big|\cG_n(2\delta^{-1/2})\big| > \epsilon^{-1} \big)  \,.
\end{equation}
But then by Proposition~\ref{p:3.1} and Proposition~\ref{p:3.2}, both terms must go to $0$ uniformly in $n$ and $r$, when $\epsilon \to 0$ followed by $\delta \to 0$. Since a similar argument shows that the same holds for the second quantity in~\eqref{e:104.27}, the proof is complete.
\end{proof}

\subsection{Soft entropic repulsion of local time trajectories} 
\label{ss:A2}
In this subsection we prove Proposition~\ref{p:8.3}. The first step is to show that, with high probability, every vertex in $\ol{\bbL}_n$ has a local time trajectory which is not too low. To make this more precise, for each $n \geq 1$ and $k=1,\dots,n$, let us define with $\eta'$ as in the proposition,
\begin{equation}\label{e:defa}
\alpha_n(k):= \sqrt{\log 2}(n-k) - r'_n
\quad; \qquad
r'_n := \lceil n^{\eta'} \rceil \,.
\end{equation} 
The statement we wish to prove first is contained in the following lemma.

\begin{lem}\label{l:8.2} 
For any $\eta' > 0$,
\begin{equation}\label{e:8.11}
\lim_{n \to \infty} \bbP\Big( \exists x \in \ol{\bbL}_n \,:\, \exists k \in [1,n-2r'_n]\,,\,  \sqrt{L_{t^{\rmA}_n}([x]_k)} \leq \alpha_n(k)\Big) = 0.
\end{equation}	
\end{lem} 

\begin{proof}[Proof] By using the union bound first over $k \in [1,n-2r'_n]$ and then over all the vertices in $\ol{\bbL}_k$, given any fixed $x \in \ol{\bbL}_n$ we can bound the probability in~\eqref{e:8.11} from above by 
\begin{equation} \label{eq:bound1}
\sum_{k=1}^{n-2r'_n} 2^k\bbP\left(\sqrt{L_{t^{\rmA}_n}([x]_{k})} \leq \alpha_n(k)\right)\,.
\end{equation}  

To give an upper bound for the probability in \eqref{eq:bound1}, we can use Lemma~\ref{l:bp} to write, for any $k=1,\dots,n-2r'_n$,
\begin{equation}
\bbP\left( \sqrt{L_{t^{\rmA}_n}([x]_k)} \leq \alpha_n(k)\right) = \bbP( Y_k \leq \alpha^2_n(k))  = \exp\left(- \frac{t^{\rmA}_n}{k}\right) + \int_0^{\alpha^2_n(k)} f_k(y)dy\,, \label{eq:den}
\end{equation} 
where $f_k$ is as in the lemma.
Since $t^{\rmA}_n\geq (\log 2)n^2 - \frac{3}{2}n\log n$, we have that
\begin{equation}\label{eq:den0}
\exp\left(-\frac{t^{\rmA}_n}{k}\right) \leq	\exp\left( -(\log 2)n + \frac{3}{2}\log n\right) = n^{\frac{3}{2}}2^{-n}.
\end{equation} 
For the second term in~\eqref{eq:den}, since $\sqrt{t^{\rmA}_n} \leq n$ and $k \geq 1$, we can use~\eqref{e:bp1} to upper bound it by
\begin{equation}
\label{eq:den2}
C n \int_{0}^{\alpha^2_n(k)} \frac{1}{2 \sqrt{\pi k y}} \exp\left(-\frac{(\sqrt{y}-\sqrt{t^{\rmA}_n})^2}{k}\right)dy, 
\end{equation} 
for some $C>0$. 
Now, let $\Phi(\cdot\,;\mu\,,\sigma^2)$ 
denote the cumulative distribution function associated with a Gaussian random variable of mean $\mu$ and variance $\sigma^2$. Since the function $y \mapsto \Phi(\sqrt{y}\,;\sqrt{t^{\rmA}_n}\,,k/2)$ defined for $y\geq 0$ is the anti-derivative of the integrand in \eqref{eq:den2}, the integral is at most
\begin{equation}
\Phi\left(\alpha_n(k)\,;\sqrt{t^{\rmA}_n}\,,k/2\right)\\
\leq \Phi\left(-\sqrt{2}\left( \sqrt{(\log2)k} + \frac{r'_n}{2\sqrt{k}}\right)\,;0\,,1\right)\,,
\end{equation} 
for all $n$ sufficiently large.
Then, by the standard Gaussian tail estimate $\Phi(-y\,;0\,,1) \leq e^{-\frac{y^2}{2}}$, the latter is at most $2^{-(k+r'_n)}$, 
for all large enough $n$. Combining this with~\eqref{eq:den2} and~\eqref{eq:den0} gives the estimate
\begin{equation}
\bbP\left( \sqrt{L_{t^{\rmA}_n}([x]_k)} \leq \alpha_n(k)\right) \leq n^{\frac{3}{2}}2^{-n} + C n 2^{-(k+r'_n)}\,.
\end{equation} Using this bound, a straightforward computation shows that \eqref{eq:bound1} is bounded by $C' n^2 2^{-2r'_n}$ for some $C'>0$ and all sufficiently large $n$, from where the lemma now immediately follows.		
\end{proof}

We now complete the proof of Proposition~\ref{p:8.3}. 
\begin{proof}[Proof of Proposition~\ref{p:8.3}]
For each $n \geq 1$ and $x \in \ol{\bbL}_n$, take $\eta'<\eta$ in~\eqref{e:defa}, with $\eta$ from~\eqref{e:2.3} and for $\alpha_n(k)$ as in~\eqref{e:defa} define the event  
\begin{equation}
A_n(x):=\left\{ \forall k \in [1,n-2r'_n]\,,\, \sqrt{L_{t^{\rmA}_n}([x]_k)}> \alpha_n(k)\right\}\,.
\end{equation} Observe that $A_n:=\cap_{x \in \ol{\bbL}_n} A_n(x)$ is the complement of the event in Lemma~\ref{l:8.2}.

Now, by the union bound and Markov's inequality, for any $\delta > 0$ we have that the probability in~\eqref{e:8.3b} is bounded from above by 
\begin{equation}\label{e:decomp}
\bbP(A_n^c) + \frac{1}{\delta \sqrt{n}}\bbE\left(\left| \cF_n \setminus \mathcal{O}^{[r_n,n-r_n], \eta'}\right| \,;\, A_n\right)\,.
\end{equation} Thus, by Lemma~\ref{l:8.2}, in order to obtain~\eqref{e:8.3b} it will suffice to show that the second term in~\eqref{e:decomp} tends to zero as $n \to \infty$. To this end, we first bound the expectation in~\eqref{e:decomp} from above by 
\begin{equation}\label{e:8.27b}
\sum_{x \in \ol{\bbL}_n} \sum_{k=r_n}^{n-r_n} \bbP\left( \{L_{t^{\rmA}_n}(x)=0 \} \cap \Big\{\sqrt{L_{t^{\rmA}_n}([x]_k)} < \sqrt{\log 2}(n-k)+r'_n\Big\} \cap A_n(x)\right)\,.
\end{equation} Now, fix any $x \in \ol{\bbL}_n$ and for each $k' \in [1,n]$ let us abbreviate $L(k'):=L_{t^{\rmA}_n}([x]_{k'})$ for simplicity. Then, by conditioning on $L(k)$, the $k$-th term of the second sum can be expressed as 
\begin{equation}\label{e:8.24}
\int_{(\sqrt{\log 2}(n-k)-r'_n)^2}^{(\sqrt{\log 2}(n-k)+r'_n)^2} \varphi_1(w)\varphi_2(w)\bbP\left(L(k) \in \rmd w\right)\,, 
\end{equation} where 
\begin{equation}\label{e:8.29b}
\varphi_1(w):= \bbP\Big( L(n) =0 \,,\,\min_{k'=k+1,\dots,n-2r'_n} (\sqrt{L(k')} -\alpha_n(k'))>0 \,\Big|\, L(k) = w\Big)
\end{equation} and
\begin{equation}
\varphi_2(w):=\bbP\Big( \min_{k'=1,\dots,k-1} (\sqrt{L(k')}-\alpha_n(k'))>0\,\Big|\,L(k) = w\Big).
\end{equation}
(Note that $2r'_n < r_n$ for all $n$ large, so that the minimum in~\eqref{e:8.29b} is over a nonempty set of $k'$.)

By Lemma~\ref{l:3.9}, we can bound $\varphi_2(w)\bbP(L(k) \in \rmd w)$ from above by  
\begin{multline}\label{e:8.27}
\left(\frac{t^{\rmA}_n}{w}\right)^{\tfrac{1}{4}} \bbP \Big( B_{k'} > \alpha_n(k') \,\forall\, k'=1,\dots,k-1 \,\Big|\, B_0 = \sqrt{t^{\rmA}_n}\, ,\, B_{k} = \sqrt{w} \Big) \\ 
\times \bbP\Big( B^2_k \in \rmd w\,,\,B_k>0 \,\Big|\, B_0 = \sqrt{t^{\rmA}_n}\Big) \,,
\end{multline}
where $(B_s : s \in [0,k])$ is a Brownian motion with variance $\tfrac{1}{2}$ starting from $\sqrt{t^{\rmA}_n}$ at time $0$ (which we write formally as conditioning).
Now, we can rewrite the first conditional probability in~\eqref{e:8.27} as
\begin{equation}\label{e:by1}
\bbP\Big( B_{k'} > (x-x') \big(1-\tfrac{k'}{k}\big) + (y-y') \big(\tfrac{k'}{k}\big)\,,\forall\,k'=1,\dots,k-1 \,\Big|\, B_0 = 0\,,\,B_k = 0\Big) \,,
\end{equation} with 
\begin{equation}
\left\{\begin{array}{l}x:=\sqrt{\log 2}n - r'_n\\
x':=\sqrt{t^{\rmA}_n}\end{array}\right. \hspace{1cm}\text{ and }\hspace{1cm}\left\{\begin{array}{l}y:=\sqrt{\log 2}(n-k) - r'_n\\
y':=\sqrt{w}.\end{array}\right.
\end{equation} 
Since $(x-x') (1-k'/k) + (y-y') (k'/k) \geq -2r'_n$ for all $n$ large enough and any $w$ in the domain of the integral in~\eqref{e:8.24}, it follows from standard Brownian motion estimates (c.f.~\cite{CHL17Supp}), that this conditional probability is at most 
\begin{equation}
C \frac{(r'_n)^2}{k} \end{equation} for some constant $C >0$ and all $n$ sufficiently large, uniformly in $k$. 

On the other hand, we have that
\begin{equation}
\bbP\Big( B^2_k \in \rmd w\,,\,B_k>0 \,\Big|\, B_0 = \sqrt{t^{\rmA}_n}\Big) = \frac{1}{\sqrt{4\pi k w}}\exp\bigg( - \frac{(\sqrt{w}-\sqrt{t^{\rmA}_n})^2}{k}\bigg)\,\rmd w \leq \frac{n^{\frac{3}{2}}2^{-k}}{\sqrt{kw}}\rme^{2\sqrt{\log 2}\hat{w}}\,\rmd w\,,
\end{equation} where $\hat{w}:=\sqrt{w} - \sqrt{\log 2} (n-k) \in (-r'_n,r'_n)$. In particular, for the range of $\sqrt{w}$ we have in~\eqref{e:8.24}, we can bound $\varphi_2(w)\bbP(L(k) \in \rmd w)$ from above by 
\begin{equation}
C(r'_n)^2 \sqrt{n} \left( \frac{n}{(n-k)k}\right)^{\tfrac{3}{2}}2^{-k}\rme^{ 2\sqrt{\log 2}\hat{w}}\,\rmd w \leq C' \sqrt{n} (r'_n)^2  \big(\wedge_n(k)\big)^{-3/2} 2^{-k}\rme^{ 2\sqrt{\log 2}\hat{w}}\,\rmd w
\end{equation} for some constants $C, C'>0$ and all $n$ sufficiently large, uniformly in $k$.

We proceed in a similar fashion to treat $\varphi_1(w)$.  Indeed, by first conditioning on $L(n-2r'_n)$ and then using the Markov property, we can write $\varphi_1(w)$ as 
\begin{equation}\label{e:8.37int}
\int_{\alpha^2_n(n-2r'_n)}^{\infty} \phi_1(s)\phi_2(s)\bbP(L(n-2r'_n) \in \rmd s \,|\, L(k)=w)\,,
\end{equation} 
where
\begin{equation}
\phi_1(s):=\bbP\big( L(n) = 0 \,\big|\, L(n-2r'_n) = s\big)
\end{equation}
and
\begin{equation}\label{e:9.36}
\phi_2(s):= \bbP\Big( \min_{k'=k+1,\dots,n -2r'_n} (\sqrt{L(k')}-\alpha_n(k'))>0\,\Big|\,L(k) = w\,,\,L(n-2r'_n)=s\Big)\,.
\end{equation}

If we write $\hat{s}:=\sqrt{s} - 2\sqrt{\log 2}\,r'_n > -r'_n$ then, by Lemma~\ref{l:bp}, we have that
\begin{equation}
\phi_1(s)=\exp\left( - \frac{s}{2r'_n}\right) = 2^{-2r'_n} \exp\left( -2\sqrt{\log 2}\hat{s}-\frac{\hat{s}^2}{2r'_n}\right)\,.
\end{equation} 
On the other hand, by proceeding exactly as we did for the term $\varphi_2(w)$ (which we \mbox{can do because} on the first event in~\eqref{e:9.36} we have $L(n-2r'_n)>\alpha_n(n-2r'_n)>0$ and hence Lemma~\ref{l:3.9} applies), we see that, for the range of $\sqrt{w}$ under consideration in~\eqref{e:8.24} and all large enough $n$, we can bound the term  $\phi_2(s)\bbP\big(L(n-2r'_n) \in \rmd s\,\big|\, L(k)=\sqrt{w}\big)$ from above (uniformly in $k$) by 
\begin{equation}\label{e:8.42}
\begin{split}
C &\left(\frac{\sqrt{w}}{\sqrt{s}}\right)^{\tfrac{1}{2}} 
\frac{(\hat{w}+r'_n)(\hat{s}+r'_n)}{(n-2r'_n-k)} \frac{1}{\sqrt{(n-2r'_n-k)s}} \exp\left( -\frac{(\sqrt{w}-\sqrt{s})^2}{n-2r'_n-k}\right)\,\rmd s\\
&
\leq C'  \frac{r'_n2^{-(n-2r'_n-k)}}{n-k} (\hat{s}+r'_n) (\sqrt{s})^{-\tfrac 32} \rme^{-2\sqrt{\log 2}(\hat{w}-\hat{s})}\,\rmd s
\end{split}
\end{equation} for some $C,C'>0$. 
In particular, upon performing the change of variables $s \mapsto \hat{s}$ in the integral in~\eqref{e:8.37int}, we obtain that $\varphi_1(w)$ is at most 
\begin{equation}
C'' \frac{r'_n 2^{-(n-k)}}{n-k} \rme^{- 2\sqrt{\log 2}\hat{w}}\int_{-r'_n}^{\infty} \frac{(\hat{s}+r'_n)}{\sqrt{ \hat{s}+2\sqrt{\log 2}\,r'_n}}\rme^{-\frac{\hat{s}^2}{2r'_n}}\rmd\hat{s}\,
\end{equation} for some $C''>0$ and all $n$ large enough, uniformly in $k$. Moreover, since it is straightforward to check that there exists $C>0$ such that
\begin{equation}\int_{-r'_n}^{\infty} \frac{(\hat{s}+r'_n)}{\sqrt{ \hat{s}+2\sqrt{\log 2}\,r'_n}}\rme^{-\frac{\hat{s}^2}{2r'_n}}\rmd\hat{s} \leq C r'_n \,,
\end{equation} for all $n$ sufficiently large, we conclude that $\varphi_1(w)\varphi_2(w)\bbP(L(k) \in \rmd w)$ is at most %we conclude that~\eqref{e:8.24} is at most
\begin{equation}
C' \sqrt{n}2^{-n}\frac{(r'_n)^4}{n-k}(\wedge_n(k))^{-\tfrac 32}
\end{equation} for some $C' >0$ and all $w$ in the domain of the integral in~\eqref{e:8.24}, so that~\eqref{e:8.24} is then bounded from above by
\begin{equation}
C \sqrt{n} 2^{-n} (r'_n)^5 (\wedge_n(k))^{-\tfrac 32}\,\,
\end{equation} for some $C > 0$ and all $n$ sufficiently large. Recalling~\eqref{e:8.27b} and~\eqref{e:8.24}, we see that 
\begin{equation}\label{e:8.46}
\frac{1}{\delta \sqrt{n}}\bbE\left(\left| \cF_n \setminus \mathcal{O}^{[r_n,n-r_n], \eta'}\right| \,;\, A_n\right) \leq \frac{C}{\delta}(r'_n)^5 \sum_{k=r_n}^{n-r_n} (\wedge_n(k))^{-\tfrac 32} \leq \frac{3C}{\delta}(r'_n)^5 (r_n)^{-\tfrac{1}{2}} 
\end{equation} for all $n$ large enough. Therefore, by choosing $\eta'$ small enough (depending on $\eta$ in the definition of $r_n$), the rightmost term in~\eqref{e:8.46} tends to zero as $n \to \infty$, from which the result follows.
\end{proof}

\section*{Acknowledgments}
The work of A.C. was supported by the Swiss National Science Foundation 200021 163170.
The work of O.L. was supported by the Israeli Science Foundation grant no. 1382/17 and by the German-Israeli Foundation for Scientific Research and Development grant no. I-2494-304.6/2017.
The work of S.S. was supported in part at the Technion by a fellowship from the Lady Davis Foundation and also by the Israeli Science Foundation grants no. 1723/14 and 765/18.
\bibliographystyle{abbrv}
\bibliography{TreeCoverTime}

\end{document}